\documentclass{article}

\usepackage[a4paper,width=150mm,top=25mm,bottom=25mm]{geometry}
\newcommand{\undertitle}[1]{}
\usepackage{amsthm}
\usepackage{titling}
\newtheorem{theorem}{Theorem}[section]
\newtheorem{lemma}[theorem]{Lemma}

\newtheorem{proposition}[theorem]{Proposition}

\newtheorem{remark}[theorem]{Remark}
\usepackage{mathtools}
\usepackage{xspace} 
\usepackage{mathrsfs}
\usepackage{esint}
\usepackage{algorithm2e}
\usepackage{algorithmic}
\usepackage{graphicx}
\usepackage{caption}
\usepackage{subcaption}
\usepackage[justification=centering]{caption} 
\usepackage{bm}

\usepackage{amsmath} 
\usepackage[utf8]{inputenc} 
\usepackage[T1]{fontenc}    
\usepackage{url}            
\usepackage{booktabs}       
\usepackage{amsfonts}       
\usepackage{nicefrac}       
\usepackage[numbers]{natbib}
\usepackage{doi}
\usepackage{mathrsfs}
\def\Qvec{\mathbf{Q}}
\def \Mvec{\mathbf{M}}
\def\nvec{\mathbf{n}}
\def\Hvec{\mathbf{H}}

\def\d{\mathrm{d}}
\graphicspath{{Pictures/}}

\usepackage{hyperref}       
\usepackage{cleveref}

\title{A study of ferronematic thin films including a stray field energy}

\author{
Shilpa Dutta\thanks{Institute of Mathematics, 
University of W\"urzburg, Germany, 
\href{mailto:shilpa.dutta@uni-wuerzburg.de}{\texttt{shilpa.dutta@uni-wuerzburg.de}}} ,\quad
James Dalby\thanks{Department of Mathematics,
King's College London, UK, 
\href{mailto:james.dalby@kcl.ac.uk}{\texttt{james.dalby@kcl.ac.uk}}} ,\quad
Apala Majumdar\thanks{University of Strathclyde,
Glasgow, Scotland, 
\href{mailto:apala.majumdar@strath.ac.uk}{\texttt{apala.majumdar@strath.ac.uk}}} ,\quad
Anja Schlömerkemper\thanks{Institute of Mathematics,
University of W\"urzburg, Germany, 
\href{mailto:anja.schloemerkemper@uni-wuerzburg.de}{\texttt{anja.schloemerkemper@uni-wuerzburg.de}}}
}

\begin{document}
\maketitle
\begin{abstract}
\noindent Ferronematic materials are colloidal suspensions of magnetic particles in liquid crystals. They are complex materials with potential applications in display technologies, sensors, microfluidics devices, etc. We consider a model for ferronematics in a 2D domain with a variational approach. The proposed free energy of the ferronematic system depends on the Landau-de Gennes (LdG) order parameter $\Qvec$ and the magnetization $\Mvec$, and incorporates the complex interaction between the liquid crystal molecules and the magnetic particles in the presence of an external magnetic field $\Hvec_{ext}$. The energy functional combines the Landau-de Gennes nematic energy density and energy densities from the theory of micromagnetics including (an approximation of) the stray field energy and energetic contributions from an external magnetic field. For the proposed ferronematic energy, we first prove the existence of an energy minimizer and then the uniqueness of the minimizer in certain parameter regimes. Secondly, we numerically compute stable ferronematic equilibria by solving the gradient flow equations associated with the proposed ferronematic energy. The numerical results show that the stray field influences the localization of the interior nematic defects and magnetic vortices.
\end{abstract}
\textbf{\textit{MSCcodes}}
68Q25, 68R10, 68U05

\section{Introduction}\label{sec:1}
 We investigate a ferronematic model in a two-dimensional domain that combines the Landau-de Gennes (LdG) theory of nematic liquid crystals \cite{Gennes} and the theory of micromagnetics \cite{brown1966magnetoelastic} allowing for the study of both the nematic ordering and the magnetization (from the suspended magnetized particles), including their interaction in the presence of an external magnetic field in a unified framework.

Nematic liquid crystals (NLCs) are prominent examples of mesophases or liquid crystalline phases that combine the fluidity of liquids with the directional order of solids \cite{Gennes}. The nematic liquid crystal phase (NLCs) is characterized by the averaged local alignment of the nematic molecules along common axes, known as directors. Consequently, NLCs are directional materials and are sensitive to external factors such as electric fields, magnetic fields, light, and temperature \cite{lagerwall2012new}. Their anisotropic or direction-dependent response to light and electric fields makes them ideal materials for various electro-optic applications. Since the response of NLCs to an external magnetic field is much weaker than that to an external electric field, the use of NLCs in magneto-optics applications is found to be limited \cite{Shui}. Therefore, a natural question arises of how to or whether one can strengthen the nematic response to an external magnetic field.

This concept was first introduced by Brochard and de Gennes more than five decades ago in their pioneering work \cite{brochard1970theory}, where they proposed the existence of a ferromagnetic phase in a colloidal suspension of magnetic nanoparticles (MNPs) in nematic liquid crystals (NLCs) at room temperature, i.e., a spontaneous magnetization without an external magnetic field. These so-called ferronematics exhibit a stronger response to external magnetic fields compared to pure NLCs. The directors of NLCs can be controlled by the coupling between MNPs and NLCs (NLC-MNP interaction). Similarly, the spontaneous magnetization induced by the suspended MNPs can be tailored by nematic directors via NLC-MNP interactions. In \cite{chen1983observation, Mertelj}, the authors experimentally confirm the existence of a ferromagnetic phase via a suspension of MNPs in a NLC host. For recent theoretical work on ferronematics, we refer to \cite{Bisht2020, Canevari2025, Canevari, Dalby, maity2021parameter}, see below for details.

The present model of a two-dimensional ferronematic material is inspired by the Landau de Gennes (LdG) energy for nematic materials in two dimensions \cite{Gennes, han2020reduced} and the micromagnetic energy for ferromagnetic materials reduced to two dimensions.   Our model extends the ferronematic model outlined in \cite{Bisht2020}. The ferronematic phase is described by two state variables on a bounded domain $\Omega \subset \mathbb{R}^2$, one of which is the reduced LdG order parameter $\Qvec: \Omega\rightarrow\mathbb{R}^{3\times 3}$ with $\Qvec^T=\Qvec, \mathrm{tr}{\Qvec}=0, Q_{3i}=Q_{i3}=0$, $i=1,2,3$ (only valid for two-dimensional scenarios and for certain choices of the boundary conditions). The eigenvector corresponding to the largest positive eigenvalue of $\Qvec$ models the nematic director $\nvec:\Omega\to\mathbb{R}^3$, c.f.\ \eqref{eq:en4}. The second state variable is the spontaneous magnetization $\Mvec: \Omega\rightarrow\mathbb{R}^3$. 

The micromagnetic part of the energy contains the exchange energy, the stray field energy, and the Zeeman energy. We refer to \cite{brown1966magnetoelastic, hubert2008magnetic,kruzikprohl} for an introduction to micromagnetics.  The exchange energy $\int_{\Omega}|\nabla\Mvec|^2\mathrm{d}x$ is invoked as an energy penalty to penalize jumps and rapid gradients in $\Mvec$.   
 The stray field $\Hvec_{\Mvec}$ is generated by the magnetization $\Mvec$. In three dimensions, the stray field is given by a solution to the stationary Maxwell equation, cf.\ e.g.\ \cite{kruzikprohl}. The stray field energy  is then given by $\frac{1}{2}\int_{\mathbb{R}^3} |\Hvec_\Mvec|^2 = -\frac{1}{2}\int_{\mathbb{R}^3} \Hvec_{\Mvec} \cdot \Mvec$ with $\Mvec$ being extended by zero to the whole space. The stray field energy may contribute to the formation of boundary vortices, edge-curling domain walls, interior walls, etc.~\cite{desimone2004recent}. 
 In this article, we approximate the stray field energy by a reduced model as derived in Gioia and James \cite{Giogia} for very thin materials, see also \cite[Sec.~10]{desimone2004recent,Garcia2004}. Then the stray field energy on the two-dimensional domain $\Omega$ reads $\frac{1}{2}\int_{\Omega}\left(M_3\right)^2\d x$, where $M_3$ denotes the third component of $\Mvec:\Omega\to \mathbb{R}^3$; see Section~\ref{sec:2} for more details.
Furthermore, we assume that the ferronematic material is exposed to a given external magnetic field $\Hvec_{ext}$; this is a three-dimensional vector field that in experiments is given on $\mathbb R^3$. However, in our analytical and numerical investigation, only its regularity on $\Omega$ is of relevance. The interaction of $\Hvec_{ext}$ with $\Mvec$ is modelled by the Zeeman energy $-\int_{\Omega}\Hvec_{ext}\cdot\Mvec\, \d x$. 

Finally, the energy not only contains Ginzburg-Landau type potentials for $\Qvec$ and  $\Mvec$ but also a coupling of $\Mvec$ and $\Qvec$, via $-\int_\Omega \Qvec \Mvec \cdot \Mvec$, cf.\ also \cite{Pleiner}. Following, \cite{Calderer, Mertelj}, we also consider the coupling between an external magnetic field $\Hvec_{ext}$ and $\Qvec$ in the form $-\int_\Omega \Qvec \Hvec_{ext} \cdot \Hvec_{ext}$. 
As we show in Section~\ref{sec:nondim}, the dimensionless form of the ferronematic energy  in two dimensions is given by 
\begin{eqnarray}\label{en:2}
\mathcal{E}\left(\Qvec, \Mvec\right)&=&\int_{\Omega}\frac{l_1}{2}|\nabla \Qvec|^2 + \frac{\xi l_2}{2}|\nabla \Mvec|^2 - \frac{c_1}{2}\Qvec\Mvec\cdot\Mvec 
+ \frac{c_3}{2}\xi\left(M_3\right)^2
\nonumber\\
&& \;\;- \frac{c_2}{2}\Qvec\Hvec_{ext}\cdot\Hvec_{ext} 
- c_3\xi\Mvec\cdot\Hvec_{ext} + \frac{1}{4}\left(\frac{1}{2}|\Qvec|^2 - 1\right)^2 + \frac{\xi}{4}\left(|\Mvec|^2 - 1\right)^2 \,\d x,
\end{eqnarray}
where
$l_1$, $l_2$, $c_1$, $c_2$, $c_3$ and $\xi$ are  positive parameters that are specified in the nondimensionalization in Section~\ref{sec:nondim}.  

The present work,~based on the proposed ferronematic energy \eqref{en:2}, has two main objectives: first, to investigate the existence of a minimizer of the energy \eqref{en:2} and to show the uniqueness of the energy minimizer in certain parameter regimes;
 and second, to numerically examine the influence of the stray field, as well as the combined role of a stray field and an external magnetic field on stable ferronematic equilibria in Sections~\ref{sec:5} and \ref{sec:numerical_experiments}. The existence of a solution for the energy minimization problem (Theorem~\ref{thm:1}) follows from the direct method in the calculus of variations and the uniqueness result (Theorem~\ref{thm:2}) is based on an application of the maximum principle.  To prove the uniqueness theorem, we derive $L^{\infty}$-bounds for $\Qvec$ and $\Mvec$ and estimates of the second partial derivatives of the energy density; the result then follows by contradiction. 

The latter part of the manuscript focuses on numerical computations of physically observable ferronematic profiles (modelled as local minimizers of the ferronematic energy). Using the energy dissipation law, we derive the gradient flow system that corresponds to the Euler-Lagrange equations for the ferronematic energy \eqref{en:2} in the stationary setting. We then employ the Crank-Nicolson finite difference technique (see e.g.\ \cite{Burden-Faires}) to solve the gradient-flow system supplemented with initial and boundary conditions. The choice of the prescribed initial and boundary conditions is motivated by \cite{Canevari}. Since the gradient flow system is a coupled non-linear system of PDEs, we use Newton's linearization technique to formulate the Crank-Nicolson discretized scheme as this allows us to tackle the convergence issues that arise due to nonlinearities.

Before we outline the results of our article further, we comment on previous work in the literature. The proposed energy \eqref{en:2} reduces to the ferronematic energy in \cite{Bisht2020} when $M_3=0$ and $\Hvec_{ext}=\mathbf{0}$. The ferronematic energy in \cite{Bisht2020} is based on the Landau description for ferronematic phase transition \cite{Pleiner}, that includes the magnetic energy due to the penalization of magnetic saturation and does not account the contribution of the stray field. In \cite{Bisht2020}, the authors first numerically compute stable ferronematic equilibria on a 2D square domain with tangent boundary conditions and report two types of stable equilibria featuring magnetic domain walls with pairs of interior and boundary nematic defects. Also, they show that the NLC-MNP coupling tailors the location and multiplicity of nematic defects and magnetic vortices. Later, Han et al.~\cite{han2021tailored} extend the previous work to 2D pentagons and hexagons and observe the co-existence of similar stable states. In this article, we investigate the effect of an approximated stray field energy and a Zeeman energy on the microstructure of the materials and consider \cite{Bisht2020} as a reference case for our numerical simulations.   

Next, we summarize relevant work with different perspectives. 
In \cite{maity2021parameter}, the authors work on the asymptotic analysis of the global minimizers of a rescaled ferronematic energy in the limit of vanishing elastic constant, i.e., in the $l_1=l_2=l\rightarrow 0$ limit. Their results focus on the strong $H^1$-convergence and an $l$-independent $H^2$-bound for global energy minimizers on a smooth, bounded 2D domain with Dirichlet boundary conditions and the results are complemented by numerical experiments. Canevari et al.~\cite{Canevari} have shown that in the super-dilute regime, the LdG order parameter is a canonical harmonic map and the limit of the magnetization is a singular vector field with line defects that connect point defects along a minimal connection. The super-dilute regime is a limiting case for which the magnetic energy has a significantly lesser contribution than the NLC energy and the NLC-MNP interaction remains weak. Their theoretical results are also supported by numerical experiments. Recently, that work has been extended to the mixed boundary conditions (Dirichlet type for $\Qvec$ and Neumann type for $\Mvec$) in \cite{Canevari2025}. 
In \cite{Dalby}, the authors study order reconstruction phenomena and perform bifurcation analysis of a ferronematic energy, in one-dimensional channel geometries.

The aforementioned studies do not consider the presence of stray fields and external magnetic fields. 
 However, the stray field is relevant because it can tailor arrangements of magnetic vortices for purely magnetic materials \cite{hubert2008magnetic} and these phenomena can be further influenced by an external magnetic field. Some interesting features, e.g., switching mechanisms and birefringence \cite{chen1983observation, Mertelj} of ferronematic substances when exposed to an external magnetic field have been observed experimentally. From the theoretical perspective, a variational homogenized ferronematic model that accounts for the presence of an external magnetic field in the Oseen-Frank framework can be found in \cite{Calderer}. 

The Gioia-James approximation \cite{Giogia} for the stray field energy does not lead to any criterion on the thickness of the material sample and is applicable for many practical purposes, e.g.\ films used in storage devices with a thickness of about 250 \AA, and recording devices with a thickness less than 100 \AA.  Moreover, we refer to \cite{kalousek2021mathematical}, where the authors deal with the coupling of the three-dimensional magnetization with the two-dimensional fluid flow in an evolutionary magnetoviscoelastic model, counting the contribution of the stray field using the Gioia and James approximation. As an aside, we mention that there are other stray field limiting cases available in the literature, e.g.\ \cite{desimone2002reduced, di2023reduced, kohn2005another} for the pure ferromagnetic thin-film regime and these different limiting cases are constructed based on their varied behavior within the film's interior and at the domain edge. It will be a challenging and interesting task to deal with other limiting cases or the full stray field energy in the study of ferronematic thin films in future works.

This article is structured as follows: we describe the notation used in this paper in \Cref{not}. We introduce the generalized ferronematic energy in \Cref{sec:2}. In \Cref{sec:4}, we prove the existence of a minimizer for our generalized ferronematic energy and prove the uniqueness of the minimizer in certain parameter regimes (Theorems~\ref{thm:1} and \ref{thm:2}). In \Cref{sec:5}, we explore the behavior of nematic and magnetic configurations due to the combined influence of a stray field and an external magnetic field, by numerically solving the gradient flow system \eqref{eq:26} for the introduced ferronematic energy \eqref{en:2} using the Crank-Nicolson finite difference scheme. The numerical results show that the stray field has an impact on the multiplicity and locations of interior nematic defects and magnetic vortices.  We close our study with some concluding remarks in \Cref{sec:6}.

\section{Notation}\label{not} This section defines the notation used throughout this paper: $\textbf{e}_i$ is the $i$-th standard basis vector in $\mathbb R^3$ for $1\le i\le 3$; $\delta_{ij}$ denotes the standard Kronecker delta function, i.e., $\delta_{ij}=1$ if $i=j$, and $\delta_{ij}=0$ if $i\ne j$; $a\cdot b$ is the scalar product for two vectors $a, b\in\mathbb{R}^3$, i.e., $a\cdot b=\sum_{i=1}^{3}a_ib_i$; $a\times b$ defines the standard cross product for $a, b\in\mathbb{R}^3$; $a\otimes b$ stands for the dyadic product for $a, b\in\mathbb{R}^3$, i.e., $\left(a\otimes b\right)=a_ib_j$ for any $i,j\in\{1,2,3\}$. Moreover, $\mathbb{R}^{3\times 3}$ denotes the set of $3\times 3$-matrices with real entries; $SO\left(3\right)$ is a subset of matrices $A \in \mathbb R^{3\times 3}$ such that $AA^T=A^TA=\mathbb{I}$ and $\det(A)=1$, where $\mathbb{I}$ denotes the identity matrix. For symmetric and traceless $z$-invariant LdG order parameters $\Qvec\in\mathbb{R}^{3\times 3}$ with $Q_{3j}=Q_{i3}=0$ for $i,j =1,2,3$, and $\Mvec\in \mathbb{R}^3$ we set
\begin{eqnarray}
|\Qvec|^2=\sum_{i,j=1}^{2}|Q_{ij}|^2;~|\nabla\Qvec|^2=\sum_{i,j=1}^2 \sum_{k=1}^2|\nabla_k Q_{ij}|^2,\nonumber\\
|\Mvec|^2=\sum_{i=1}^{3}|M_i|^2;~|\nabla\Mvec|^2=\sum_{i=1}^{3}\sum_{k=1}^{2}|\nabla_k M_i|^2~.\nonumber
\end{eqnarray}
We assume $\Omega\subset\mathbb {R}^2$ to be bounded, connected, open and with Lipschitz boundary. 
For $1\leq p< \infty$ and $k\in \mathbb{N}_0$, the function spaces are used as per the standard definition:
\begin{eqnarray}
C\left(\overline{\Omega}\right)&=& \bigg\{f:\overline{\Omega}\rightarrow\mathbb R: f \mbox{ is continuous, } \sup_{\overline{\Omega}}|f|<\infty  \bigg\},\nonumber\\
C^k\left(\Omega\right)&=& \bigg\{f:\Omega\rightarrow\mathbb R: f \mbox{ is } k\mbox{-times continuously differentiable and has finite norm}  \bigg\} \nonumber \\ && \mbox{ for } k=0,1,2,..., \nonumber\\
C^{k}_{c}\left(\Omega\right)&=& \bigg\{f:\Omega\rightarrow\mathbb R: f\in C^k\left(\Omega\right) \mbox{ with compact support}\bigg\} \mbox{ for } k=0,1,2,..., \nonumber\\
C\left(\overline{\Omega}; \mathbb R^3\right)&=& \bigg\{f:\bar{\Omega}\rightarrow\mathbb R^3: f=\left(f_1, f_2, f_3\right) \mbox{ with } f_i\in C\left(\overline{\Omega}\right), \mbox{ for } i=1,2,3 \bigg\},\nonumber\\
L^p\left(\Omega; \mathbb R^3\right) &=& \bigg\{f:\Omega\rightarrow\mathbb R^3: f \mbox{ is Lebesgue measurable, } \|f\|_{L^p(\Omega ;\mathbb{R}^3)}=\left(\int_{\Omega}|f|^p\d x\right)^{\frac{1}{p}}<\infty \bigg\},\nonumber\\
L^{\infty}\left(\Omega; \mathbb R^3\right) &=& \bigg\{f:\Omega\rightarrow\mathbb R^3: f \mbox{ is Lebesgue measurable, } \|f\|_{L^\infty(\Omega; \mathbb{R}^3)} = \operatorname*{ess\,sup}_{\Omega}~|f|<\infty \bigg\},\nonumber\\
W^{k,p}\left(\Omega; \mathbb R^3\right)&=&\bigg\{f\in L^p\left(\Omega; \mathbb R^3\right): D^{\alpha}f\in L^p\left(\Omega; \mathbb R^3\right)~\forall |\alpha|\le k\bigg\}\nonumber,
\end{eqnarray}
where the derivative $D^{\alpha}$ is interpreted in the weak sense. \sloppy In the case $p=2$, we set $H^k\left(\Omega; \mathbb R^3\right)=W^{k,2}\left(\Omega; \mathbb R^3\right).$ If the range of the functions is $\mathbb{R}^{3\times 3}$, the spaces are defined correspondingly.

\section{Ferronematic energy}
\label{sec:2}
As mentioned in the introduction, we study ferronematics confined to two-dimensional domains $\Omega\subset\mathbb R^2$. The ferronematic phase is described by two order parameters: the reduced Landau-de Gennes (LdG) $\Qvec$-tensor order parameter for the nematic phase that describes the orientational ordering of the ferronematic phase (see \cite{han2020reduced}), and a magnetization vector $\Mvec: \Omega\rightarrow\mathbb R^3$ to describe the magnetic properties of the ferronematic phase. In three dimensions, the full LdG $\Qvec$-tensor order parameter is a symmetric, traceless $3\times 3$ matrix with five degrees of freedom \cite{Gennes}.  
In two-dimensional scenarios such as ours, one can use the reduced LdG (rLdG) description for which the physically relevant $\Qvec$-tensors have a fixed eigenvector with an associated fixed eigenvalue and all quantities of interest are independent of the $z$-coordinate \cite{han2021tailored}; this automatically implies that the rLdG $\Qvec$-tensor has two degrees of freedom as opposed to five for the fully three-dimensional scenario. This dimension reduction from three dimensions to two dimensions can be rigorously justified by using $\Gamma$-convergence techniques \cite{golovaty2015dimension}. We refer to $\Qvec$ as the LdG order parameter, instead of the rLdG order parameter in the remainder of the text, for brevity. We define the set of matrices: 
\begin{eqnarray}
\mathcal{Q}=\bigl\{A\in\mathbb R^{3\times 3}: A^T=A, \operatorname{tr}(A)=0, A_{3i}=A_{i3}=0, \, i=1,2, 3 \bigl\},\nonumber
\end{eqnarray} and $\Qvec \in \mathcal{Q}$ in subsequent sections. One can also write $\Qvec$ as follows:
\begin{eqnarray}\label{eq:en4}
\Qvec=s\left(2\nvec\otimes\nvec - \mathbb{J}\right),
\end{eqnarray}
where $\nvec = (n_1, n_2, 0)$ is the nematic director or the eigenvector of $\Qvec$ with the largest positive eigenvalue, $s$ is the scalar order parameter that measures the degree of the orientational order about $\nvec$ and
\begin{displaymath}
\mathbb{J}=\begin{pmatrix}
1 & 0 & 0\\
0 & 1 & 0\\
0 & 0 & 0
\end{pmatrix}.
\end{displaymath}
In particular, one can label planar defects in $\mathbb{R}^2$ as being the nodal set of the scalar order parameter $s$ \cite{han2021tailored}.

In what follows, we measure the magnetization $\Mvec=\left(M_1, M_2, M_3\right): \Omega\rightarrow\mathbb R^3$ in units of a saturation magnetisation, $M_s$, so that $\Mvec$ is dimensionless by definition, as is the LdG order parameter.

The magnetization $\Mvec$  generates a non-local magnetic field which is denoted by $\Hvec_{\Mvec}: \mathbb R^3\rightarrow\mathbb R^3$ and is called the stray field \cite{brown1966magnetoelastic}. Following the theory of micromagnetics, we assume the material to be a non-conducting ferromagnetic material. 
Here, instead  
 we work with an approximated stray field energy derived by Gioia and James \cite{Giogia} to overcome the technical difficulties by the non-locality. The approximation is a 3D to 2D reduction of the stray field energy obtained by using asymptotic analysis based on weak convergence methods. In this approximation, the stray field contributes locally to the film's interior without any additional penalization on the film's edges, the stray field energy is independent of the space coordinate normal to the film and there is no remnant of the stationary Maxwell equation. In the approximation by Gioia and James \cite{Giogia}  for very thin films that we use in our two-dimensional framework, the stray field energy $-\frac{1}{2}c_3\xi\int_\Omega \Hvec_\Mvec \cdot \Mvec \,  \d x$ is replaced by $\frac{1}{2}c_3\xi\int_\Omega (M_3)^2 \d x$. 

Further, we assume that there is a constant dimensionless (re-scaled) external magnetic field denoted by $\Hvec_{ext} = (H_1, H_2, H_3): \mathbb R^3\rightarrow\mathbb R^3$, that interacts with the nematic order parameter and the magnetization. Following \cite{Calderer, Mertelj}, we model the interaction between an external magnetic field and the nematic order parameter by the energy density
\begin{eqnarray}\label{eq:en6'}
- \chi_1\Qvec\Hvec_{ext}\cdot\Hvec_{ext},
\end{eqnarray}
where $\chi_1$ is a characteristic positive material-dependent constant related to magnetic susceptibility/anisotropy of the ferronematic material. 

The total ferronematic energy comprises several contributions: the LdG energy for the nematic/orientational ordering in two dimensions \cite{Gennes, han2020reduced}, a Ginzburg-Landau type energy for $\Mvec$ including the exchange energy for the magnetization \cite{brown1966magnetoelastic, hubert2008magnetic}, a nemato-magnetic coupling energy that describes the coupling between $\Qvec$ (or the nematic director $\nvec$) and $\Mvec$ via the surface-induced coupling on the magnetic nanoparticles \cite{Bisht2020}, an approximated stray field energy \cite{Giogia}, the Zeeman energy that accounts for the interaction between $\Hvec_{ext}$ and $\Mvec$ \cite{brown1966magnetoelastic} and the interaction energy between $\Hvec_{ext}$ and $\Qvec$, as in \eqref{eq:en6'}. 
Therefore, the total ferronematic energy is given by
\begin{equation} \label{mathcalF}
\begin{aligned}
\mathscr{F}\left(\Qvec, \Mvec\right) =& \int_\Omega\frac{K}{2} | \nabla \Qvec|^2 + \frac{A}{2}|\Qvec|^2 + \frac{C}{4}|\Qvec|^4 + \frac{\kappa}{2} |\nabla \Mvec|^2
+ \frac{\alpha}{2}|\Mvec|^2 + \frac{\beta}{4}|\Mvec|^4 - \frac{\gamma_1\mu}{2}\Qvec\Mvec\cdot\Mvec 
\\\\& 
\;\;\;+ 
\frac{\mu}{2}\left(M_3\right)^2 - \mu\Mvec\cdot \Hvec_{ext}- \frac{\chi_1\mu}{2}\Qvec\Hvec_{ext}\cdot\Hvec_{ext}\;\d x.
\end{aligned}
\end{equation}
\noindent In the above, $K$ is referred to as the positive nematic elastic constant \cite{Gennes} with unit N (newton) and $\kappa$ is the positive rescaled exchange stiffness constant \cite{hubert2008magnetic}, also with unit N. 
The parameters $A$ and $C$ are LdG bulk coefficients typically related to the temperature and material properties, with unit N$\cdot$m$^{-2}$ \cite{newtonmottram}. More precisely, $A$ is often interpreted to be a rescaled temperature and the bulk ferronematic system favours an ordered nematic state with $\Qvec \neq 0$ for $A<0$, and a disordered isotropic phase with $\Qvec = 0$ for $A>0$. The coefficients $\alpha$ and $\beta$ are Landau coefficients that describe the ferromagnetic transition \cite{Mertelj}. In what follows, we assume that $A<0$ and $\alpha<0$, so that the LdG potential for $\Qvec$ and the Landau potential for $\Mvec$ favour an ordered ferronematic phase (with $\Qvec, \Mvec \neq 0$), in the absence of spatial inhomogeneities. Further, $\gamma_1$ is a constant that measures the strength of the nemato-magnetic coupling \cite{Mertelj} and we assume that $\gamma_1 >0$ in what follows, so that it coerces co-alignment between $\nvec$ and $\Mvec$ (see \cite{Bisht2020}). Finally, $\mu$ is the re-scaled magnetic permeability of the vacuum \cite{hubert2008magnetic}. The coefficients $\alpha, \beta, \mu$ have units N$\cdot$m$^{-2}$, 
while $\gamma_1$ is unitless.

We notice that the energy \eqref{mathcalF} reduces to the ferronematic energy of \cite{Bisht2020} when $M_3=0$ and $\Hvec_{ext}=\mathbf{0}$ i.e., for planar magnetization vectors and hence without the considered stray field energy approximation, and in the absence of external magnetic fields.

\subsection{Nondimensionalization}
\label{sec:nondim}
Next, we re-scale the ferronematic energy \eqref{mathcalF} to identify the physically relevant variables. Let $L$ be a characteristic length scale of the domain, $\Omega\subset\mathbb R^2$.  We use the upcoming scalings following \cite{Bisht2020}: $$\mathcal{W'}=\mathcal{W}\frac{C}{|A|^2},~\left(x', y'\right)=\frac{\left(x, y\right)}{L},~\Qvec'=\sqrt\frac{2C}{|A|}\Qvec,~\Mvec'=\sqrt\frac{\beta}{|\alpha|}\Mvec,~\Hvec'_{ext}=\sqrt\frac{\beta}{|\alpha|}\Hvec_{ext},$$
where $\mathcal{W}$ denotes the energy density in \eqref{mathcalF}.
The non-dimensional/re-scaled ferronematic energy then reads (up to a constant)  
\begin{equation}
\begin{aligned}
\mathcal{E}'\left(\Qvec', \Mvec'\right) = \int_{\Omega'}\mathcal{W}'\left(\Qvec', \nabla'\Qvec', \Mvec', \nabla'\Mvec'\right)\;\d x',\label{eq:en12}
\end{aligned}
\end{equation}
where
\begin{eqnarray}
\mathcal{W'}\left(\Qvec', \nabla'\Qvec', \Mvec', \nabla'\Mvec'\right)&=&\frac{l_1}{2}|\nabla' \Qvec'|^2 + \frac{\xi l_2}{2}|\nabla' \Mvec'|^2 + \frac{1}{4}\left(\frac{1}{2}|\Qvec'|^2 - 1\right)^2 \\ &&+ \frac{\xi}{4}\left(|\Mvec'|^2 - 1\right)^2 - \frac{c_1}{2}\Qvec'\Mvec'\cdot\Mvec'+ \frac{c_3}{2}\xi\left(M_3'\right)^2 \\
&&- \frac{c_2}{2}\Qvec'\Hvec_{ext}'\cdot\Hvec_{ext}' - c_3\xi\Mvec'\cdot\Hvec_{ext}'.\nonumber
\end{eqnarray}
Here, $\Omega'$ is the re-scaled two-dimensional domain.
The relevant dimensionless parameters are defined to be
$$l_1=\frac{K}{2|A|L^2},~l_2=\frac{\kappa}{|\alpha| L^2},~\xi=\frac{C}{|A|^2}\frac{|\alpha|^2}{\beta},~c_1=\frac{\gamma_1\mu}{|A|}\sqrt{\frac{C}{2|A|}}\frac{|\alpha|}{\beta},$$
$$c_2=\frac{\chi_1\mu}{|A|}\sqrt{\frac{C}{2|A|}}\frac{|\alpha|}{\beta},~c_3=\frac{\mu}{|\alpha|}.$$
We drop the primes in what follows and all subsequent results are to be understood in terms of the re-scaled variables. Hence, we obtain \eqref{en:2}.

We briefly discuss the physical interpretation of the dimensionless parameters. The re-scaled elastic constants, $l_1$ and $l_2$ are the ratios of the characteristic material-dependent length scales (e.g., $\sqrt{\frac{K}{|A|}}$, $\sqrt{\frac{\kappa}{|\alpha|}}$) and geometric length scales, such as $L$ in this setting. For simplicity, $l_1$ and $l_2$ are sometimes assumed to be equal to $l$, which is regarded as the elastic constant. The non-negative parameter $\xi$ is interpreted as a measure of the relative strengths of the magnetic energy (the Ginzburg-Landau potential for $\Mvec$ and the exchange energy) and the LdG energy (depending on $\Qvec$ and $\nabla \Qvec$). The magnetic energy is dominant for large values of $\xi$ and we expect to see the nematic director profiles being tailored by $\Mvec$ for large $\xi$ as in \cite{Bisht2020}. The parameter $c_1>0$ is simply a re-scaled measure of the nematic-magnetic coupling strength constant $\gamma_1$. The parameter $c_2$ is a measure of the relative strength of the interaction energy \eqref{eq:en6'}, i.e., this interaction energy between the NLC and the external magnetic field becomes increasingly dominant as $c_2$ increases. In addition, we have the parameter, $c_3$, which accounts for the stray field energy and the Zeeman energy. Below, we elaborate on the effects of these parameters with some analytic and numerical results, although far more comprehensive studies are needed for a complete understanding.

\begin{remark}
The energy density in \eqref{eq:en12} is frame-indifferent.
\end{remark}

\section{Existence and uniqueness}\label{sec:4}
In this section, we study existence and uniqueness properties of energy minimizers of the functional $\mathcal{E}$ as introduced in \eqref{en:2}. 
For convenience, we also write 
\begin{eqnarray}
\mathcal{E}\left(\Qvec, \Mvec\right)&=&\int_{\Omega}\frac{l_1}{2}|\nabla \Qvec|^2 + \frac{\xi l_2}{2}|\nabla \Mvec|^2  + f\left(\Qvec, \Mvec\right)
+ \frac{c_3}{2}\xi\left(M_3\right)^2
\nonumber\\\nonumber\\
&& \;\;- \frac{c_2}{2}\Qvec\Hvec_{ext}\cdot\Hvec_{ext} 
- c_3\xi\Mvec\cdot\Hvec_{ext} \,\d x,\nonumber
\end{eqnarray}
where
\begin{eqnarray} \label{def-f}
f\left(\Qvec, \Mvec\right)&=&  \frac{1}{4}\left(\frac{1}{2}|\Qvec|^2 - 1\right)^2 + \frac{\xi}{4}\left(|\Mvec|^2 - 1\right)^2 -  \frac{c_1}{2}\Qvec\Mvec\cdot\Mvec.
\end{eqnarray}
We define the following admissible spaces
\begin{displaymath}
\mathcal{A} = \Big \{ \Qvec \in W^{1,2} \left(\Omega, \mathcal{Q}\right):  \Qvec|_{\partial \Omega} = \Qvec_{bd} \mbox{ in the sense of traces}\Big \},
\end{displaymath}
\begin{displaymath}
\mathcal{S} = \left \{ \Mvec \in W^{1,2} \left( \Omega, \mathbb R^3 \right) : \Mvec|_{\partial \Omega} = \Mvec_{bd} \mbox{ in the sense of traces}\right \},
\end{displaymath}
\begin{displaymath}
\mathcal{A}\times\mathcal{S} = \left \{ \left(\Qvec, \Mvec\right) : \Qvec \in \mathcal{A} \mbox{ and } \Mvec \in \mathcal{S} \right \}
\end{displaymath}
for some given Lipschitz mappings $\Qvec_{bd}: \partial\Omega\rightarrow\mathcal{Q}$ and $\Mvec_{bd}: \partial\Omega\rightarrow\mathbb R^3$, which define our Dirichlet boundary conditions.
We prove the existence of global minimizers of the ferronematic energy functional $\mathcal{E}$ in \eqref{en:2}, in the admissible space $\mathcal{A} \times \mathcal{S}$, using the direct methods in the calculus of variations. Firstly, we check that $\mathcal{E}$ is bounded below by a fixed constant independent of $\Qvec$ and $\Mvec$. To do so, it is enough to show that the coupled terms are bounded below since the other terms are non-negative. 
We set $\mathcal{I}_1= -c_1\Qvec\Mvec\cdot\Mvec$ and observe that
\begin{eqnarray}
\mathcal{I}_1&=&-\frac{c_1}{2}\sum_{i,j=1}^{3}Q_{ij}M_iM_j \nonumber\\
&=&-\frac{c_1}{2}Q_{11}\left(M_1^2 - M_2^2\right) - c_1Q_{12}M_1M_2~~(\mbox{since } Q_{i3}=Q_{3j}=0,Q_{12}=Q_{21}, Q_{22}=-Q_{11})\nonumber\\
& \ge & -\frac{c_1}{4}\left(\epsilon_1|\Qvec|^2 + \frac{1}{\epsilon_1}|\Mvec|^4\right)\label{eq:ex1}
\end{eqnarray}
for some arbitrary constant $\epsilon_1>0$. Similarly, 
\begin{eqnarray}
\mathcal{I}_2 &=&  -\frac{c_2}{2}\Qvec\Hvec_{ext}\cdot\Hvec_{ext} 
\ge -\frac{c_2}{4}\left(\epsilon_2|\Qvec|^2 + \frac{1}{\epsilon_2}|\Hvec_{ext}|^4\right)\label{eq:ex2}
\end{eqnarray}
for some arbitrary $\epsilon_2>0$. Finally, note that \begin{eqnarray}
\mathcal{I}_3 &=& -c_3\xi\Mvec\cdot\Hvec_{ext} 
\ge-c_3\xi\left(\epsilon_3|\Mvec|^2 + \frac{1}{\epsilon_3}|\Hvec_{ext}|^2\right)\label{eq:ex3}
\end{eqnarray}
for some arbitrary $\epsilon_3>0$. 

Combining \eqref{eq:ex1}, \eqref{eq:ex2} and \eqref{eq:ex3}, we obtain 
\begin{eqnarray}
& &f\left(\Qvec, \Mvec\right)-\frac{c_2}{2}\Qvec\Hvec_{ext}\cdot\Hvec_{ext}-c_3\xi\Mvec\cdot\Hvec_{ext}\nonumber\\
&\ge& \frac{1}{4}\left(\frac{1}{2}|\Qvec|^2 - 1\right)^2 + \frac{\xi}{4}\left(|\Mvec|^2 - 1\right)^2 -\frac{c_1}{4}\left(\epsilon_1|\Qvec|^2 + \frac{1}{\epsilon_1}|\Mvec|^4\right) -\frac{c_2}{4}\left(\epsilon_2|\Qvec|^2 + \frac{1}{\epsilon_2}|\Hvec_{ext}|^4\right)\nonumber\\
& &~~~~~-c_3\xi \epsilon_3|\Mvec|^2 -\frac{c_3\xi}{\epsilon_3}|\Hvec_{ext}|^2\nonumber\\
&\ge&\frac{1}{4}\left(\frac{1}{2}|\Qvec|^2-\left(1+c_1\epsilon_1+c_2\epsilon_2\right)\right)^2+\left(\frac{\xi}{4}-\frac{c_1}{4\epsilon_1}\right)\left(|\Mvec|^2 -\frac{\frac{\xi}{4}+\frac{c_3\xi\epsilon_3}{2}}{\frac{\xi}{4}-\frac{c_1}{4\epsilon_1}}\right)^2\nonumber\\
& &~~~~~~~~~-\frac{c_3\xi}{\epsilon_3}|\Hvec_{ext}|^2-\frac{c_2}{2\epsilon_2}|\Hvec_{ext}|^4+g\left(c_1,c_2,\xi,\epsilon_1,\epsilon_2,\epsilon_3\right),\label{eq:ex5}
\end{eqnarray}
where we choose $\epsilon_1>\frac{c_1}{\xi}$ and $g$ is an explicitly computable constant that is independent of $\Qvec$ and $\Mvec$. The lower bound of the energy functional \eqref{en:2} follows immediately due to the estimate \eqref{eq:ex5} together with $\Hvec_{ext}\in C\left(\overline{\Omega}; \mathbb R^3\right)$. 
The existence result in Theorem~\ref{thm:1} below then relies on the following compactness and lower-semicontinuity arguments.
\begin{proposition}\label{prop:1}
Assume that  $\mathcal{A}\times\mathcal{S}\ne\emptyset$ and $\Hvec_{ext}\in C\left(\overline{\Omega}; \mathbb R^3\right)$. Then the following compactness result holds:\\
If $\mathcal{E}\left(\Qvec_k, \Mvec_k\right)\le\Lambda$ for a sequence $\left(\Qvec_k, \Mvec_k\right)\in\mathcal{A}\times\mathcal{S}$ and some $\Lambda\in\mathbb R$ independent of $k$, then $\left(\Qvec_k, \Mvec_k\right)$ has a weakly converging subsequence in $\mathcal{A}\times\mathcal{S}$.
\end{proposition}
\begin{proof}
Let $\left(\Qvec_k, \Mvec_k\right)_k \subseteq \mathcal{A} \times \mathcal{S}$ be an infimizing sequence for the functional $\mathcal{E}\left(\Qvec, \Mvec\right)$ such that $\mathcal{E}\left(\Qvec_k, \Mvec_k\right)<\Lambda$ for some positive $\Lambda\in\mathbb{R}$ and all $k\in \mathbb{N}$. Then, in particular,
$$ \int_{\Omega} \frac{l_1}{2}|\nabla \Qvec_k|^2 + \frac{l_2\xi}{2}|\nabla \Mvec_k|^2 + \frac{c_3}{2}\xi\left(M_{3k}\right)^2\d x < \Lambda$$
uniformly for any $k\in\mathbb{N}$.
We fix $\left(\Tilde{\Qvec}, \Tilde{\Mvec}\right)\in\mathcal{A}\times\mathcal{S}$. Then $\Qvec_k - \Tilde{\Qvec} \in W^{1,2}_{0}\left(\Omega; \mathbb{R}^{3\times 3} \right)$ and $\Mvec_k - \Tilde{\Mvec}\in W^{1,2}_{0}\left(\Omega; \mathbb R^3\right)$. Thus, Poincar\'e's inequality yields:
\begin{eqnarray}
||\Qvec_k||_{L^2\left(\Omega; \mathbb R^{3\times 3}\right)}&\le&||\Qvec_k - \Tilde{\Qvec}||_{L^2\left(\Omega; \mathbb R^{3\times 3}\right)} + ||\Tilde{\Qvec}||_{L^2\left(\Omega; \mathbb R^{3\times 3}\right)} \le C_p||\nabla \Qvec_k - \nabla \Tilde{\Qvec}||_{L^2\left(\Omega; \mathbb R^{3\times 3}\right)} + \Tilde{C} \nonumber \\ &\le&C_p\left(||\nabla \Qvec_k||_{L^2\left(\Omega; \mathbb R^{3\times 3}\right)} + ||\nabla \Tilde{\Qvec}||_{L^2\left(\Omega; \mathbb R^{3\times 3}\right)}\right) + \Tilde{C}\nonumber\\
&\le& C\nonumber
\end{eqnarray}
for any $k\in\mathbb{N}$. Similarly, we have that $||\Mvec_k||_{L^2\left(\Omega; \mathbb R^3\right)}\le C$ for any $k\in\mathbb{N}$. Here $C>0$ is a generic constant and $C_p>0$ is the Poincar\'e constant.
Hence, the sequence $\left(\Qvec_k, \Mvec_k\right)_k$ is bounded in the reflexive Banach space $W^{1,2}\left(\Omega; \mathbb R^{3\times 3}\right)\times W^{1,2}\left(\Omega; \mathbb R^3\right)$ for any $k\in\mathbb{N}$.
By the Banach-Alaoglu theorem, we can extract a subsequence without relabeling such that
\begin{eqnarray}\label{eq:ex6}
\left(\Qvec_k, \Mvec_k\right)\rightharpoonup \left(\Qvec, \Mvec\right) \mbox{ in } W^{1,2}\left(\Omega; \mathbb R^{3\times 3}\right)\times W^{1,2}\left(\Omega; \mathbb R^3\right) \quad \mbox{ as } k\to \infty
\end{eqnarray} 
for some $(\Qvec,\Mvec) \in W^{1,2}\left(\Omega; \mathbb R^{3\times 3}\right)\times W^{1,2}\left(\Omega; \mathbb R^3\right)$.
To show that the limit pair $\left(\Qvec, \Mvec\right)$ satisfies the prescribed Dirichlet boundary condition, that is 
$\left(\Qvec, \Mvec\right) = \left(\Qvec_{bd}, \Mvec_{bd}\right) \mbox{ on } \partial\Omega$
in the sense of traces, it is enough to note that the trace operator $\gamma: W^{1, \frac{3}{2}}\left(\Omega;~ \cdot~\right)\rightarrow L^2\left(\partial\Omega\right)$ is compact (see e.g., \cite[Theorem 6.1-7(b)]{ciarlet2021mathematical}). By \eqref{eq:ex6}, $\Qvec$ and $\Mvec$ are bounded in $W^{1,2}$ and hence in $W^{1,\frac{3}{2}}$, we obtain,  up to subsequences that  
\begin{alignat}{8}
&\gamma\left(\Qvec_k\right)& &\rightarrow~& &\gamma\left(\Qvec\right)& &\mbox{ in }& &L^2\left(\partial\Omega; \mathbb R^{3\times 3}\right)& &\mbox{ as }& &k\rightarrow\infty&,\label{eq:ex7}\\
&\gamma\left(\Mvec_k\right)& &\rightarrow~& &\gamma\left(\Mvec\right)& &\mbox{ in }& &L^2\left(\partial\Omega; \mathbb R^3\right)& &\mbox{ as }& &k\rightarrow\infty&.\label{eq:ex8}
\end{alignat}
By extracting a pointwise converging sequence, we show that
$$\left(\Qvec, \Mvec\right) = \left(\Qvec_{bd}, \Mvec_{bd}\right) \mbox{ on } \partial\Omega$$
in the sense of traces. Hence the limit pair, $\left(\Qvec, \Mvec\right)\in\mathcal{A}\times\mathcal{S}$, and Proposition~\ref{prop:1} follows.
\end{proof}
Next, we show weak lower semicontinuity of the energy functional $\mathcal{E}$ in \eqref{en:2}.
\begin{proposition}\label{prop:2}
Let $\Hvec_{ext}\in C\left(\overline{\Omega}; \mathbb R^3\right)$. 
Then 
\begin{eqnarray}\label{wls1}
\mathcal{E}\left(\Qvec, \Mvec\right)\le \liminf_{k\rightarrow\infty} \mathcal{E}\left(\Qvec_k, \Mvec_k\right),
\end{eqnarray}when $\left(\Qvec_k, \Mvec_k\right)\rightharpoonup \left(\Qvec, \Mvec\right) \mbox{ in } W^{1,2}\left(\Omega; \mathbb{R}^{3\times 3}\right)\times W^{1,2}\left(\Omega; \mathbb R^3\right) \mbox{ as } k\to\infty.$
\end{proposition}
\begin{proof}
Set 
\begin{eqnarray}
\mathcal{G}\left(\Qvec,\Mvec\right)
=f\left(\Qvec, \Mvec\right)+ \frac{c_3}{2}\xi\left(M_3\right)^2 - \frac{c_2}{2}\Qvec\Hvec_{ext}\cdot\Hvec_{ext} - c_3\xi\Mvec\cdot\Hvec_{ext}\nonumber
\end{eqnarray} 
and assume without loss of generality that $\mathcal{G}\left(\Qvec,\Mvec\right)$ is non-negative, c.f.\ \eqref{eq:ex5}.
By the lower semicontinuity property of the norm, we have that
\begin{eqnarray}
||\nabla \Qvec||^2_{L^2\left(\Omega \right)} \le  \liminf_{k\rightarrow\infty}||\nabla \Qvec_k||^2_{L^2\left(\Omega\right)},\label{eq:ex9}\\
||\nabla \Mvec||^2_{L^2\left(\Omega\right)}\le  \liminf_{k\rightarrow\infty}||\nabla \Mvec_k||^2_{L^2\left(\Omega\right)}. \label{eq:ex10}
\end{eqnarray}
Using Rellich-Kondra\v{s}ov's embedding (see e.g., \cite[Theorem 6.1-5]{ciarlet2021mathematical}), we get   
\begin{alignat}{5}
&\Qvec_k& &\rightarrow\Qvec &~\mbox{ in }& L^q\left(\Omega;\mathbb R^{3\times 3}\right) &\quad q\in [1, \infty)&,\\
&\Mvec_k& &\rightarrow\Mvec &~\mbox{ in }& L^q\left(\Omega;\mathbb R^3\right) &\quad q\in [1, \infty)&.
\end{alignat}\label{eq:ex12}
Up to a subsequence without relabelling, we have
\begin{alignat}{5}
&\Qvec_k& &\rightarrow \Qvec& \mbox{ a.e. in } \Omega,\label{eq:ex13}\\
&\Mvec_k& &\rightarrow \Mvec& \mbox{ a.e. in } \Omega.\label{eq:ex14}
\end{alignat}
Using the closure property of pointwise convergence and Fatou's lemma \cite{Rindler}, we conclude that
\begin{eqnarray}
\int_{\Omega}\mathcal{G}\left(\Qvec,\Mvec\right)\d x\le\liminf_{k\rightarrow\infty}\int_{\Omega}\mathcal{G}\left(\Qvec_k,\Mvec_k\right)\d x.\label{eq:ex16}
\end{eqnarray}
Equations \eqref{eq:ex9}, \eqref{eq:ex10} and \eqref{eq:ex16} collectively yield the lower semi-continuity \eqref{wls1} of the ferronematic energy \eqref{en:2}. 
\end{proof}
Now, we state and prove the existence theorem.
\begin{theorem}[Existence of energy minimizers]\label{thm:1}
Assume that $\mathcal{A}\times\mathcal{S}\ne\emptyset$ and $\Hvec_{ext}\in C\left(\overline{\Omega}; \mathbb R^3\right)$. Then the energy functional $\mathcal{E}$ given by \eqref{en:2} admits a global minimizer in the set $\mathcal{A}\times\mathcal{S}$.
\end{theorem}
\begin{proof}
Using Propositions~\ref{prop:1} and \ref{prop:2}, we prove Theorem~\ref{thm:1} by the direct method in the calculus of variations (see e.g., \cite[Theorem 2.1]{Rindler}).
\end{proof}
Next, we show that the minimizer is unique for sufficiently large $l_1$ and $l_2$. 
The Euler-Lagrange equations associated with the energy functional $\mathcal{E}$ given by \eqref{en:2}  in the strong form, correspond to a vanishing $L^2$-gradient of $\mathcal{E}\left(\Qvec, \Mvec\right)$, i.e., $\nabla_{L^2}\mathcal{E}\left(\Qvec, \Mvec\right)=0$. The $L^2$-gradient of the functional $\mathcal{E}\left(\Qvec, \Mvec\right)$, $\nabla_{L^2}\mathcal{E}\left(\Qvec, \Mvec\right)$ is obtained by the Riesz representative of the Fr\'echet derivative of $\mathcal{E}\left(\Qvec, \Mvec\right)$ with respect to the given inner product. A standard calculation yields 
\begin{subequations}\label{EL}
\begin{align}
\label{EL:1}
2l_1\Delta Q_{11} &=  Q_{11}\left(Q_{11}^2 + Q_{12}^2 -1\right) - \frac{c_1}{2}\left(M_1^2 - M_2^2\right) - \frac{c_2}{2}\left(H_1^2 - H_2^2\right),\\
\label{EL:2}
2l_1\Delta Q_{12} &=  Q_{12}\left(Q_{11}^2 + Q_{12}^2 - 1\right) - c_1M_1M_2 - c_2H_1H_2,\\
\label{EL:3}
l_2\xi\Delta M_1 &= \xi M_1\left(M_1^2 + M_2^2 + M_3^2 - 1\right) - c_1\left(Q_{11}M_1 + Q_{12}M_2\right) - c_3\xi H_1,\\
\label{EL:4}
l_2\xi\Delta M_2 &= \xi M_2\left(M_1^2 + M_2^2 + M_3^2 - 1\right) - c_1\left(Q_{12}M_1 - Q_{11}M_2\right) - c_3\xi H_2,\\
\label{EL:5}
l_2\xi\Delta M_3 &= \xi M_3\left(M_1^2 + M_2^2 + M_3^2 - 1\right) - c_3\xi \left(H_3-M_3\right).
\end{align}
\end{subequations}
As an aside we mention that the Euler-Lagrange equations are the stationary case of the gradient flow equations considered in  \eqref{eq:26} below.

In the following, we derive local $L^{\infty}$-bounds for $H^1$ weak solutions of the Euler-Lagrange equations \eqref{EL:1}--\eqref{EL:5}. It is worth mentioning that we work with a bounded Lipschitz domain $\ \Omega\subset\mathbb R^2$, which only technically allows us to obtain a local bound for this elliptic system of PDEs.
\begin{lemma}\label{lem1}
Let $\left(\Qvec, \Mvec\right)\in\mathcal{A}\times\mathcal{S}$ be an $H^1$ weak solution of the system of PDEs \eqref{EL:1}--\eqref{EL:5}, where $\Hvec_{ext}\in C\left(\overline{\Omega}; \mathbb R^3\right)$ with $\sup_{\bar{\Omega}}|\Hvec_{ext}|\le k_1$. Then the weak solution $\left(\Qvec, \Mvec\right)$ has an interior local bound in terms of the positive parameters, $c_1, c_2, c_3,$ $k_1$ and $\xi$, i.e., for any compact set $K\subset\Omega$, there exists $\sigma = \sigma\left(c_1, c_2, c_3, \xi, k_1\right)$ such that 
\begin{eqnarray} 
\begin{aligned}
|Q_{11}(z)| &\leq \sigma~~~~~~~~~~~~~~~~~~~~~\forall z\in K,\label{eq:Qmax}\\
|Q_{12}(z)| &\leq \sigma~~~~~~~~~~~~~~~~~~~~~\forall z\in K,\\
|\Mvec(z)|^2 &\leq A c_1 \frac{\sigma}{\xi} + Bc_3k_1 + 1~~\forall z\in K, \label{eq:Mmax}
\end{aligned}
\end{eqnarray}
for some positive constants $A$ and $B$ independent of $c_1, c_2, c_3, \xi, l_1, l_2, k_1$.
\end{lemma}

\begin{proof}
By the Sobolev embedding $W^{1,2}\left(\Omega\right)\hookrightarrow L^q\left(\Omega\right)~\left(1\leq q<\infty\right)$ (see e.g., \cite[Theorem 6.1-3]{ciarlet2021mathematical}), $\left(\Qvec, \Mvec\right)\in L^q\left(\Omega; \mathbb R^{3\times 3}\right)\times L^q\left(\Omega; \mathbb R^3\right)$ for $1\leq q<\infty$. Applying the elliptic regularity theorem (c.f.\ \cite{evans2022partial}, sec. 6.3.1), we obtain that $\left(\Qvec, \Mvec\right)\in W^{2,2}_{loc}\left(\Omega; \mathbb R^{3\times 3}\right)\times W^{2,2}_{loc}\left(\Omega; \mathbb R^{3}\right)$. Then by the Sobolev embedding $W^{2,2}\left(\Omega\right)\hookrightarrow C^{0}\left(\overline{\Omega}\right)$ \cite[Theorem 6.1-3]{ciarlet2021mathematical} and continuing boot-strapping, we have $\left(\Qvec, \Mvec\right)\in C^{k}_{loc}\left(\Omega; \mathbb R^{3\times 3}\right)\times C^{k}_{loc}\left(\Omega; \mathbb R^{3}\right)$ for $k=0, 1, 2,\hdots$. Thus, we can assume that $|\Qvec|=\sqrt{2\left(Q_{11}^2+Q_{12}^2\right)}$ and $|\Mvec|=\sqrt{M_1^2+M_2^2+M_3^2}$ attain their maxima on a compact set $K\subset\Omega$. Suppose that $|\Qvec|$ attains its maximum at $z_1\in K$ and $|\Mvec|$ attains its maximum at $z_2\in K$. Then the symmetric Hessian $D^2\left(\frac{1}{2}|\Qvec|^2\right)$ is negative semi-definite at $z_1$, and the symmetric Hessian $D^2\left(\frac{1}{2}|\Mvec|^2\right)$ is negative semi-definite at $z_2$. Consequently, we have
$$\Delta \left(\frac{1}{2}|\Qvec|^2\right)\left(z_1\right)\le 0 \mbox{ and }\Delta \left(\frac{1}{2}|\Mvec|^2\right)(z_2)\le 0.$$
To proceed, we multiply \eqref{EL:1} by $Q_{11}$ and \eqref{EL:2} by $Q_{12}$, and add the resulting two equations. We then make use of the following identity $
\Delta\left(\frac{1}{2}|\Qvec|^2\right)=2Q_{11}\Delta Q_{11} + 2Q_{12}\Delta Q_{12}+|\nabla\Qvec|^2\le 0,$
and get
\begin{eqnarray}\label{max:28}
\bigg\{\left(\frac{1}{2}|\Qvec|^2-1\right)\frac{|\Qvec|^2}{2}-\frac{c_1}{2}Q_{11}(M_1^2 - M_2^2) - c_1Q_{12}M_1M_2\nonumber\\-\frac{c_2}{2}Q_{11}\left(H_1^2-H_2^2\right)-c_2Q_{12}H_1H_2\bigg\}\bigg|_{z=z_1}\le 0.
\end{eqnarray}
Similarly, we obtain
\begin{eqnarray}\label{max:29}
\bigg\{\xi\left(|\Mvec|^2-1\right)|\Mvec|^2-c_1M_1\left(Q_{11}M_1+Q_{12}M_2\right)-c_1M_2\left(Q_{12}M_1-Q_{11}M_2\right)\nonumber\\
-c_3\xi\left(H_1M_1+H_2M_2+H_3M_3\right)+c_3\xi M_3^2\bigg\}\bigg|_{z=z_2}\le 0.
\end{eqnarray}
Using the substitutions \begin{eqnarray}
Q_{11}&=\tau \cos\left(\theta\right), Q_{12}=\tau \sin\left(\theta\right),\nonumber\\
M_1&=\gamma \cos\left(\phi\right), M_2=\gamma \sin\left(\phi\right)
\end{eqnarray}
for arbitrary $\theta, \phi\in [0, 2\pi)$, where $\sqrt{2}\tau=|\Qvec|>0$ and $\gamma=\sqrt{M_1^2+M_2^2}>0$ in \eqref{max:28}, it follows that  
\begin{eqnarray} \label{eq:max2}
0 &\geq& \bigg\{\left(\frac{1}{2}|\Qvec|^2-1\right)\frac{|\Qvec|^2}{2}-\frac{c_1}{2}Q_{11}(M_1^2 - M_2^2) -c_1Q_{12}M_1M_2\nonumber\\ &&~~~~~~~ -\frac{c_2}{2}Q_{11}\left(H_1^2-H_2^2\right)-c_2Q_{12}H_1H_2\bigg\}\bigg|_{z=z_1}\\
&\geq & \bigg\{\tau^2(\tau^2 -1 ) - \frac{c_1}{2} \tau \gamma^2 - c_2 \tau k_1^2\bigg\} \bigg|_{z=z_1}. \nonumber
\end{eqnarray}
Similarly, we deduce from \eqref{max:29} that
\begin{eqnarray}\label{bd-M}
0 &\geq& \bigg\{\xi\left(|\Mvec|^2-1\right)|\Mvec|^2-c_1M_1\left(Q_{11}M_1+Q_{12}M_2\right)-c_1M_2\left(Q_{12}M_1-Q_{11}M_2\right)\nonumber\\
&&~~~~~~~~~ -c_3\xi\left(H_1M_1+H_2M_2+H_3M_3\right)+c_3\xi M_3^2\bigg\}\bigg|_{z=z_2}\nonumber\\
&\geq& \bigg\{\xi(\gamma^2 + M_3^2)\left(\gamma^2 + M_3^2 - 1 \right) - c_1\tau\gamma^2 - 2\xi c_3 k_1 \left(\gamma + |M_3| \right)\bigg\}\bigg|_{z=z_2}.\nonumber\\
\end{eqnarray}
Now, we deal with the following two possible cases in the estimate \eqref{bd-M}.\newline
\textbf{Case: I.}
If
$\gamma(z_2) + |M_3|(z_2)> 1$, it implies straightaway that
\[
\gamma(z_2) + |M_3|(z_2) < 2\left(\gamma^2 + M_3^2 \right)(z_2),
\]
and \eqref{bd-M} then implies that
\begin{eqnarray}
\bigg\{\xi\left(\gamma^2 + M_3^2 - 1 \right) - c_1\tau - 4\xi c_3k_1\bigg\}\bigg|_{z=z_2}\le 0.\label{max1'}
\end{eqnarray}
Using the fact that $\left(\gamma^2 + M_3^2\right)(z)\leq \left(\gamma^2 + M_3^2\right)(z_2)$ for any $z\in K$ and $\tau(z_2) \leq \tau(z_1)$ in \eqref{max1'}, we have
\begin{equation}
\label{eq:max1}
\left(\gamma^2 + M_3^2\right)(z) < \frac{c_1\tau(z_1)}{\xi} + 4 c_3k_1 + 1 \qquad \forall z \in K.
\end{equation}
Substituting \eqref{eq:max1} in \eqref{eq:max2}, we find that
\begin{equation} \label{eq:max3}
    g\left(\tau(z_1)\right) := \tau(\tau^2 -1 ) - c_2 k_1^2 - \frac{c_1}{2} \left(\frac{c_1\tau}{\xi} + 4 c_3k_1 + 1 \right) \bigg|_{z=z_1} \le 0.
\end{equation}This is a cubic polynomial in $\tau$. By Descartes' rule, the cubic polynomial $g\left(\tau\right)$ has at least one positive root. Let $\sigma$ denote the largest positive root of $g$. Consequently,
\begin{eqnarray} \label{sigma-sqrt}
\sigma > \sqrt{ 1 + \frac{c_1^2}{2\xi}}.
\end{eqnarray}
One can check that $g\left(\tau\right)$ is increasing in $[\sigma, \infty)$ and hence $g\left(\tau\left(z_1\right)\right) >0$ if $\tau(z_1) > \sigma$. Then \eqref{eq:max3} yields
\begin{equation} \label{eq:max4}
\tau(z)\leq\tau\left(z_1\right) \leq \sigma( c_1, c_2, c_3, k_1, \xi) \qquad \forall z \in K.
\end{equation}
Substituting \eqref{eq:max4} into \eqref{eq:max1}, as a result, we deduce the interior upper bound for $|\Mvec|^2 = \gamma^2 + M_3^2$ as stated in \eqref{eq:Mmax} for Case I.\newline\\
\textbf{Case: II.} If $\gamma(z_2) + |M_3|(z_2) \leq 1$, then it is immediate that $|\Mvec|^2 \leq 1$ for all $z \in K$. Further, by an argument as in Case I, we obtain
\begin{eqnarray}
\begin{aligned}
\tau &\leq \sigma^{\prime}\left(c_1, c_2, k_1\right) \qquad \forall z \in K
\end{aligned}
\end{eqnarray}
for some positive $\sigma^{\prime}$. 
\end{proof}

Next, using these local $L^\infty$-bounds for any weak solutions of the system of equations \eqref{EL}, we 
prove the uniqueness of global minimizers of the energy functional \eqref{eq:en12} for sufficiently large $l_1$ and $l_2$.
\begin{theorem}[Uniqueness of minimizer]\label{thm:2}
Let the parameters $c_1,c_2,c_3, \xi>0$ be fixed and $\Hvec_{ext}\in C\left(\overline{\Omega}; \mathbb R^3\right)$  with $\sup_{\bar{\Omega}}|\Hvec_{ext}|\le k_1$. For sufficiently large $l_1$ and $l_2$, there exists a unique minimizer of the ferronematic energy functional $\mathcal{E}$, defined in \eqref{en:2}, in the admissible space $\mathcal{A}\times\mathcal{S}$. 
\end{theorem}
\begin{proof}
On the contrary, we assume that $\left(\Qvec, \Mvec\right)$ and $\left(\overline{\Qvec}, \overline{\Mvec}\right)$ $\in$ $\mathcal{A}\times\mathcal{S}$ are two distinct minimizers of $\mathcal{E}$. By an elementary calculation and the definition of $f$ in \eqref{def-f}, we obtain
\begin{eqnarray}
& &\mathcal{E}\left(\frac{\Qvec+\overline{\Qvec}}{2}, \frac{\Mvec+\overline{\Mvec}}{2}\right)\nonumber\\
&=&\int_{\Omega}\frac{l_1}{2}\bigg|\nabla\left(\frac{\Qvec+\overline{\Qvec}}{2}\right)\bigg|^2 + \frac{\xi l_2}{2}\bigg|\nabla\left(\frac{\Mvec+\overline{\Mvec}}{2}\right)\bigg|^2 + f\left(\frac{\Qvec+\overline{\Qvec}}{2}, \frac{\Mvec+\overline{\Mvec}}{2}\right)+\frac{c_3}{2}\xi\left(\frac{M_3+\overline{M_3}}{2}\right)^2\nonumber\\
& &- \frac{c_2}{2}\left(\frac{\Qvec+\overline{\Qvec}}{2}\right)\Hvec_{ext}\cdot\Hvec_{ext} 
- c_3\xi\left(\frac{\Mvec+\overline{\Mvec}}{2}\right)\cdot\Hvec_{ext}\;\d x\nonumber\\
&=&\frac{1}{2}\mathcal{E}\left(\Qvec, \Mvec\right) + \frac{1}{2}\mathcal{E}\left(\overline{\Qvec}, \overline{\Mvec}\right) + \int_{\Omega}f\left(\frac{\Qvec+\overline{\Qvec}}{2}, \frac{\Mvec+\overline{\Mvec}}{2}\right)-\frac{1}{2}\left(f\left(\Qvec, \Mvec\right) + f\left(\overline{\Qvec}, \overline{\Mvec}\right)\right)\nonumber\\
& &- \frac{l_1}{8}\left(\nabla\Qvec-\nabla\overline{\Qvec}\right)^2-\frac{l_2\xi}{8}\left(\nabla\Mvec-\nabla\overline{\Mvec}\right)^2 - \frac{c_3\xi}{8}\left(M_3-\overline{M_3}\right)^2\d x\nonumber\\
&\le&\frac{1}{2}\left(\mathcal{E}\left(\Qvec, \Mvec\right) + \mathcal{E}\left(\overline{\Qvec}, \overline{\Mvec}\right)\right)+\int_{\Omega}f\left(\frac{\Qvec+\overline{\Qvec}}{2}, \frac{\Mvec+\overline{\Mvec}}{2}\right)-\frac{1}{2}\left(f\left(\Qvec, \Mvec\right) + f\left(\overline{\Qvec}, \overline{\Mvec}\right)\right)\d x\nonumber\\
& &- \frac{l_1}{8c_p}||\Qvec-\overline{\Qvec}||^2_{L^2\left(\Omega; \mathbb R^{3\times 3}\right)} -\frac{l_2\xi}{8c_p}||\Mvec-\overline{\Mvec}||^2_{L^2\left(\Omega; \mathbb R^3\right)} - \frac{c_3\xi}{8}||M_3-\bar{M_3}||^2_{L^2\left(\Omega; \mathbb R\right)},\nonumber\\ \label{eq:un1}
\end{eqnarray}
where we have used the Poincar\'e inequality 
with constant $c_p$ (independent of $\Omega$) in the last line. Since $\Qvec-\overline{\Qvec}, \Mvec-\overline{\Mvec}\in \overline{C^{\infty}_{c}\left(\Omega\right)}$, there exists a compact set $K\subset\Omega$ such that
\begin{eqnarray}
\begin{aligned}
\Qvec-\overline{\Qvec}=0 \mbox{ a.e. in } \Omega\setminus K,\label{eq:tay}\\
\Mvec-\overline{\Mvec}=0 \mbox{ a.e. in } \Omega\setminus K.
\end{aligned}
\end{eqnarray}
Moreover, $f\left(\Qvec, \Mvec\right)$ is $C^{\infty}$ w.r.t $Q_{ij}, M_j$ and we deal with the following two scenarios.\newline\\
\textbf{Case I.} Assume that $|\left(\Qvec, \Mvec\right)-\left(\overline{\Qvec}, \overline{\Mvec}\right)|\geq 1$. We know that $f$ is locally Lipschitz. As a consequence, we have
\begin{eqnarray}
& &f\left(\frac{\Qvec+\overline{\Qvec}}{2}, \frac{\Mvec+\overline{\Mvec}}{2}\right)-\frac{1}{2}\left(f\left(\Qvec, \Mvec\right) + f\left(\overline{\Qvec}, \overline{\Mvec}\right)\right)\nonumber\\
&\le&\lambda|\left(\Qvec, \Mvec\right)-\left(\overline{\Qvec}, \overline{\Mvec}\right)|^2,\label{eq:case1}
\end{eqnarray}
where $\lambda$ is the Lipschitz constant and $\lambda=\max_{\left(\Qvec, \Mvec\right)}|f'\left(\Qvec, \Mvec\right)|$. Applying Lemma~\ref{lem1}, the first partial derivatives of $f$ over the compact set $K$ can be estimated as 
\begin{eqnarray}
\Big|\dfrac{\partial f}{\partial Q_{1j}}\Big|&\le&|Q_{ij}|\left|\tfrac12|\Qvec|^2-1\right|+c_1|\Mvec|^2,\nonumber\\
&\le& \sigma^3 + \sigma + c_1\left(A c_1 \frac{\sigma}{\xi} + Bc_3k_1 + 1\right):=a_1,\nonumber\\
\label{eq:first_de1}\\
\Big|\dfrac{\partial f}{\partial M_{i}}\Big|&\le&\xi|M_i|\left||\Mvec|^2-1\right|+2c_1\left(\tfrac12|\Qvec|^2+|\Mvec|^2\right)\nonumber\\
&\le&\xi\sqrt{\left(A c_1 \frac{\sigma}{\xi} + Bc_3k_1 + 1\right)}\left(A c_1 \frac{\sigma}{\xi} + Bc_3k_1 + 2\right)+2c_1\left(\sigma^2+ A c_1 \frac{\sigma}{\xi} + Bc_3k_1  + 1\right)\nonumber\\
&:=&a_2.\nonumber\\
\label{eq:first_de2}
\end{eqnarray}
\textbf{Case II.} Assume that $|\left(\Qvec, \Mvec\right)-\left(\overline{\Qvec}, \overline{\Mvec}\right)| < 1$.\newline
By Taylor's expansion followed by an elementary calculation, we obtain that
\begin{eqnarray}\label{eq:un2'}
\lefteqn{
f\left(\frac{\Qvec+\overline{\Qvec}}{2}, \frac{\Mvec+\overline{\Mvec}}{2}\right)-\frac{1}{2}\left(f\left(\Qvec, \Mvec\right) + f\left(\overline{\Qvec}, \overline{\Mvec}\right)\right)}\nonumber\\
&\le& \frac{1}{4}\bigg|f^{''}\left(\frac{\Qvec+\overline{\Qvec}}{2}, \frac{\Mvec+\overline{\Mvec}}{2}\right)\bigg||\left(\Qvec, \Mvec\right)-\left(\overline{\Qvec}, \overline{\Mvec}\right)|^2 \nonumber\\
&&+\frac{1}{16}\bigg|f^{iv}\left(\frac{\Qvec+\overline{\Qvec}}{2}, \frac{\Mvec+\overline{\Mvec}}{2}\right)\bigg||\left(\Qvec, \Mvec\right)-\left(\overline{\Qvec}, \overline{\Mvec}\right)|^4\nonumber\\
&\le& \frac{1}{4}\bigg|f^{''}\left(\frac{\Qvec+\overline{\Qvec}}{2}, \frac{\Mvec+\overline{\Mvec}}{2}\right)\bigg||\left(\Qvec, \Mvec\right)-\left(\overline{\Qvec}, \overline{\Mvec}\right)|^2 \nonumber\\ && +\frac{1}{8}\bigg|f^{iv}\left(\frac{\Qvec+\overline{\Qvec}}{2}, \frac{\Mvec+\overline{\Mvec}}{2}\right)\bigg||\left(\Qvec, \Mvec\right)-\left(\overline{\Qvec}, \overline{\Mvec}\right)|^2.\nonumber\\
\end{eqnarray}
Applying Lemma~\ref{lem1}, the estimates for the second partial derivatives of $f\left(\Qvec, \Mvec\right)$ on the compact set $K\subset\Omega$, are obtained as 
\begin{eqnarray}
\begin{aligned}
\bigg|\frac{\partial^2 f}{\partial Q_{1i}\partial Q_{1j}}\bigg|&\le A_1\sigma^2:=a'_1,\\
\biggl|\frac{\partial^2 f}{\partial M_{i}\partial M_{j}}\biggr|&\le A_2c_1\sigma+B_2\xi c_3k_1+C_2\xi:=a'_2,\\
\biggl|\frac{\partial^2 f}{\partial Q_{1i}\partial M_{j}}\biggr|&\le A_3c_1\sqrt{Ac_1\frac{\sigma}{\xi}+Bc_3k_1+1 }:=a'_3,\\
\end{aligned}\label{eq:estimates}
\end{eqnarray}
where $A_1, A_2, A_3, B_2$ and $C_2$ are some positive constants independent of $c_1, c_2, c_3, \xi, k_1, l_1, l_2$.\newline\\
Due to the estimates \eqref{eq:case1}, \eqref{eq:first_de1}, \eqref{eq:first_de2}, \eqref{eq:un2'}, \eqref{eq:estimates}, and using \eqref{eq:tay}, we get
\begin{eqnarray}\label{eq:un2}
& &\int_{\Omega}f\left(\frac{\Qvec+\overline{\Qvec}}{2}, \frac{\Mvec+\overline{\Mvec}}{2}\right)-\frac{1}{2}\left(f\left(\Qvec, \Mvec\right) + f\left(\overline{\Qvec}, \overline{\Mvec}\right)\right)\d x\nonumber\\
&\le& \left(a_1+a'_1+a'_3+\zeta\right)||\Qvec-\overline{\Qvec}||^2_{L^2\left(\Omega; \mathbb R^{3\times 3}\right)} + \left(a_2+a'_2+a'_3+\zeta\right)||\Mvec-\overline{\Mvec}||^2_{L^2\left(\Omega;\mathbb R^3\right)},
\end{eqnarray}
where the constant $\zeta := \bigg|f^{iv}\left(\frac{\Qvec+\overline{\Qvec}}{2}, \frac{\Mvec+\overline{\Mvec}}{2}\right)\bigg|$. 
Now, using \eqref{eq:un2} in \eqref{eq:un1}, we have
\begin{eqnarray}\label{eq:un4}
\lefteqn{\mathcal{E}\left(\frac{\Qvec+\overline{\Qvec}}{2}, \frac{\Mvec+\overline{\Mvec}}{2}\right)}\nonumber\\
&\le&\frac{1}{2}\mathcal{E}\left(\Qvec, \Mvec\right) + \frac{1}{2}\mathcal{E}\left(\bar{\Qvec}, \overline{\Mvec}\right)+\left(a_1+a'_1+a'_3+\zeta-\frac{l_1}{8c_p}\right)||\Qvec-\overline{\Qvec}||^2_{L^2\left(\Omega; \mathbb R^{3\times 3}\right)}\nonumber\\
&&+ \left(a_2+a'_2+a'_3+\zeta-\frac{l_2\xi}{8c_p}\right)||\Mvec-\overline{\Mvec}||^2_{L^2\left(\Omega; \mathbb R^3\right)} - \frac{c_3\xi}{8}||M_3-\overline{M_3}||^2_{L^2\left(\Omega; \mathbb R\right)}.
\end{eqnarray}
From \eqref{eq:un4}, it can be observed that
\begin{eqnarray}\label{eq:un5}
\begin{aligned}
\mathcal{E}\left(\frac{\Qvec+\overline{\Qvec}}{2}, \frac{\Mvec+\overline{\Mvec}}{2}\right)<\frac{1}{2}\mathcal{E}\left(\Qvec, \Mvec\right) + \frac{1}{2}\mathcal{E}\left(\overline{\Qvec}, \overline{\Mvec}\right),
\end{aligned}
\end{eqnarray}
when $l_1, l_2>l'\left(c_1, c_2, c_3, \xi, k_1\right)=\max \Big\{8c_p\left(a_1+a'_1+a'_3+\zeta\right), \frac{8c_p}{\xi}\left(a_2+a'_2+a'_3+\zeta\right)\Big\}$.\newline
To obtain a contradiction, we assume that $\left(\Qvec, \Mvec\right)\ne \left(\overline{\Qvec}, \overline{\Mvec}\right)\in\mathcal{A}\times\mathcal{S}$ are minimizers of $\mathcal{E}$ for $l_1,l_2\in\left(l', \infty\right)$. Then the mapping $\phi:[0,1]\rightarrow\mathbb R$, defined by
\begin{eqnarray}
\begin{aligned}
\phi\left(\lambda\right)=\mathcal{E}\left(\lambda\Qvec+\left(1-\lambda\right)\overline{\Qvec}, \lambda\Mvec + \left(1-\lambda\right)\overline{\Mvec}\right),
\end{aligned}
\end{eqnarray}
is $C^1$ and $\phi'\left(\lambda\right)=0$ at $\lambda=0,1$. This contradicts the strictly convex nature c.f.\ \eqref{eq:un5} of the energy functional $\mathcal{E}$. Hence Theorem~\ref{thm:2} follows.
\end{proof}
\section{Numerical approach}\label{sec:5}
We now define the numerical method and setup for our numerical experiments in the next section. We take the re-scaled domain $\Omega$ to be such that:
\[
\Omega = \left\{(x,y) \in \mathbb{R}^2: 0\leq x,y \leq 1 \right\}.
\]We use the following energy dissipation law \cite{onsager1931reciprocal}
\begin{eqnarray}\label{eq:25}
\frac{d}{dt}\mathcal{E}\left(\Qvec, \Mvec\right) = -\int_{\Omega}\left(\eta_1|\partial_t\Qvec|^2 + \eta_2|\partial_t\Mvec|^2\right),
\end{eqnarray}
where $\eta_1>0$ and $\eta_2>0$ are arbitrary friction coefficients \cite{Canevari} that control the energy dissipation. The $L^2$-gradient flow equations for the ferronematic energy \eqref{en:2} can be written as
\begin{eqnarray}
\left\{ \begin{aligned}
\eta_1\frac{\partial}{\partial t}\begin{pmatrix}
Q_{11}\\
Q_{12}
\end{pmatrix} &= 2l_1\Delta \begin{pmatrix}
Q_{11}\\
Q_{12}
\end{pmatrix} - \left(\frac{1}{2}|\Qvec|^2 -1\right)\begin{pmatrix}
Q_{11}\\
Q_{12}
\end{pmatrix} + \frac{c_1}{2}\begin{pmatrix}
M_1^2 - M_2^2\\
2M_1M_2
\end{pmatrix}+ \frac{c_2}{2}\begin{pmatrix}
H_1^2 - H_2^2\\
2H_1H_2
\end{pmatrix},\\
\eta_2\frac{\partial}{\partial t}\begin{pmatrix}
M_1\\
M_2\\
M_3
\end{pmatrix}&= \xi l_2\Delta \begin{pmatrix}
M_1\\
M_2\\
M_3
\end{pmatrix} - \xi\left(|\Mvec|^2-1\right)\begin{pmatrix}
M_1\\
M_2\\
M_3
\end{pmatrix} + c_1\begin{pmatrix}
Q_{11}M_1 + Q_{12}M_2\\
Q_{12}M_1 - Q_{11}M_2\\
0
\end{pmatrix}
+ c_3\xi\begin{pmatrix}
H_1\\
H_2\\
H_3-M_3
\end{pmatrix}.
\end{aligned}\right.\nonumber\\ \label{eq:26}
\end{eqnarray}

We complete the system of PDEs \eqref{eq:26} by the continuous $+k$ degree boundary and initial conditions where $k$ measures the number of rotations by the director or by the magnetization around the boundary. We set $\theta\left(x, y\right) = \operatorname{arctan2}\left(y - 0.5, x - 0.5\right) - \frac{\pi}{2}$ for any $(x,y) \in \overline{\Omega}$ and define the boundary conditions as
\begin{gather}
\begin{aligned}
\left(Q^b_{11}, Q^b_{12}\right)&=\left(\cos(2k\theta),\sin(2k\theta)\right),\\
\left(M^b_{1}, M^b_{2}, M^b_{3}\right)&=\left(\sqrt 3\cos(k\theta), \sqrt 3\sin(k\theta), 0\right),\label{eq:27'}
\end{aligned}
\end{gather}
and the initial conditions as
\begin{gather}
\begin{aligned}
\left(Q^0_{11}, Q^0_{12}\right)&=\left(\cos(2k\theta),\sin(2k\theta)\right),\\
\left(M^0_{1}, M^0_{2}, M^0_{3}\right)&=\left(\sqrt 3\cos(k\theta), \sqrt 3\sin(k\theta), \sqrt 3\right).\label{eq:28'}
\end{aligned}
\end{gather}
The choice of boundary condition is motivated by \cite{Canevari}, where a super-dilute ferronematic suspension is considered. That particular super-dilute suspension is based on the following assumptions: the domain size is very large and the nematic-magnetic coupling strength is very small. In the energy law, the parameters $l_1$ and $l_2$ are assumed to scale as $\mathcal{O}(\xi^2)$ and $c_1=\mathcal{O}\left(\xi\right)$ with $\xi$ being sufficiently small. In addition, we assume that the magnetostatic energy or stray field energy is the minimum on the boundary, i.e., $M_3=0$ on the boundary. 
To define a suitable boundary condition, (inspired by \cite{Canevari}) we let $c_1=2\beta\xi$ for some $\beta>0$ and redefine the bulk potential 
\begin{eqnarray}
f_{\xi}\left(\Qvec, \Mvec\right)&=&\frac{1}{4}\left(\frac{|\Qvec|^2}{2}-1\right)^2+\frac{\xi}{4}\left(|\Mvec|^2-1\right)^2-\beta\xi\Qvec\Mvec\cdot\Mvec+k_{\xi}\nonumber\\
&=& \frac{1}{4}\left(Q_{11}^2+Q_{12}^2-1\right)^2+\frac{\xi}{4}\left(M_1^2+M_2^2-1\right)^2-\beta\xi\sum_{i,j=1,2}Q_{ij}M_iM_j+k_{\xi},\nonumber\\ \label{eq:bulk'}
\end{eqnarray}
where $k_{\xi}$ is uniquely defined by imposing the condition $\inf f_{\xi}=0$, and satisfies $k_{\xi}\rightarrow 0$ as $\xi\rightarrow 0$. The choice of $\Qvec$ and $\Mvec$ on the boundary is determined by the condition that $f_{\xi}(\Qvec, \Mvec)\rightarrow 0$ when $\xi\rightarrow 0$. 
For any $\beta>0$, a calculation following \cite[Appendix~B]{Canevari} yields a unique choice $|\Mvec|=\sqrt{2\beta+1}$ and $|\Qvec|=\sqrt{1+2\beta^2\xi+\beta\xi}$, that satisfies the mentioned physical assumption. Further, the boundary conditions in \eqref{eq:27'} are obtained in the limit $\xi\rightarrow 0$ and $\beta=1$.
\begin{remark}
On the boundary, we impose $M_3=0$ so that the magnetic stray field energy remains minimal. A non-zero choice of $M_3$ can also be considered following a similar argument as discussed above. As we are interested in non-zero solutions of $M_3$, we prescribe a non-zero $M_3$ in the initial condition \eqref{eq:28'}. Moreover, we choose the boundary condition \eqref{eq:27'} and the initial condition \eqref{eq:28'} in such a way that the director and magnetization vector maintain co-alignment. \newline
\end{remark}
\subsection{Spatial and temporal discretization}
The time-independent or equilibrium solutions to the $L^2$-gradient flow system, are solutions of the Euler--Lagrange equations associated with \eqref{eq:en12}. From the numerical point of view, solving the $L^2$-gradient flow system is more straightforward than solving a coupled Euler--Lagrange system, primarily due to the inclusion of time relaxation that allows us to use the Crank--Nicolson finite difference method \cite{Burden-Faires, smith1985numerical} as outlined below.
We discretise the unit square $\Omega$ with a spatial grid of $N$ points in the $x$-direction and $N$ points in the $y$-direction. We assume that each point in the 2D grid is represented by $x_i=i*\delta x$ and $y_j=j*\delta y$ with equal step sizes $\delta x$, $\delta y$ and spatial indices $i,j$ $(i,j=1,2,\ldots, N)$. Similarly, we represent temporal grids by $t_n=n*\delta t$ with step size $\delta t$ and temporal index $n$ $(n=0,1,2,\ldots )$. Therefore, the numerical solution of the system of PDEs, at a given spatial point and temporal point, can be denoted by $\mathcal{Z}^n_{ij}$.
\subsection{Crank--Nicolson method}
To obtain the working formula of the Crank--Nicolson method, we introduce the following forward and central finite differences \cite{Burden-Faires, smith1985numerical} to replace the partial derivatives in the system of PDEs with their discrete versions:
$$\partial_t \mathcal{Z}^n_{ij} \approx \frac{\mathcal{Z}^{n+1}_{ij} - \mathcal{Z}^{n}_{ij}}{\delta t}$$
and
$$\partial_{xx} \mathcal{Z}^n_{ij} \approx \frac{\mathcal{Z}^{n}_{i-1j} - 2\mathcal{Z}^{n}_{ij} + \mathcal{Z}^{n}_{i+1j}}{\delta x^2};~\partial_{yy} \mathcal{Z}^n_{ij} \approx \frac{\mathcal{Z}^{n}_{ij-1} - 2 \mathcal{Z}^{n}_{ij} + \mathcal{Z}^{n}_{ij+1}}{\delta y^2}.$$
Here, $\mathcal{Z}$ is the prototype of the variables $Q_{11}, Q_{12}, M_1, M_2$ and $M_3$. In the general setting, the gradient flow system \eqref{eq:26} can be written as
\begin{eqnarray}
\eta\frac{\partial \mathcal{Z}}{\partial t}=l\Delta\mathcal{Z} + f\left(\mathcal{Z}\right),\label{CN}
\end{eqnarray}
where $\eta\in\{\eta_1,\eta_2\}$ and $l\in\{2l_1,l_2\xi\}$. 

The finite difference-based Crank--Nicolson discretization of \eqref{CN} yields:
\begin{eqnarray}
\begin{aligned}
\eta \frac{\mathcal{Z}^{n+1}_{ij}-\mathcal{Z}^n_{ij}}{\delta t}=\frac{l}{2}\left(\frac{\mathcal{Z}^{n+1}_{i-1j} - 2\mathcal{Z}^{n+1}_{ij} + \mathcal{Z}^{n+1}_{i+1j}}{\delta x^2} + \frac{\mathcal{Z}^{n}_{i-1j} - 2\mathcal{Z}^{n}_{ij} + \mathcal{Z}^{n}_{i+1j}}{\delta x^2}\right )\\
+l\left(\frac{\mathcal{Z}^{n}_{ij-1} - 2 \mathcal{Z}^{n}_{ij} + \mathcal{Z}^{n}_{ij+1}}{\delta y^2}\right)+f(\mathcal{Z}^{n+1}_{ij}).\label{CN1}
\end{aligned}
\end{eqnarray}
We implement the finite difference scheme along the $x$-direction i.e., the derivative in the $x$-direction is computed as the average of two time steps at the $\left(n+1\right)$th and $n$th time-steps, while the derivative in the $y$-direction is discretized at the $n$th time step. To linearize the nonlinear terms, we employ Newton's linearization technique which involves the substitution: 
\begin{eqnarray}
\mathcal{Z}^{n+1}_{ij}=\mathcal{Z}^n_{ij}+\delta \mathcal{Z}^n_{ij},\label{CN2}
\end{eqnarray}
where $\delta \mathcal{Z}^n_{ij}$ labels the error committed at each grid point in two consecutive approximations (i.e., at the $(n+1)$th and $n$th time steps). Since the error $\delta \mathcal{Z}^n_{ij}$ is very small, terms involving the square and higher powers of $\delta \mathcal{Z}^n_{ij}$ are disregarded. The resulting system of algebraic equations leads to a tri-diagonal system of the form
$$a_{i}\delta \mathcal{Z}^{n}_{i-1j}+b_{i}\delta \mathcal{Z}^{n}_{ij}+c_{i}\delta \mathcal{Z}^{n}_{i+1j}=d_{i}(\mathcal{Z}^n_{ij}),$$
where $a_{i}$, $b_{i}$, $c_{i}$, and $d_{i}$ are known from the solutions at the $n$th time step. The system of algebraic equations is solved for $\delta \mathcal{Z}_{ij}^n$ by employing the Tri-diagonal matrix algorithm \cite{Burden-Faires}. Once we determine the error at each grid point $\left(i, j\right)$, the solution is improved at the $\left(n+1\right)$th step by equation \eqref{CN2}. We refer to a complete implementation of the algorithm as an iteration 
and we iterate until convergence is achieved. 
 At each fixed $n$, we obtain a sequence of approximations $\{(\mathcal{Z}^{n}_{ij})^{(k)}\}_{k=1}^{\infty}$ that converge to a solution until $$\left\|\left(\mathcal{Z}^{n}_{ij}\right)^{(k+1)}-\left(\mathcal{Z}^{n}_{ij}\right)^{(k)}\right\|<\epsilon\left(=10^{-6}< \mathcal{O}\left(\delta x^2 + \delta t^2\right)\right)$$ is satisfied.
 
\subsection{Numerical stability}
\label{sec:experiments}
Our numerical results are generated in MATLAB. We perform numerical experiments with $\eta_1=\eta_2=1$ to deduce the values of $\delta x, \delta y$ and $\delta t$ that ensure convergence. 
 These experiments suggest the following choices: $\delta x=\delta y=0.02$ and $\delta t=0.00001$. 
To support these choices, we present the spatial convergence analysis in \Cref{Fig:Com1} and the temporal convergence analysis in \Cref{Fig:Com2}. \begin{figure}[ht!]
     \centering
     \begin{subfigure}[b]{0.45\textwidth}
         \centering
         \includegraphics[width=\textwidth]{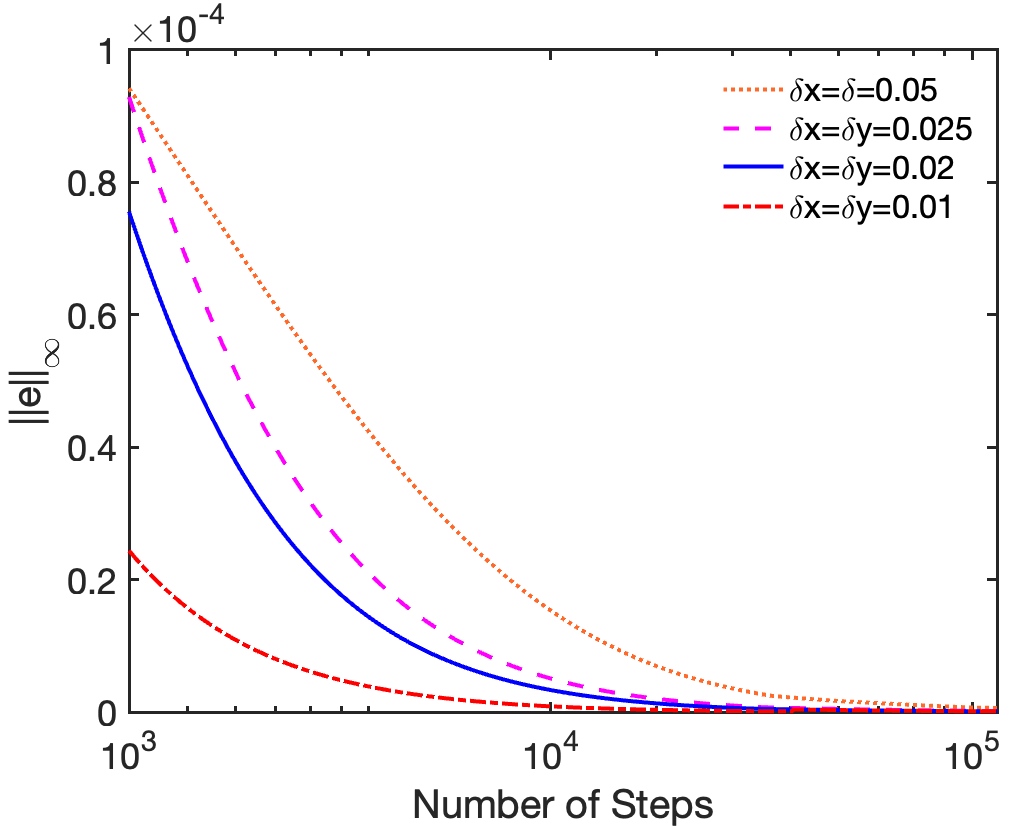}
         \caption{}
         \label{Fig:Com1}
     \end{subfigure}
     \begin{subfigure}[b]{0.45\textwidth}
         \centering
         \includegraphics[width=\textwidth]{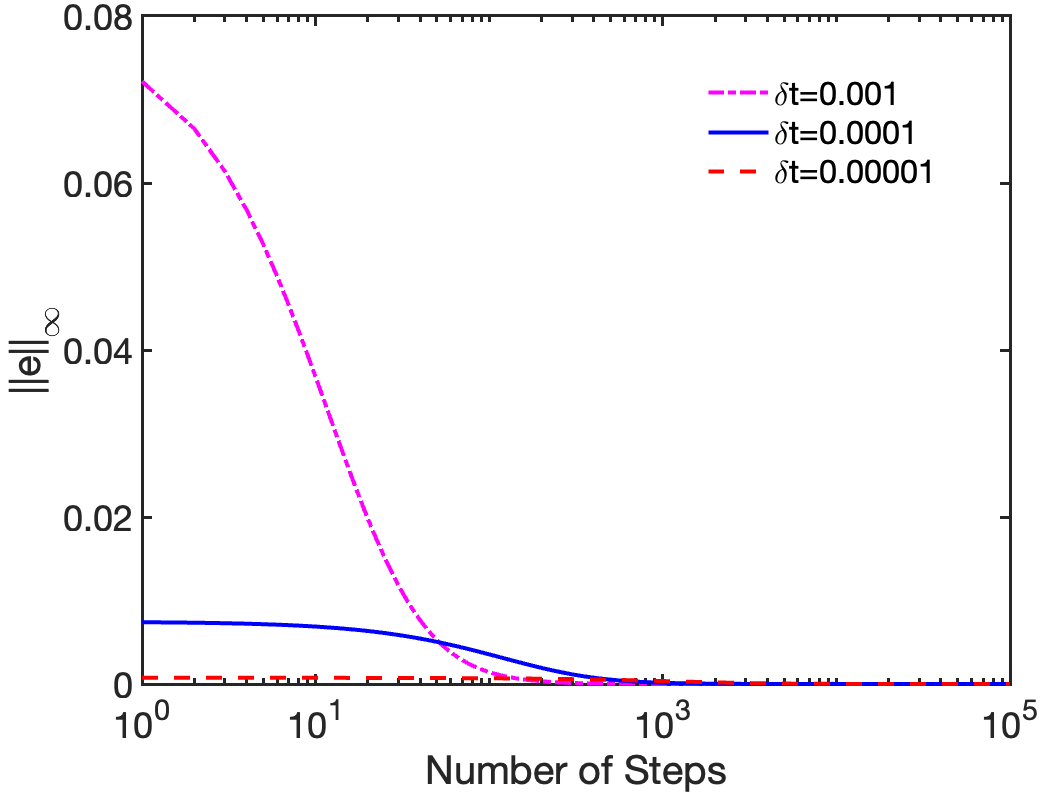}
         \caption{}
         \label{Fig:Com2}
     \end{subfigure}
\caption{(a) Spatial convergence, (b) Temporal convergence. The error $||e||_\infty $ decreases to $10^{-6}$ when the number of steps exceeds $10^{4}$. These figures indicate that the selection $\delta x=\delta y=0.02$ and $\delta t=0.00001$ is reasonable to meet our stability requirements.}
     \label{Fig1}
\end{figure}
For the convergence tests, the error $||e||_{\infty}:=||\left(\mathcal{Z}^{n}\right)^{(k+1)}-\left(\mathcal{Z}^{n}\right)^{(k)}||_{\infty}$ is calculated at each iteration. For the spatial convergence test, we study the error propagation with the number of iterations for different grid sizes $\delta x$ and $\delta y$, with a fixed $\delta t=0.00001$. For the temporal convergence test, we study the error propagation with the number of iterations for different $\delta t$ with fixed $\delta x=\delta y=0.02$. In both cases, the error $||e||_\infty$ decreases to $10^{-6}$ when the number of iterations exceeds $10^{4}$ and the error becomes independent of the choice of grid sizes. Furthermore, this study indicates that the computational scheme has second-order accuracy in space and time, which is consistent with the theory of the Crank--Nicolson FD method \cite{Burden-Faires, smith1985numerical} (see \Cref{Fig1}).

\subsection{Test Example}\label{sec:test}

We validate our numerical scheme and gain some preliminary insights by reproducing the results in Bisht et al.\ \cite{Bisht2020}. The energy \eqref{eq:en12} reduces to the ferronematic energy employed in \cite{Bisht2020} when $M_3 = 0$ and $\Hvec_{ext}=0$, 
and \eqref{eq:26} reduces to the associated set of Euler--Lagrange equations when $\eta_1=\eta_2=0$.
We use the following Dirichlet boundary conditions for $\Qvec$ and $\Mvec$: $\left(Q^b_{11}, Q^b_{12}\right)=\left(-1, 0\right)$ and $\left(M^b_{1}, M^b_{2}, M^b_{3}\right)=\left(0, 1, M^b_{3}\right)$ at $x=0$, $\left(Q^b_{11}, Q^b_{12}\right)=\left(-1, 0\right)$ and $\left(M^b_{1}, M^b_{2}, M^b_{3}\right)=\left(0, -1, M^b_{3}\right)$ at $x=1$, $\left(Q^b_{11}, Q^b_{12}\right)=\left(1, 0\right)$ and $\left(M^b_{1}, M^b_{2}, M^b_{3}\right)=\left(-1, 0, M^b_{3}\right)$ at $y=0$, $\left(Q^b_{11}, Q^b_{12}\right)=\left(1, 0\right)$ and $\left(M^b_{1}, M^b_{2}, M^b_{3}\right)=\left(1, 0, M^b_{3}\right)$ at $y=1$. These Dirichlet boundary conditions are compatible with tangent boundary conditions for $\Qvec$, i.e., the nematic director $\nvec = (\pm 1, 0)$ on $y=0,1$ and $\nvec = (0, \pm 1)$ on $x=0,1 $, so that $\nvec$ is tangent to the square edges. Tangent boundary conditions are experimentally relevant and have been used extensively in experiments and theoretical studies \cite{luo2012, tsakonas2007}. In this 2D framework, defects are identified by $\Qvec=0$ (where the nematic director is not defined) and the nodal set of $\Mvec$.
We complete the system with the arbitrary initial guesses $\left(Q^i_{11}, Q^i_{12}\right)=\left(c_1, c_1\right)$ and $\left(M^i_{1}, M^i_{2}, M^i_{3}\right)=\left(c_1, c_1, M^i_{3}\right)$.

We test our code by setting $M_3=0$, $M^i_3=0$, $M^b_{3}=0$, $\Hvec_{ext}=\mathbf{0}$ and considering the limit $\eta_1,\eta_2\to0$, so that the system of equations \eqref{eq:26} agrees with the system of PDEs in \cite{Bisht2020}. We set  $l'_1=l_2=0.001$ where $l_1'=2l_1$ and $\xi=1$. For these choices, we obtain the stable ferronematic profiles $\left(\Qvec^1_{D, 0.5}, \Mvec^1_{D, 0.5}\right)$ for $c_1=0.5$, as reported in \cite{Bisht2020}. Similarly, we fix $l'_1=l_2=0.001$, $\xi=10$ and obtain the ferronematic profiles $\left(\Qvec^{10}_{D, 0.25}, \Mvec^{10}_{D, 0.25}\right)$ for $c_1=0.25$, as reported in \cite{Bisht2020}. In both cases $\eta_1=\eta_2=0.0005$ to capture the limit $\eta_1,\eta_2\to0$. The solutions are plotted in \Cref{PRE1}, thus verifying the accuracy of the numerical method and MATLAB implementation. 
\begin{figure}[ht!]
\vspace{0.5cm}
     \centering
     \begin{subfigure}{0.25\textwidth}
         \centering
         \includegraphics[width=\textwidth]{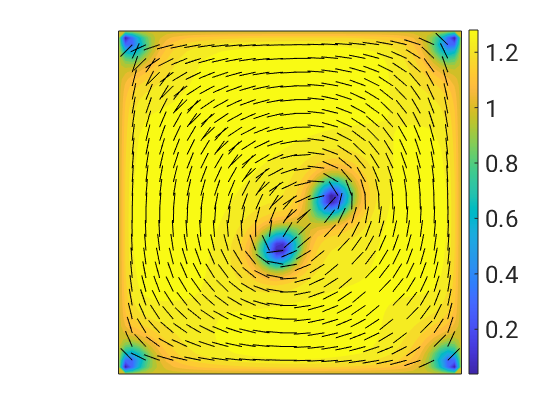}
         \caption{$\Qvec^1_{D, 0.5}$}
         \label{Fig:Comparison1}
     \end{subfigure}
     \hspace{0.1cm}
     \begin{subfigure}{0.25\textwidth}
         \centering
         \includegraphics[width=\textwidth]{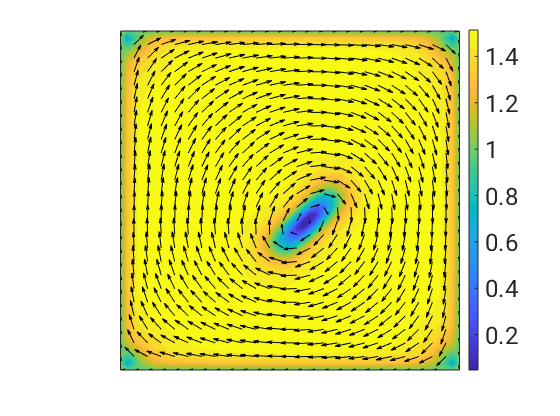}
          \caption{$\Mvec^1_{D, 0.5}$}
         \label{Fig:Comparison2}
     \end{subfigure}
     \\
     \centering
     \begin{subfigure}[b]{0.25\textwidth}
         \centering
         \includegraphics[width=\textwidth]{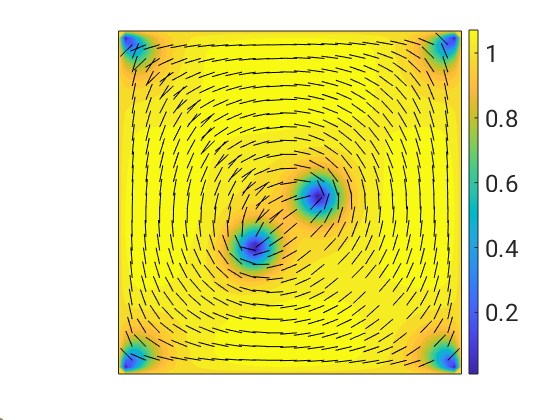}
         \caption{$\Qvec^{10}_{D, 0.25}$}
         \label{Fig:Comparison3}
     \end{subfigure}
     \hspace{0.1cm}
     \centering
     \begin{subfigure}[b]{0.25\textwidth}
         \centering
         \includegraphics[width=\textwidth]{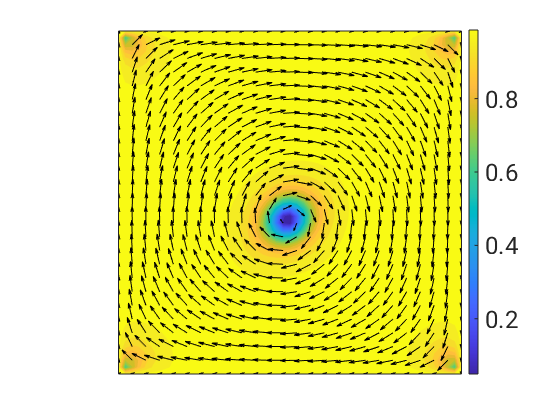}
        \caption{$\Mvec^{10}_{D, 0.25}$}
        \label{Fig:Comparison4}
     \end{subfigure}
\caption{Ferronematic profiles $\left(\Qvec^1_{D, 0.5}, \Mvec^1_{D, 0.5}\right)$ and $\left(\Qvec^{10}_{D, 0.25}, \Mvec^{10}_{D, 0.25}\right)$ for $l'_1=l_2=0.001$, $M_3=0, \Hvec_{ext}=\left(0, 0, 0\right)$, $\eta_1=\eta_2=0.0005$ and $\xi=1,10$ respectively, consistent with \cite{Bisht2020}. In the $\Qvec$ plots, black bars represent the director field $\nvec$ and the color represents $\frac{1}{2}|\Qvec|^2=Q_{11}^2+Q_{12}^2$. In the $\Mvec$ plots, black arrows represent the 2D magnetization and the color bar represents $|\Mvec|^2=M_1^2+M_2^2$.}
\label{PRE1}
\end{figure}

We briefly comment on the solutions in \Cref{PRE1}. The boundary conditions for $\Mvec$ topologically require a $+1$-degree interior magnetic vortex or defect, labelled by the blue regions at the square centre. When $\xi=1$, the magnetic energy is comparable to the nematic energy and a strong nemato-magnetic coupling enforces the creation of two $+1/2$-defects in $\Qvec$, coerced by the $+1$-degree magnetic vortex. When $\xi=10$, the magnetic energy is much stronger and we get the same effect with weaker nemato-magnetic coupling, i.e., two interior non-orientable nematic defects near the square centre that are tailored by the $+1$-degree interior magnetic vortex.

\section{Numerical results - Impact of an external magnetic field and the stray field energy} \label{sec:numerical_experiments}
 This section is devoted to numerical results that demonstrate intriguing physics emerging from the incorporation of magnetic particles into nematic media in the presence of an external magnetic field and investigate the relevance  of the inclusion of the Gioia-James approximation of the stray field energy.
\subsection{Euler--Lagrange equation solutions}

We begin by studying the effect of an external magnetic field and the stray field energy on the ferronematic profiles in \Cref{PRE1}. As such, we consider $\eta_1=\eta_2=0.0001$ in \Cref{Fig-new1} and $\eta_1=\eta_2=0.0005$ in \Cref{Fig-new2}  to capture the limit $\eta_1,\eta_2\to 0$ and obtain solutions of the Euler--Lagrange system \eqref{EL}. We use the boundary and initial conditions described in \Cref{sec:test} with $M_3^b=0.25$.
\begin{figure}[ht!]
    \centering
    $\Qvec$$\left(M_3=0\right)$\hspace{1.79cm}$\Qvec$$\left(M_3\ne0\right)$\hspace{1.79cm}$\Mvec$$\left(M_3=0\right)$\hspace{1.79cm}$\Mvec$$\left(M_3\ne 0\right)$\\
    \vspace{1.5mm}
    \includegraphics[width=3.65cm]{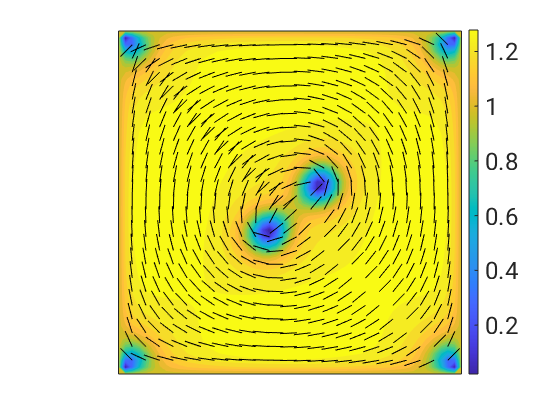}
    \includegraphics[width=3.65cm]{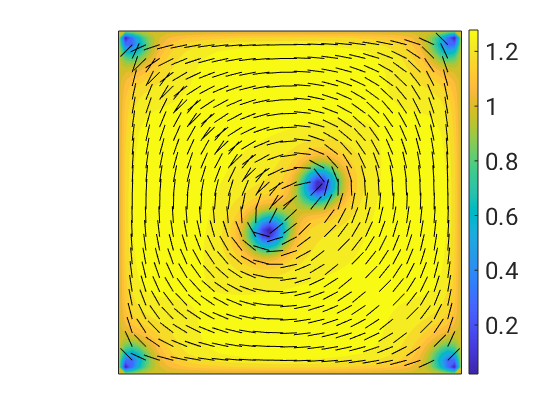}
    \includegraphics[width=3.65cm]{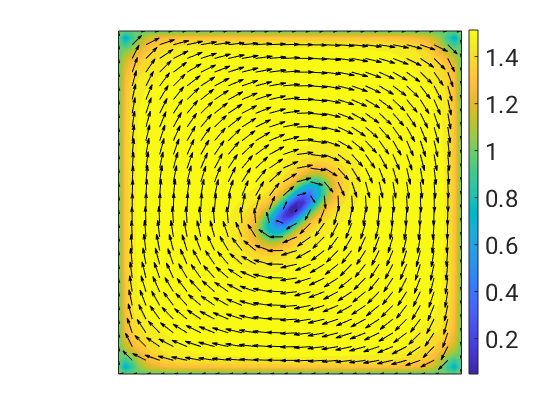}
    \includegraphics[width=3.65cm]{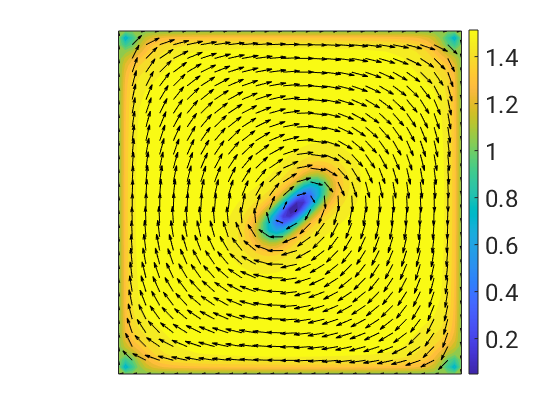}
    \\
    \vspace{2mm}
    \includegraphics[width=3.65cm]{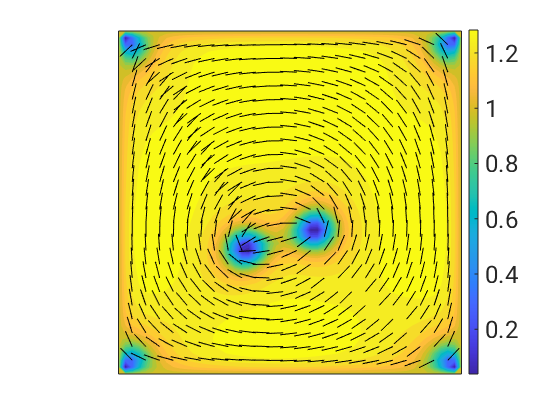}
    \includegraphics[width=3.65cm]{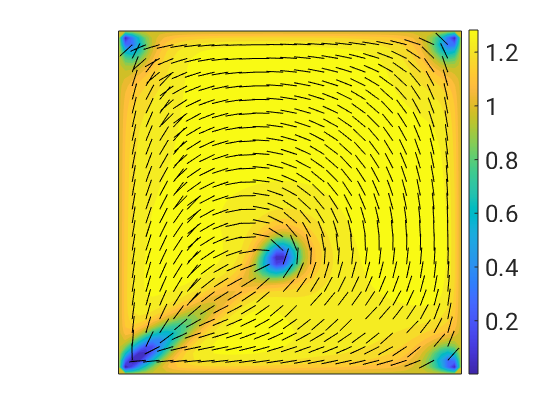}
    \includegraphics[width=3.65cm]{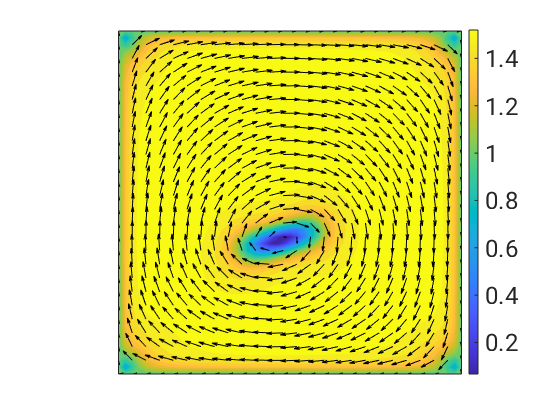}
    \includegraphics[width=3.65cm]{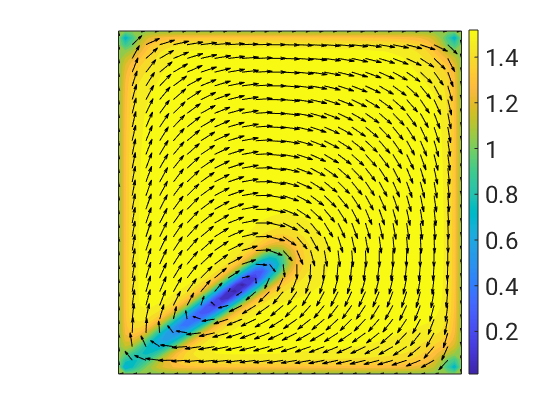}
    \\
    \vspace{2mm}
    \includegraphics[width=3.65cm]{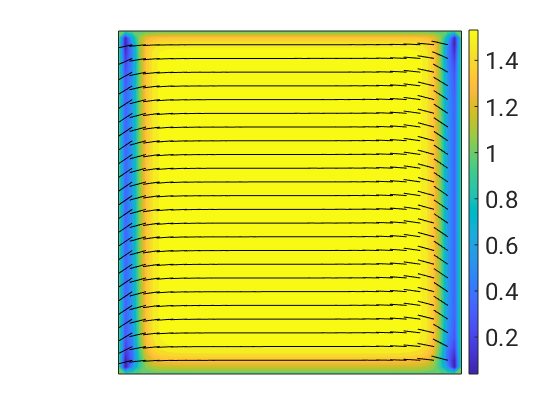}
    \includegraphics[width=3.65cm]{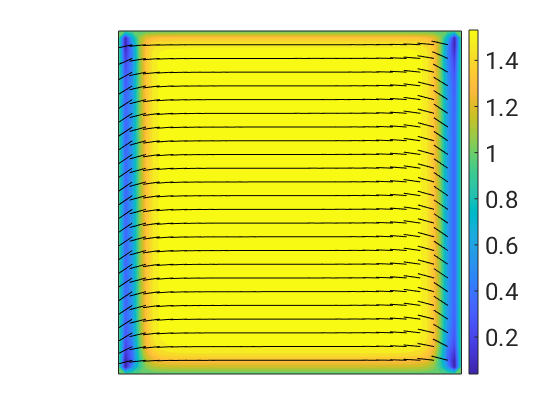}
    \includegraphics[width=3.65cm]{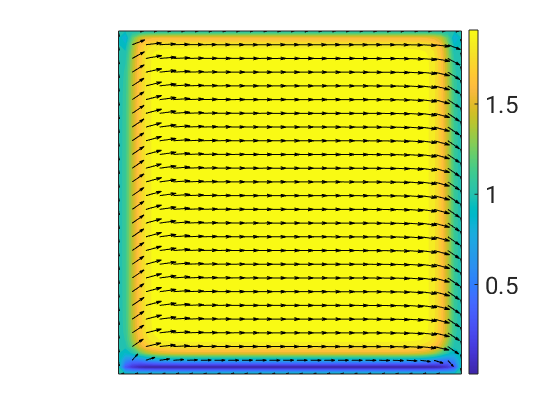}
    \includegraphics[width=3.65cm]{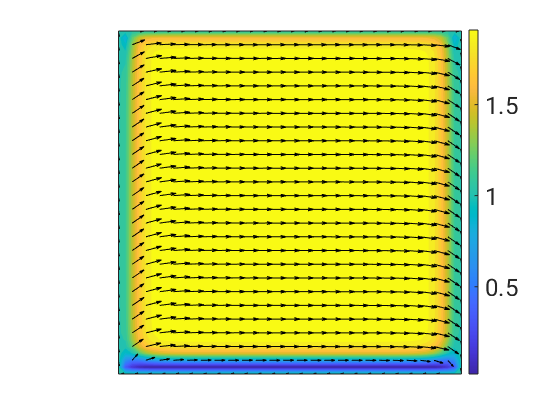}
    \caption{Nematic and magnetic configurations at different $\Hvec_{ext}$ strengths. $\Hvec_{ext}=\left(0, 0, 0\right), \left(0.00265, 0, 0\right), \left(0.25, 0, 0\right)$, varying vertically downwards. The first and second columns depict the director profiles, ignoring and including the stray field energy, respectively. The third and fourth columns show the magnetization profiles, ignoring and including the stray field energy, respectively. Parameter set: $l_1'=l_2=0.001$, $c_2=8$, $c_3=2$, $\xi=1$, $c_1=0.5$, $\eta_1=\eta_2=0.0001$. 
    }
    \label{Fig-new1}
\end{figure}

\begin{figure}[ht!]
    \centering
    $\Qvec$$\left(M_3=0\right)$\hspace{1.79cm}$\Qvec$$\left(M_3\ne0\right)$\hspace{1.79cm}$\Mvec$$\left(M_3=0\right)$\hspace{1.79cm}$\Mvec$$\left(M_3\ne 0\right)$\\
    \includegraphics[width=3.65cm]{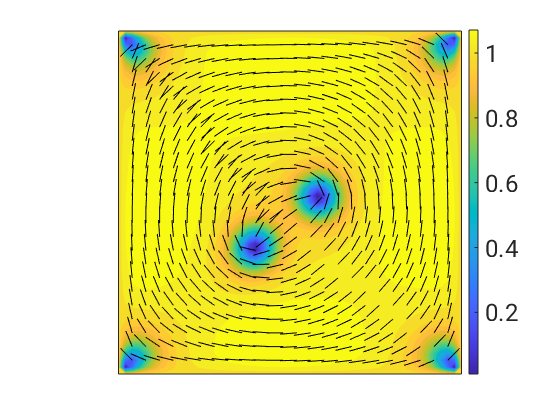}
    \includegraphics[width=3.65cm]{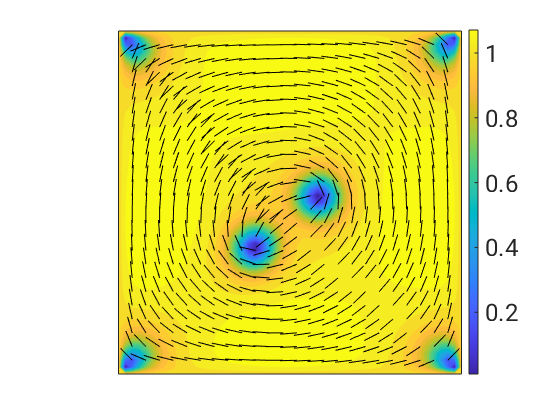}
    \includegraphics[width=3.65cm]{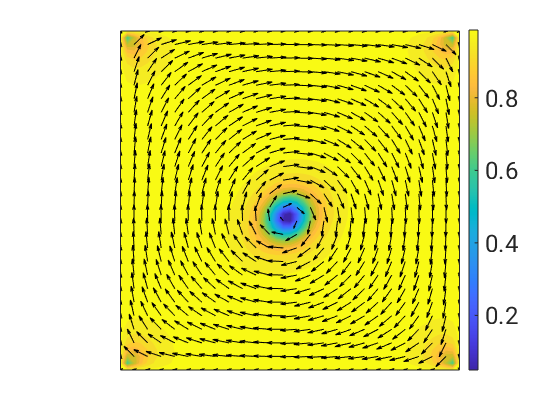}
    \includegraphics[width=3.65cm]{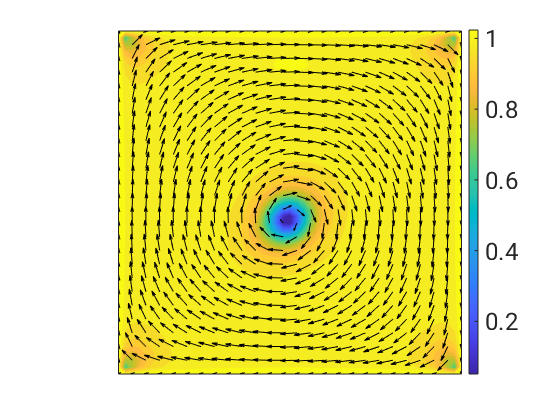}
    \\
    \vspace{2mm}
    \includegraphics[width=3.65cm]{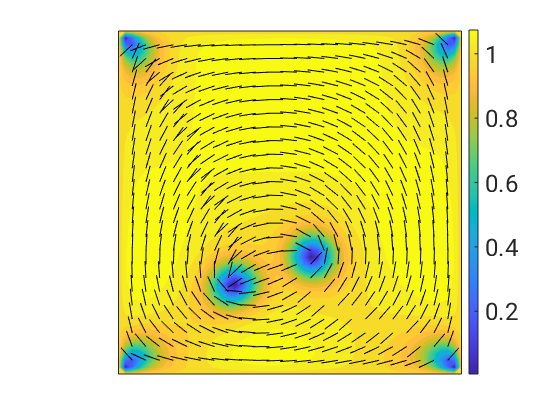}
    \includegraphics[width=3.65cm]{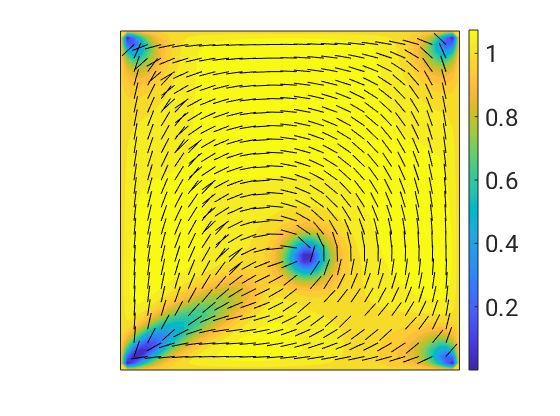}
    \includegraphics[width=3.65cm]{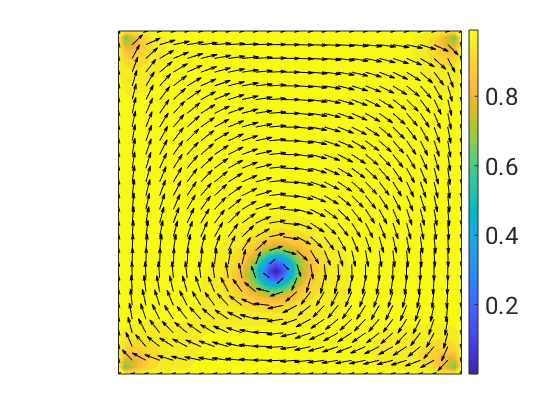}
    \includegraphics[width=3.65cm]{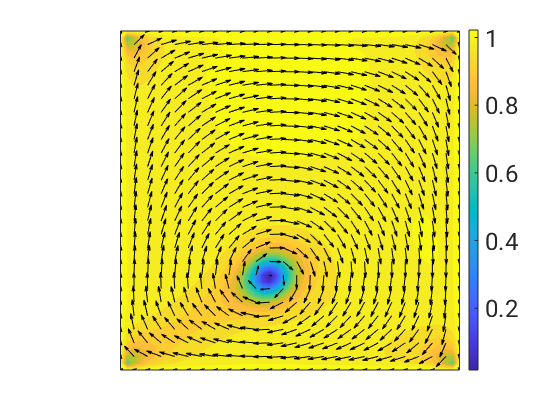}
    \\
    \vspace{2mm}
    \includegraphics[width=3.65cm]{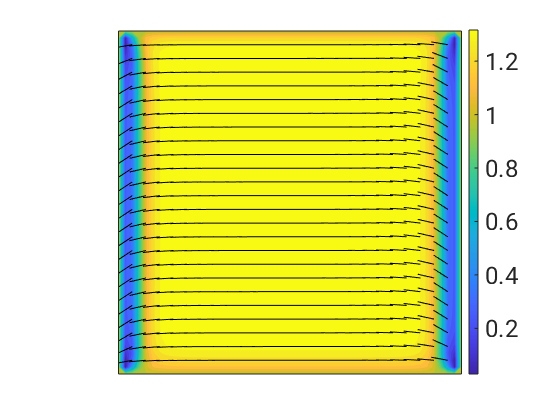}
    \includegraphics[width=3.65cm]{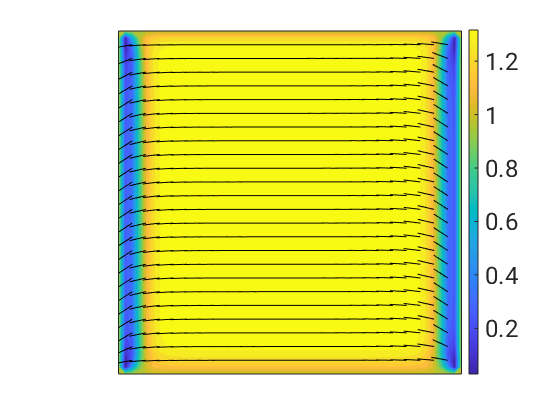}
    \includegraphics[width=3.65cm]{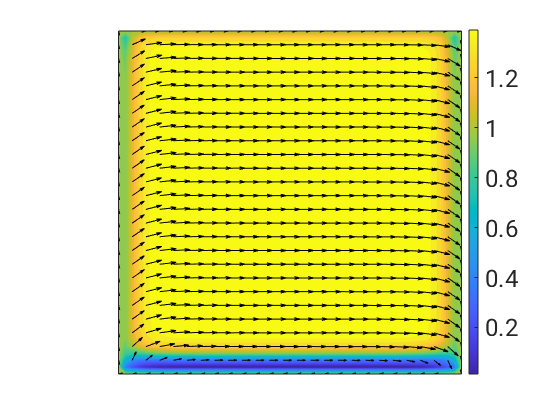}
    \includegraphics[width=3.65cm]{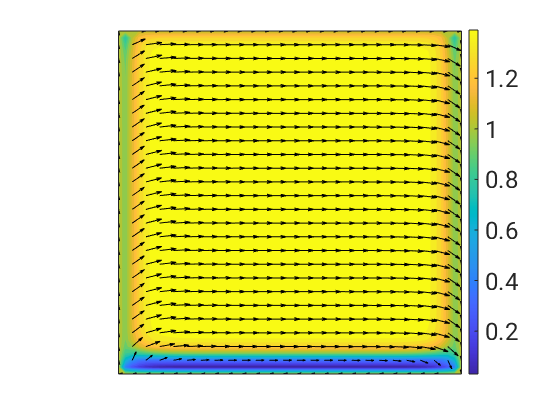}
    \caption{Nematic and magnetic configurations for different strengths of $\Hvec_{ext}$, both with and without a stray field energy i.e., $M_3\neq0$ and $M_3=0$. $\Hvec_{ext}=\left(0, 0, 0\right), \left(0.002825, 0, 0\right), \left(0.25, 0, 0\right)$, varying vertically downwards. 
    Parameter set: $l_1'=l_2=0.001$, $c_2=8$, $c_3=2$, $\xi=10, c_1=0.25$, $\eta_1=\eta_2=0.0005$.  
    }
    \label{Fig-new2}
\end{figure}

In \Cref{Fig-new1}, we focus on two cases with the following parameter values: $l'_1=l_2=0.001$, $c_2=8$, $c_3=2$, $\xi=1$, $c_1=0.5$. The cases are distinguished by zero and non-zero stray field energy, captured via the $M_3=0$ and $M_3\ne 0$ scenarios, respectively (the stray field energy is zero for vanishing $M_3$ throughout the domain). We perform numerical tests for various values of $\Hvec_{ext}$ and present results for three values that show interesting trends: $\Hvec_{ext}=\left(0,0,0\right)$, $ \left(0.00265, 0, 0\right)$, $ \left(0.25, 0, 0\right)$ (the external magnetic field is increasing as we move down the columns). When $\Hvec_{ext}=\left(0, 0, 0\right)$, we recover the results in \cite{Bisht2020}, i.e., the ferronematic profiles $\left(\Qvec^1_{D, 0.5}, \Mvec^1_{D, 0.5}\right)$ with $c_1=0.5$ as in \cite{Bisht2020}. For $\left(\Qvec^1_{D, 0.5}, \Mvec^1_{D, 0.5}\right)$, the nematic director profile has two interior defects of $+\frac{1}{2}$ charge while the magnetization profile shows an interior vortex maintaining the co-alignment with the nematic director. These interior defects and vortices are characterized by $|\Qvec|=0$ and $|\Mvec|=0$, respectively. We refer to \cite{ball_notes} for a detailed study on nematic defects. 
When we activate the external magnetic field with $\Hvec_{ext}=\left(0.00265, 0, 0\right)$, we observe two scenarios depending on whether we set $M_3 \equiv 0$ (including setting $M_3=0$ in the Dirichlet boundary condition and initial condition, or solve for $M_3$ with non-zero boundary conditions and initial conditions for $M_3$.  
The second case with $M_3 \neq 0$ describes the effects of the stray field energy. In the first and third columns of \Cref{Fig-new1}, the effects of $\Hvec_{ext}$ are apparent - the external magnetic field is in the $x$-direction and re-orients the nematic and magnetic defects in the $x$-direction. In fact, the nematic director and the magnetization vector co-align with $\Hvec_{ext}$ by means of the nematic coupling with the external magnetic field (captured by $c_2$) and the Zeeman energy (captured by $c_3$) respectively. For $\Hvec_{ext}=\left(0.25, 0, 0\right)$, the nematic director and the magnetization vector are completely aligned with $\Hvec_{ext}$. When $M_3 \neq 0$ (including non-zero boundary conditions  
and initial conditions for $M_3$), the stray field energy can play a role. This is evident from the second row of the second and fourth columns. For the $\Qvec$-profile, one nematic defect is pushed away from the square centre and the second defect is almost pinned at the bottom left square vertex. The magnetic vortex clearly migrates towards the bottom left square vertex and we get a dispersed magnetic vortex near the bottom left square vertex.  For $\Hvec_{ext}=\left(0.25, 0, 0\right)$, the profiles with $M_3=0$ and $M_3 \neq 0$ are almost indistinguishable and we get perfectly co-aligned nematic directors and $\Mvec$-profiles, tailored by the dominant external magnetic field. These preliminary observations suggest that $\Hvec_{ext}$ clearly reorients the nematic director and the magnetization vector and the stray field tends to expel nematic and magnetic defects towards the boundaries.

In \Cref{Fig-new2}, we examine the impact of a stray field for the following parameter values: $l'_1=l_2=0.001$, $c_2=8$, $c_3=2$, $\xi=10$, $c_1=0.25$. 
Again we study various values for $\Hvec_{ext}$ and present three cases that appear interesting: $\Hvec_{ext}=\left(0,0,0\right)$, $\left(0.002825, 0, 0\right)$, $\left(0.25, 0, 0\right)$. When $\Hvec_{ext}=\left(0, 0, 0\right)$, we find ferronematic profiles $\left(\Qvec^{10}_{D, 0.25}, \Mvec^{10}_{D, 0.25}\right)$ with $c_1=0.25$, and there are negligible differences between the $M_3=0$ and $M_3 \neq 0$ cases. In the case of $\left(\Qvec^{10}_{D, 0.25}, \Mvec^{10}_{D, 0.25}\right)$, the director profile has two interior defects of $+\frac{1}{2}$ charge while the magnetization profile shows an interior vortex maintaining the co-alignment. When we activate the external magnetic field with $\Hvec_{ext}=\left(0.002825, 0, 0\right)$, we observe two scenarios depending on whether $M_3=0$ or $M_3\ne 0$. When $M_3=0$, we observe the realignment of the nematic and magnetic defects in the $x$-direction, induced by the coupling with the external field and all defects tend to move towards the square edges. For $M_3 \neq 0$, the effect is more pronounced in the sense that one nematic defect gets pinned at the bottom left square vertex. We note that the magnetic defects are more influenced by the stray field (captured by $M_3 \neq 0$) for $\xi=1$, than for $\xi=10$. This is readily explained by the much increased dominance of the magnetic energy for $\xi=10$ and the magnetic energy favours the creation of an interior $+1$-vortex.  We observe (almost) complete alignment of the nematic director and the magnetization vector in the $x$-direction for $\Hvec_{ext}=\left(0.25, 0, 0\right)$. We note that the differences between the $M_3=0$ and $M_3 \neq 0$ solutions are not appreciable for $\Hvec_{ext} = (0,0,0)$ and large $|\Hvec_{ext}|$. We speculate that $M_3$ remains small when $\Hvec_{ext} = (0,0,0)$ even when we include the non-zero boundary and initial conditions for $M_3$ and solve for $\Mvec = (M_1, M_2, M_3)$. For large $|\Hvec_{ext}|$, the Zeeman energy dominates the stray field energy and hence, there are negligible differences between the $M_3=0$ and $M_3 \neq 0$ scenarios. 

\begin{remark}
We note that the influence of the stray field on defect patterns in the above ferronematic profiles can also be visible with $M_3^b=0$ under suitable external magnetic fields $\Hvec_{ext}$.
\end{remark}

\subsection{Gradient flow solutions}
The next set of numerical results concern the nematic and magnetic profiles obtained by solving the gradient flow equations \eqref{eq:26}, supplemented with the boundary conditions \eqref{eq:27'} and the initial conditions \eqref{eq:28'} for degrees $k=1, 2$. As such, we take $\eta_1=\eta_2=1$ in this subsection and study the time-relaxed solutions at time $t=0.01$. 
\begin{figure}[ht!]
    \centering
    $\Qvec$$\left(M_3=0\right)$\hspace{1.79cm}$\Qvec$$\left(M_3\ne0\right)$\hspace{1.79cm}$\Mvec$$\left(M_3=0\right)$\hspace{1.79cm}$\Mvec$$\left(M_3\ne 0\right)$\\
    \includegraphics[width=3.65cm]{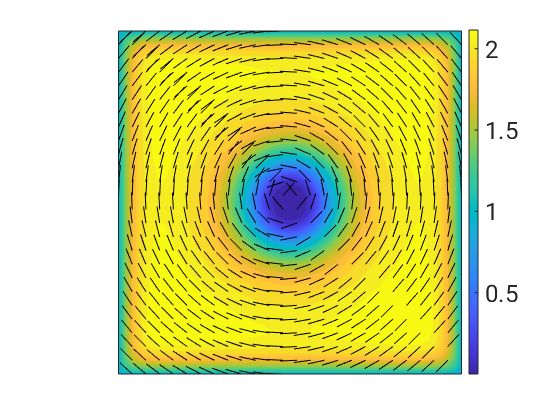}
    \includegraphics[width=3.65cm]{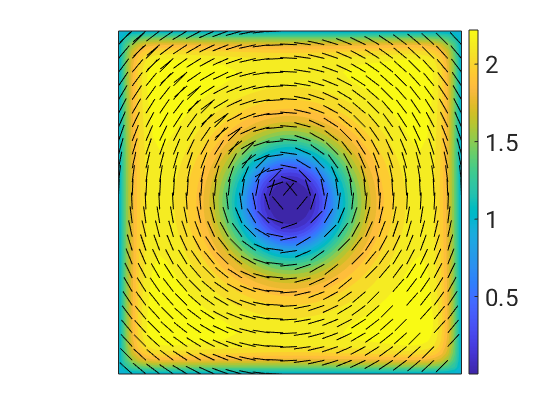}
    \includegraphics[width=3.65cm]{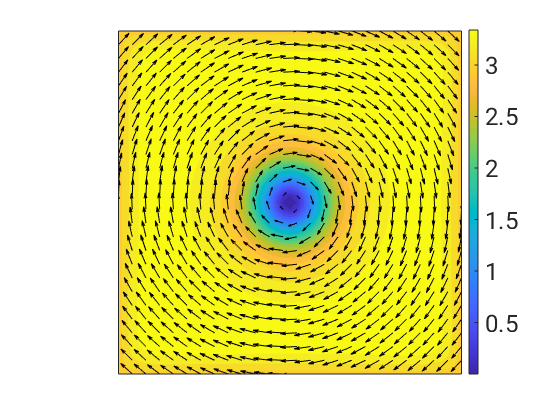}
    \includegraphics[width=3.65cm]{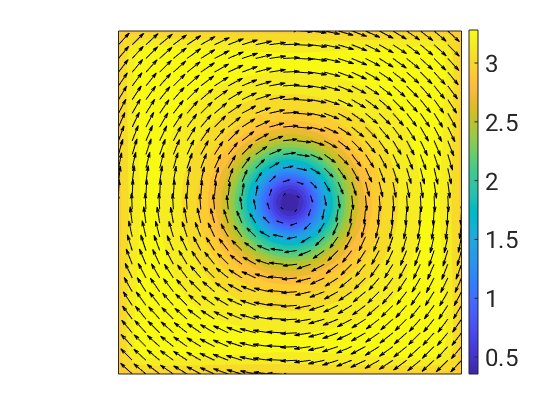}
    \\
    \vspace{2mm}
    \includegraphics[width=3.65cm]{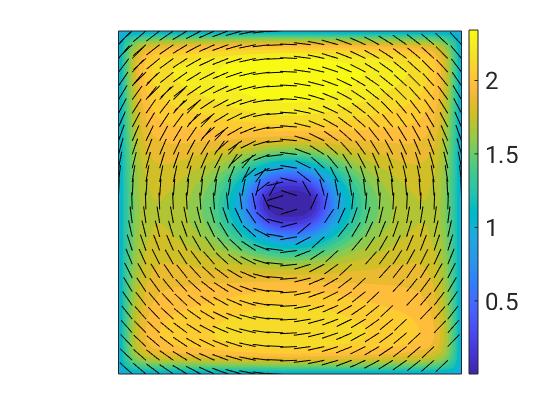}
    \includegraphics[width=3.65cm]{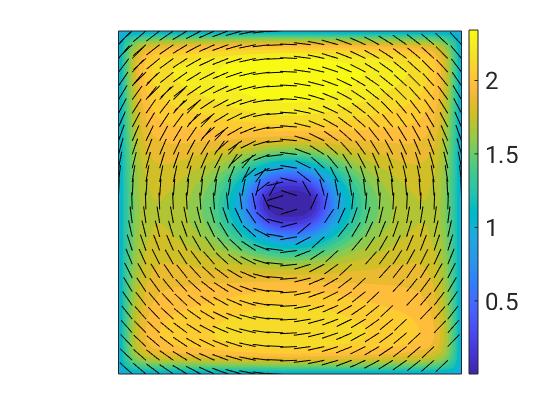}
    \includegraphics[width=3.65cm]{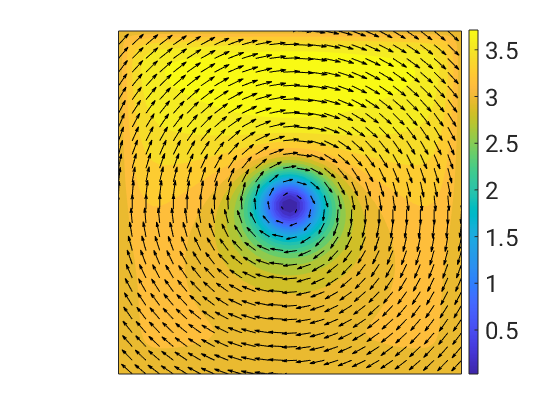}
    \includegraphics[width=3.65cm]{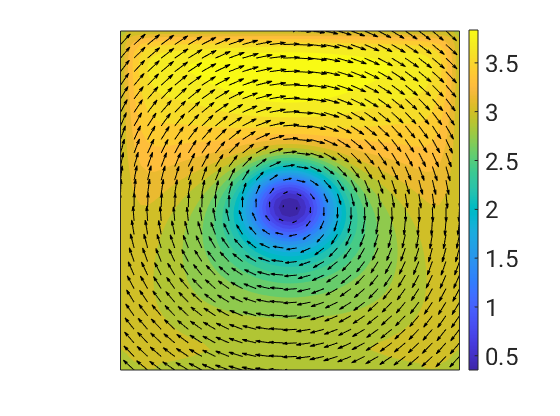}
    \\
    \vspace{2mm}
    \includegraphics[width=3.65cm]{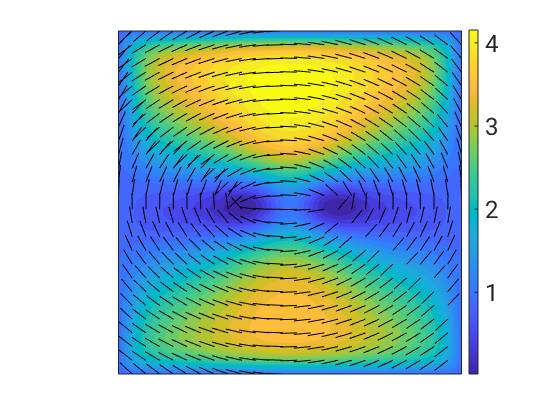}
    \includegraphics[width=3.65cm]{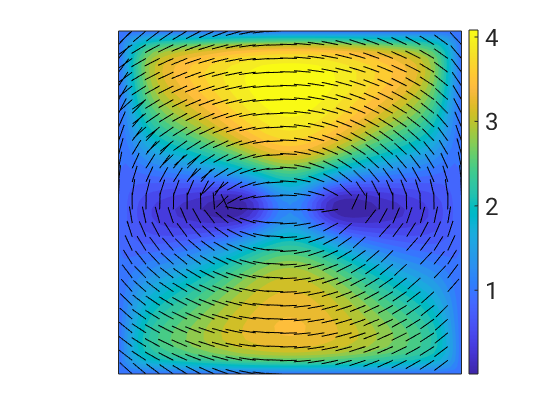}
    \includegraphics[width=3.65cm]{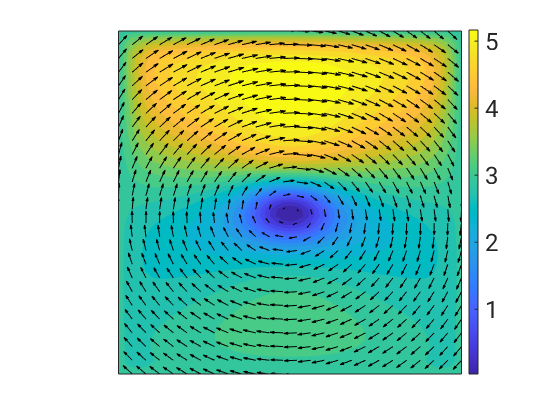}
    \includegraphics[width=3.65cm]{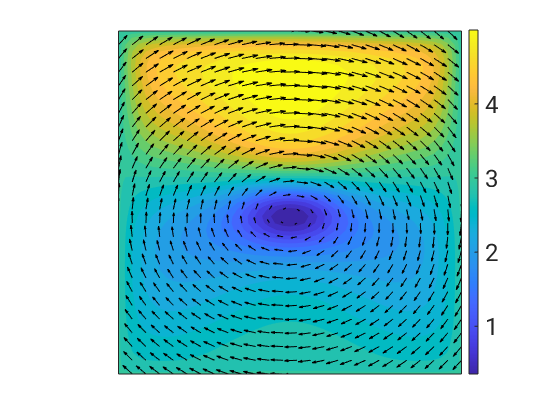}
    \\
    \vspace{2mm}
    \includegraphics[width=3.65cm]{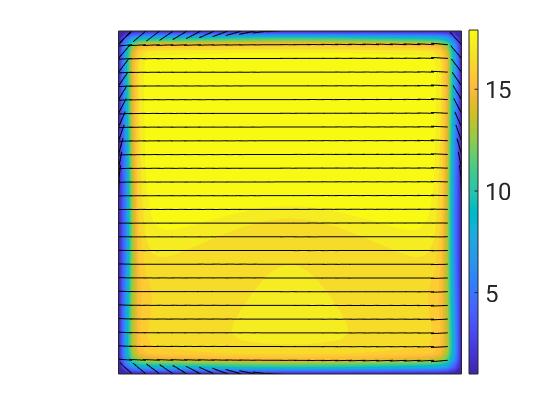}
    \includegraphics[width=3.65cm]{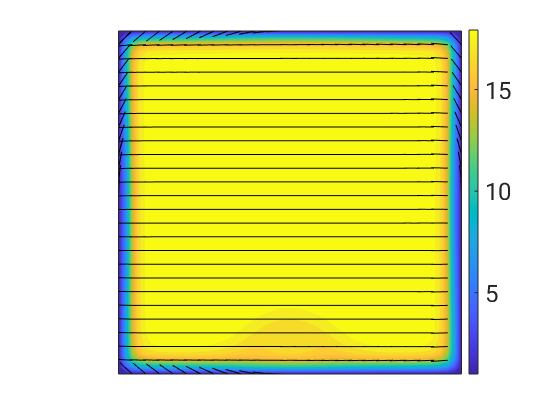}
    \includegraphics[width=3.65cm]{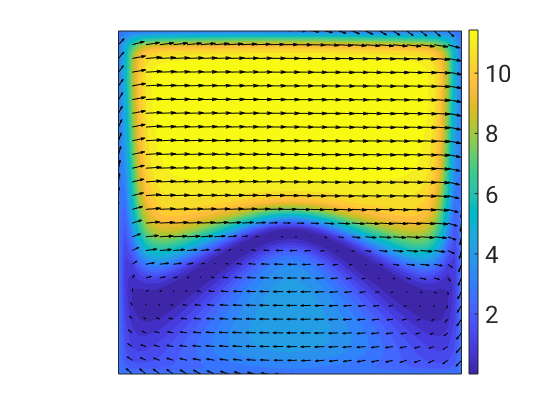}
    \includegraphics[width=3.65cm]{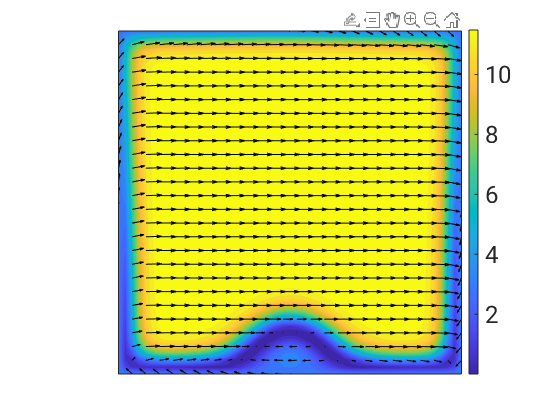}
    \caption{Nematic and magnetic configurations for different choices of $\Hvec_{ext}$, both with and without a stray field energy, i.e., $M_3\neq0$ and $M_3=0$. $\Hvec_{ext}=\left(0, 0, 0\right), \left(0.4, 0, 0\right), \left(0.875, 0, 0\right), \left(4, 0, 0\right)$, varying vertically downwards. 
    Parameter set: $l'_1=l_2=0.01$, $c_1=2$, $c_2=8$, $c_3=2$, $\xi=1$, $\eta_1=\eta_2=1$; $k=1$ and at time $t=0.01$.}
    \label{Fig2}
\end{figure}
In \Cref{Fig2}, we examine two cases with the following parameter values: $l'_1=l_2=0.01$, $c_1=2$, $c_2=8$, $c_3=2$, $\xi=1$ and $k=1$. The two cases differ on whether we include or exclude the $M_3$ component of the magnetization vector in the governing equations, boundary and initial conditions, recalling that the stray field energy is necessarily zero if $M_3=0$ everywhere. 
We work with four different constant external magnetic fields  $\textbf{H}_{ext}$: $\left(0, 0, 0\right), \left(0.4, 0, 0\right), \left(0.875, 0, 0\right)$ and $\left(4, 0, 0\right)$ ($|\Hvec_{ext}|$ is increasing as we move down the columns), with a view to study how external magnetic fields reorient ferronematic profiles and their defects.

In the first column of \Cref{Fig2}, we observe the reconstruction of the director field $\nvec$ as the external magnetic field $\Hvec_{ext}$ increases, under the assumption that $M_3=0$ everywhere. We notice the following trends: (i) when $\Hvec_{ext} = \left(0, 0, 0\right)$, the nematic director exhibits a clear $+1$ interior vortex or two $+1/2$ defects pinned together (distinguished by a blue region of low or almost zero $Q_{11}$ and $Q_{12}$) at the square centre, as dictated by the topology of the boundary datum; (ii) when $\Hvec_{ext} = \left(0.4, 0, 0\right)$, the core of the interior defect is elongated but the nematic director profile is largely unaltered; (iii) when $\Hvec_{ext}=\left(0.875, 0, 0\right)$, the nematic director exhibits significant reorientation in the $x$-direction and the $+1$-defect visibly splits into two $+1/2$ defects that are pushed towards the square edges (or the two almost pinned $+1/2$ defects move away from each other and move towards the square edges) and (iv) when $\Hvec_{ext}=\left(4, 0, 0\right)$, the nematic director is completely co-aligned with $\Hvec_{ext}$ and the defects are expelled from the interior entirely. These observations indicate that the external magnetic field reorients the nematic directors, as expected, and pushes interior defects towards the boundaries to facilitate the reorientation or realignment of the nematic directors with $\Hvec_{ext}$.

In the second column of \Cref{Fig2}, we observe the overall trends are similar to the case with $M_3=0$, but the non-zero $M_3$ and consequently, the non-zero stray field energy tend to enlarge the defect core sizes, split the defect core and push them towards the square edges more strongly compared to the $M_3=0$ case. As before, the differences between the $M_3=0$ and $M_3 \neq 0$ cases are negligible when $\Hvec_{ext} = (4,0,0)$. 
In the third and fourth columns of \Cref{Fig2}, we numerically compute $\Mvec$-profiles for the four different choices of $\Hvec_{ext}$, considering the two cases as above. The $\Mvec$-profile always exhibits a $+1$-degree interior vortex that is localised near the square centre, and the $\Mvec$-profile is roughly co-aligned with the nematic director away from the defects, since $c_1$ is positive. There is no defect splitting since non-orientable or fractional defects are not allowed in a vector model.  
On comparing the profiles with $M_3=0$ and $M_3 \neq 0$ respectively in each row (with the same value of $\Hvec_{ext}$), we deduce that the stray field energy enlarges or elongates the size of the central magnetic vortex and effectively creates a band of low order (labelled by low values of $|\Mvec|^2$) near the square centre as we move from the first to the third row. There is quite a jump between $\Hvec_{ext}=(0.875, 0,0)$ and $\Hvec_{ext}=(4,0,0)$ and we speculate that the band of low order migrates to the bottom edge as $|\Hvec_{ext}|$ increases and is eventually almost expelled from the interior. The expulsion is more pronounced when $M_3 \neq 0$, when we compare the third and fourth columns of the last row of \Cref{Fig2}. We notice a hump in $|\Mvec|$ near the bottom square edge in the last row; we speculate that this hump comes from memory of the magnetic vortex and the hump would disappear for larger values of $|\Hvec_{ext}|$; the hump is smaller and less pronounced when the stray field energy is included by means of the $M_3 \neq 0$ case in the fourth column of \Cref{Fig2}.  

In \Cref{Fig3}, we numerically compute the nematic and magnetization profiles for degree $k=2$ boundary conditions, with the following parameter values: $l'_1=l_2=0.01$, $c_1=2$, $c_2=8$, $c_3=2$ and $\xi=1$. The figures are presented in the same manner as in \Cref{Fig2}.  In the first column of \Cref{Fig3}, there are four distinct nematic profiles (ignoring the stray field energy): when (i) $\Hvec_{ext} = \left(0, 0, 0\right)$, the nematic director profile exhibits four non-orientable interior defects almost pinned together near the square centre and distinguished by a core of low nematic order of $|\Qvec|$; (ii) when $\Hvec_{ext} = \left(0.4, 0, 0\right)$, the defects move further apart and there is a splitting effect; (iii) when $\Hvec_{ext}=\left(0.875, 0, 0\right)$, we see two clear pairs of $+1/2$ defects and four bands of low order connecting each fractional defect to a square vertex; and (iv) when $\Hvec_{ext}=\left(4, 0, 0\right)$, the nematic directors are almost fully co-aligned with $\Hvec_{ext}$. We see some slight humps of relatively low order near the left and right edges of the fourth row and first column; these are signatures of the expelled interior nematic defects. In the second column of \Cref{Fig3}, the overall trends are the same but the nematic defect cores are larger when $M_3 \neq 0$ for $\Hvec_{ext} = (0,0,0)$ and $\Hvec_{ext} = (0.4,0,0)$ respectively. There is stronger repulsion between the four $+1/2$ defects in the third row, for $M_3 \neq 0$ as compared to the $M_3=0$ case. As before, there are no appreciable differences between the nematic director profiles for the first and second columns of the fourth row, with $\Hvec_{ext}=(4,0,0)$. We note that the total topological degree of the interior defects (in this case four $+1/2$ defects) must match the topological degree, $k=2$, of the boundary datum. 

\begin{figure}[ht!]
    \centering
    $\Qvec$$\left(M_3=0\right)$\hspace{1.79cm}$\Qvec$$\left(M_3\ne0\right)$\hspace{1.79cm}$\Mvec$$\left(M_3=0\right)$\hspace{1.79cm}$\Mvec$$\left(M_3\ne 0\right)$\\
    \includegraphics[width=3.65cm]{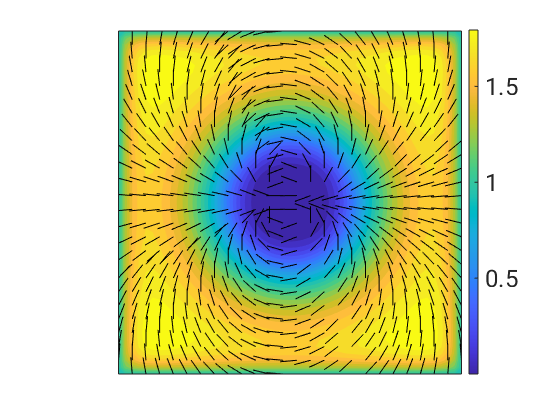}
    \includegraphics[width=3.65cm]{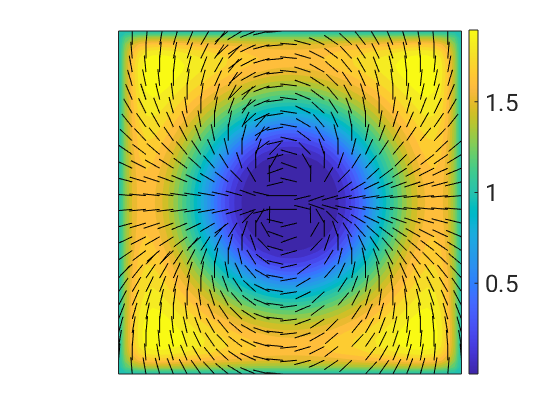}
    \includegraphics[width=3.65cm]{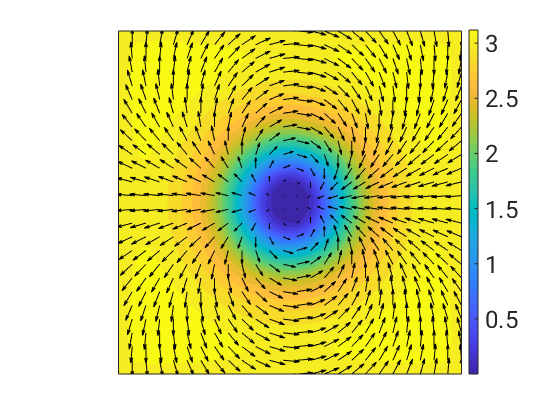}
    \includegraphics[width=3.65cm]{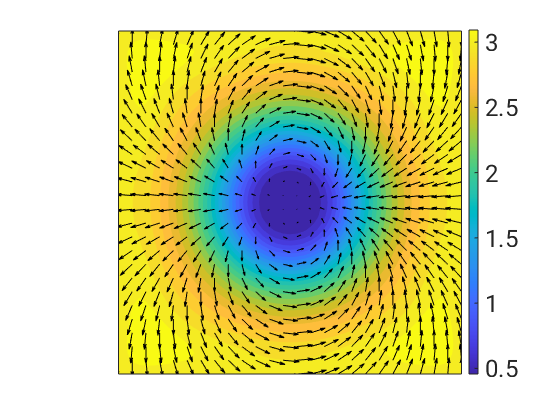}
    \\
    \vspace{2mm}
    \includegraphics[width=3.65cm]{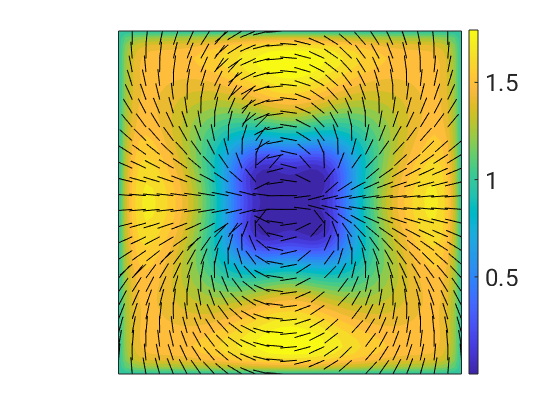}
    \includegraphics[width=3.65cm]{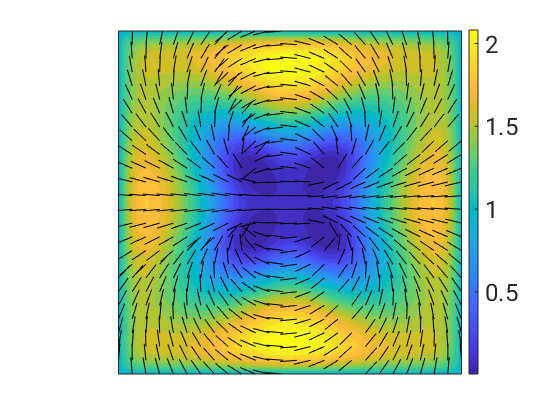}
    \includegraphics[width=3.65cm]{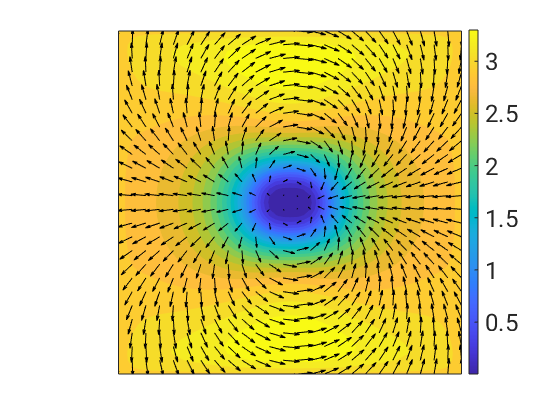}
    \includegraphics[width=3.65cm]{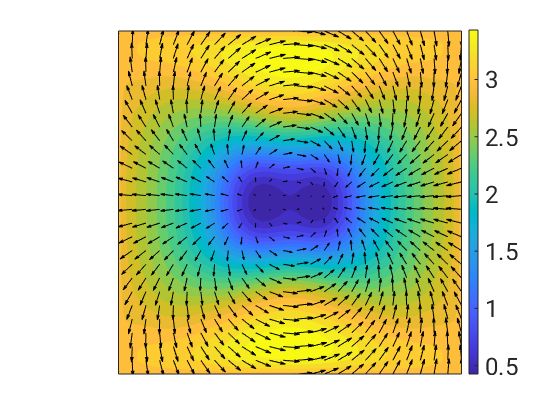}
    \\
    \vspace{2mm}
    \includegraphics[width=3.65cm]{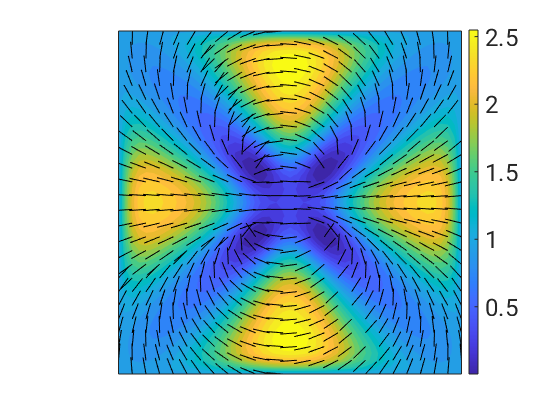}
    \includegraphics[width=3.65cm]{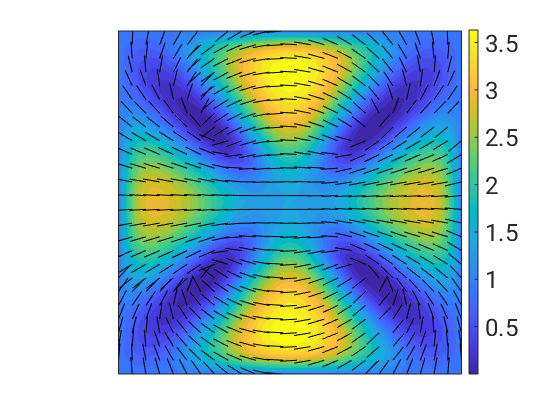}
    \includegraphics[width=3.65cm]{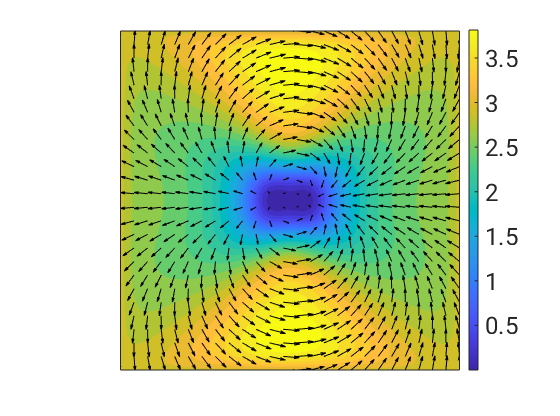}
    \includegraphics[width=3.65cm]{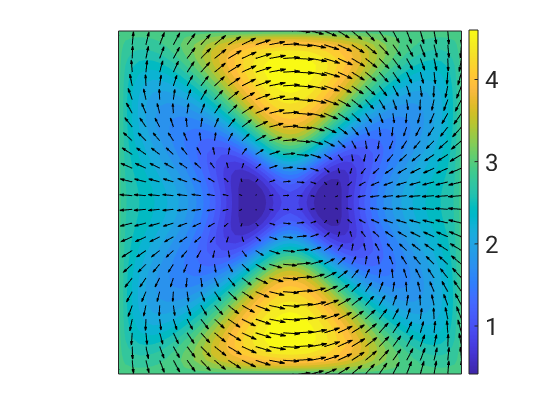}
    \\
    \vspace{2mm}
    \includegraphics[width=3.65cm]{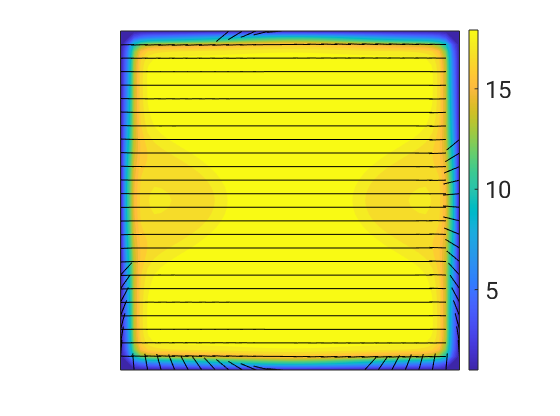}
    \includegraphics[width=3.65cm]{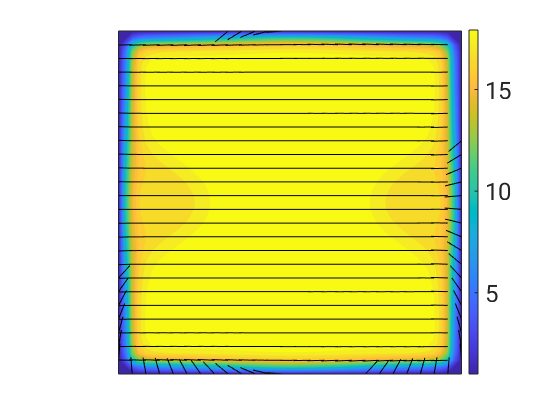}
    \includegraphics[width=3.65cm]{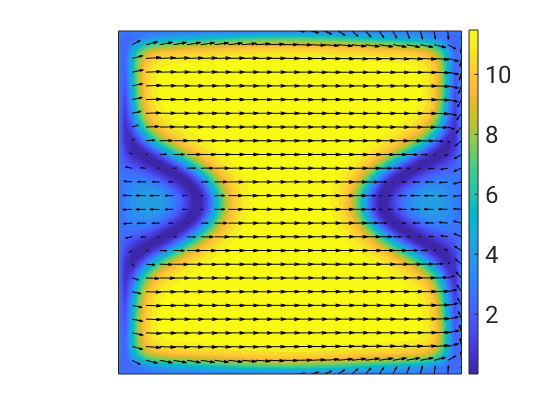}
    \includegraphics[width=3.65cm]{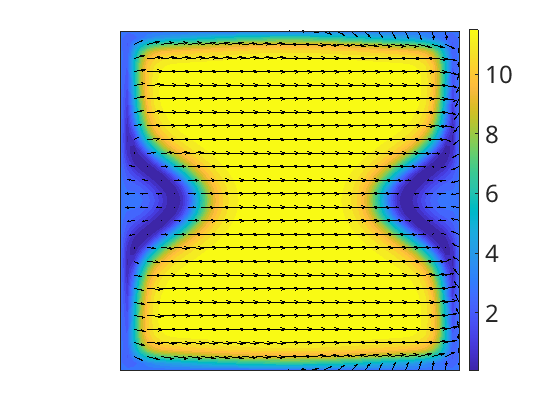}
    \caption{Nematic and magnetic configurations for different choices of $\Hvec_{ext}$, both with and without a stray field energy, i.e., $M_3\neq0$ and $M_3=0$. $\Hvec_{ext}=\left(0, 0, 0\right), \left(0.4, 0, 0\right), \left(0.875, 0, 0\right), \left(4, 0, 0\right)$, varying vertically downwards. Parameter set: $l'_1=l_2=0.01$, $c_1=2$, $c_2=8$, $c_3=2$, $\xi=1$, $\eta_1=\eta_2=1$; $k=2$ and at time $t=0.01$.}
    \label{Fig3}
\end{figure}

In the third and fourth columns of \Cref{Fig3}, we numerically compute the $\Mvec$-profiles for the two cases (without and with stray field energy). We observe four distinct types of $\Mvec$-profiles: (i) when $\Hvec_{ext} = \left(0, 0, 0\right)$, there is a distinct interior central magnetc vortex of degree $+2$, consistent with the imposed boundary condition and initial condition, and the core is larger for $M_3 \neq 0$ than the $M_3=0$ case. (ii) When $\Hvec_{ext} = \left(0.4, 0, 0\right)$, the vortex is smeared out for $M_3=0$ but splits into two central interior vortices for $M_3 \neq 0$; (iii) when $\Hvec_{ext} = \left(0.875, 0, 0\right)$, there is a general tendency for $\Mvec$ to realign with $\Hvec_{ext}$ and we observe bands of low order connecting the magnetic vortices to the square vertices when $M_3 \neq 0$. The magnetic vortices tend to move towards the square edges as $|\Hvec_{ext}|$ increases and this tendency is more pronounced when $M_3 \neq 0$. (iv) When $\Hvec_{ext} = \left(4, 0, 0\right)$, $\Mvec$ is almost perfectly aligned with $\Hvec_{ext}$ and we observe the two signature humps near the left and right square edges. These are the memories of the expelled magnetic vortices. As before, we do not get fractional magnetic vortices and consequently, we get splitting into two magnetic vortices (of degree $+1$) as opposed to four nematic vortices (of degree $+1/2$). More precisely, if we interpret $\Qvec=(Q_{11}, Q_{12})$ as a two-dimensional vector, then the degree of the nematic vortex is half the degree of the $\Qvec$-vector.

Thus, we deduce that stray field energies enlarge, split and repel interior nematic and magnetic defects towards the square edges and therefore co-operate with the external magnetic field to promote the co-alignment of $\nvec$ and $\Mvec$ with $\Hvec_{ext}$. There are parallels with the experimental evidence in \cite{Shui}, which reports the migration of nematic defects and magnetic vortices in the BF/BuOH ferronematic phase, under an activated magnetic field, in some circumstances. Moreover, the shape and size of the core of magnetic vortices are observed to depend on the stray field, which agrees with the report  for a pure ferromagnetic substance as mentioned in \cite{wachowiak2002direct}. 

\subsubsection{Role of the parameter \texorpdfstring{$\xi$}{xi}}\label{role of parameter chi}

We study the ferronematic solutions for four different values of the parameter $\xi$: $\xi=0.025, 0.25, 0.5, 2.5$ with fixed $\Hvec_{ext}=\left(0.4, 0, 0\right)$ in \Cref{Fig4}. Also, we fix the parameters $l_1'=l_2=0.01$, $c_1=2$, $c_2=8$, $c_3=2$ and choose $k=2$, for comparison with \Cref{Fig3}. 
\begin{figure}[ht!]
    \centering
     $\Qvec$$\left(M_3=0\right)$\hspace{1.79cm}$\Qvec$$\left(M_3\ne0\right)$\hspace{1.79cm}$\Mvec$$\left(M_3=0\right)$\hspace{1.79cm}$\Mvec$$\left(M_3\ne 0\right)$\\
    \includegraphics[width=3.65cm]{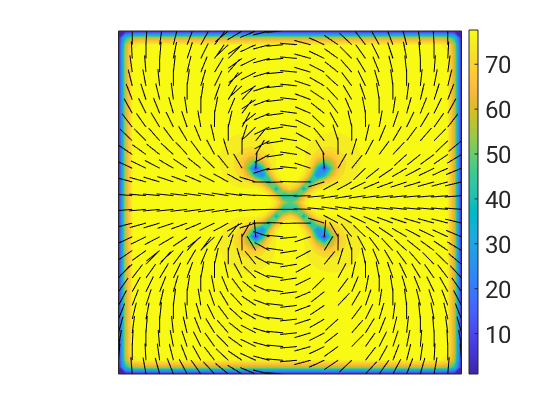}
    \includegraphics[width=3.65cm]{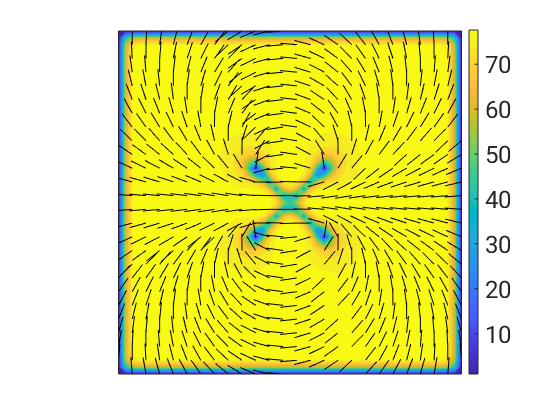}
    \includegraphics[width=3.65cm]{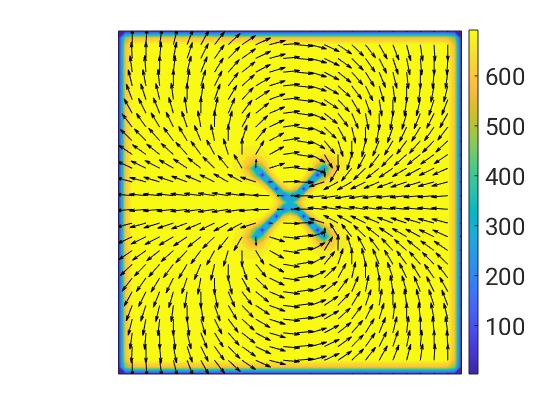}
    \includegraphics[width=3.65cm]{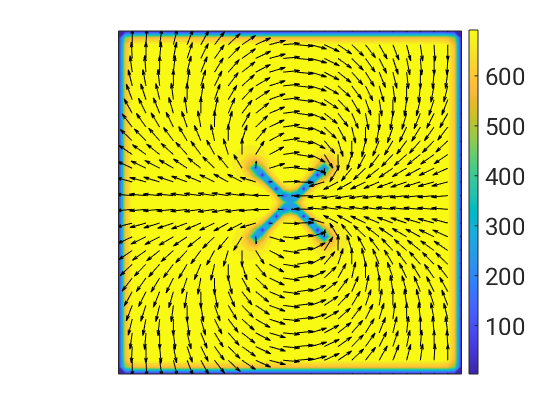}
    \\
    \vspace{2mm}
    \includegraphics[width=3.65cm]{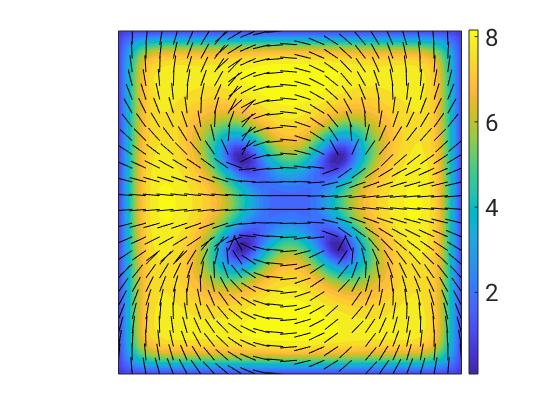}
    \includegraphics[width=3.65cm]{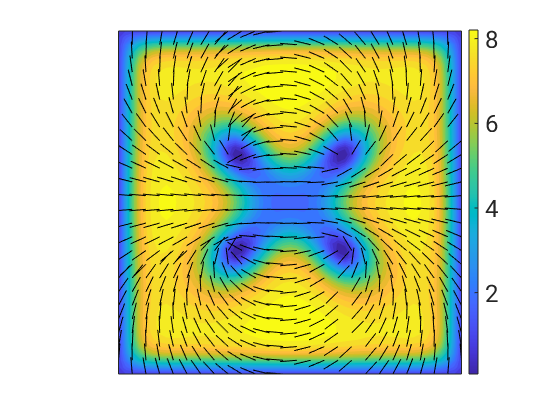}
    \includegraphics[width=3.65cm]{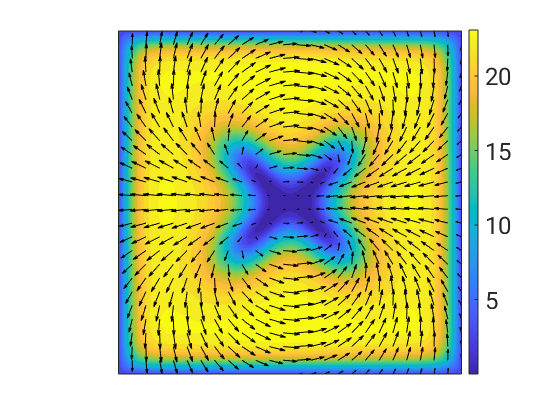}
    \includegraphics[width=3.65cm]{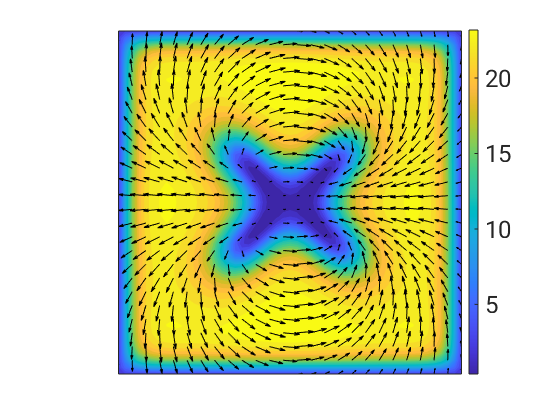}
    \\
    \vspace{2mm}
    \includegraphics[width=3.65cm]{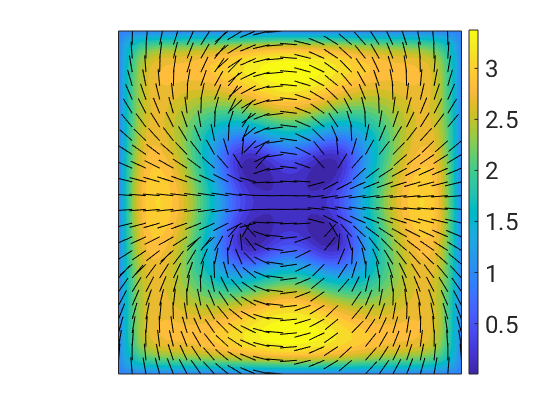}
    \includegraphics[width=3.65cm]{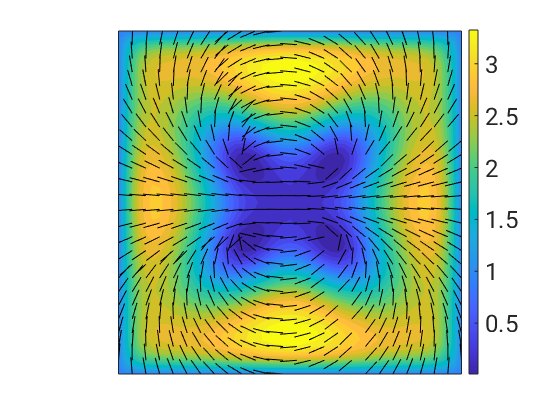}
    \includegraphics[width=3.65cm]{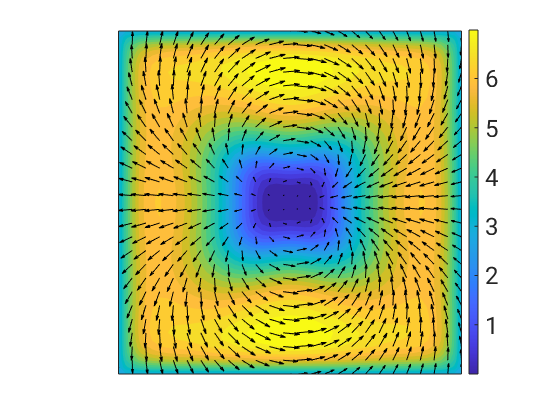}
    \includegraphics[width=3.65cm]{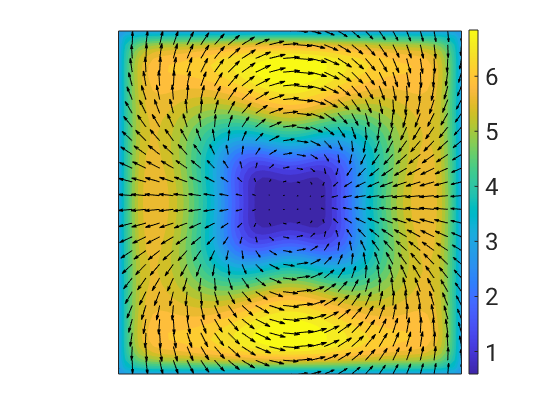}
    \\
    \vspace{2mm}
    \includegraphics[width=3.65cm]{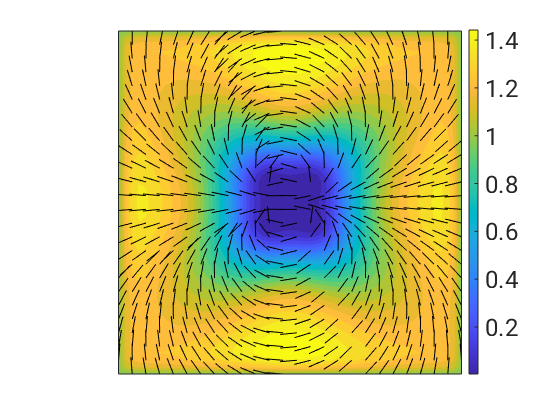}
    \includegraphics[width=3.65cm]{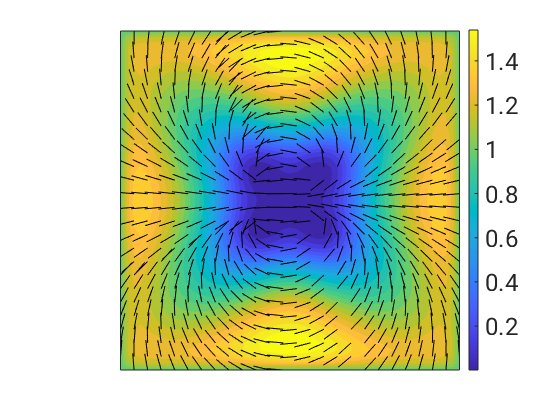}
    \includegraphics[width=3.65cm]{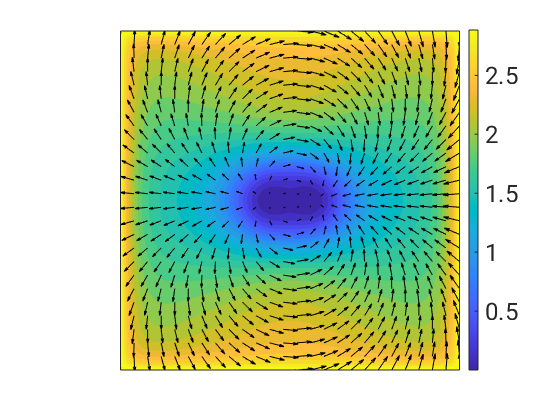}
    \includegraphics[width=3.65cm]{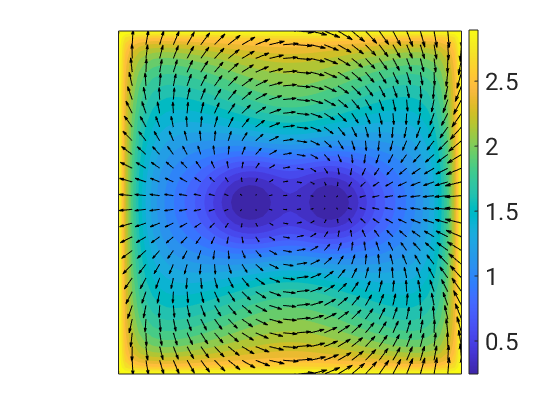}
    \caption{Nematic and magnetic configurations with $\Hvec_{ext}=\left(0.4, 0, 0\right)$ and different $\xi=0.025, 0.25, 0.5, 2.5$, varying vertically downwards, both with and without a stray field energy i.e., $M_3\neq0$ and $M_3=0$. 
    Parameter set: $l'_1=l_2=0.01$, $c_1=2$, $c_2=8$, $c_3=2$, $\eta_1=\eta_2=1$; $k=2$ and at time $t=0.01$.}
    \label{Fig4}
\end{figure}
The parameter $\xi$ increases as we move vertically downwards between the rows. The external magnetic field is relatively weak and hence, the co-alignment effects with the external magnetic field are relatively weak for the nematic director and magnetization profiles respectively. We summarise the main trends for both the $M_3=0$ and $M_3 \neq 0$ cases. For small $\xi$, the magnetic energies are weak compared to their nematic counterparts. We see a cross of low order (relatively small values of $|\Mvec|$) in the $\Mvec$-profile, tailored by the four fractional point defects in the nematic director, localised near the centre. We notice the co-alignment of $\nvec$ and $\Mvec$ prominently, since the coupling parameter $c_1= 2 = 80 * 0.025 = 80 \xi$ in the first row and $c_1 = 2 = 8 * 0.25= 8 \xi$ in the second row. As $\xi$ increases, the magnetic energies become more dominant and this can be seen in the third row with $\xi =0.5$. The magnetic energy, particularly the Ginzburg-Landau part given by $\xi l_2 \frac{|\nabla \Mvec|^2}{2} + \frac{\xi}{4}\left(|\Mvec|^2 - 1 \right)^2$, favours magnetic vortices of integer degrees \cite{bethuel1994ginzburg}. Hence, we see the magnetic bands of low order coalescing and converging to two magnetic vortices of degree $+1$ each, similar to the second row of \Cref{Fig3}. The nematic directors and magnetization vectors prefer to be co-aligned or follow each other, as dictated by the positive nemato-magnetic coupling parameter $c_1$. Consequently, we see four distinct fractional nematic defects localised near the square centre, in the third and fourth rows of \Cref{Fig4}. Therefore, the primary effect of increasing $\xi$ is to induce the creation of $k$ magnetic vortices of degree $+1$, near the square centre, as preferred by the magnetic energies and dictated by the topological degree $k$ of the boundary datum. This in turn induces the creation of $2k$ fractional interior nematic defects near the square centre. When $M_3 \neq 0$, the defects tend to be further apart since the stray field energy induces repulsion between defects as previously noted. We also comment on the color bars in \Cref{Fig4}. We note that the maximum values of $|\Qvec|$ and $|\Mvec|$ for the numerically computed solutions of \eqref{eq:26} decrease with increasing $\xi$ and this is consistent with the $L^\infty$ bounds derived in Lemma~\ref{lem1}.  

\subsubsection{Role of the coupling parameter \texorpdfstring{$c_1$}{c1}}\label{role of parameter c1}
We fix the parameters $l_1'=l_2=0.01$, $c_2=8$, $c_3=2$ and take $k=2$ for the Dirichlet boundary condition, for a fixed $\Hvec_{ext} = (0.4,0,0)$. We consider four different values of $c_1=1, 2, 3$ and $5$, as we move from the first row to the fourth row respectively in \Cref{Fig5}. \begin{figure}[ht!]
    \centering
    $\Qvec$$\left(M_3=0\right)$\hspace{1.79cm}$\Qvec$$\left(M_3\ne0\right)$\hspace{1.79cm}$\Mvec$$\left(M_3=0\right)$\hspace{1.79cm}$\Mvec$$\left(M_3\ne 0\right)$\\
    \includegraphics[width=3.65cm]{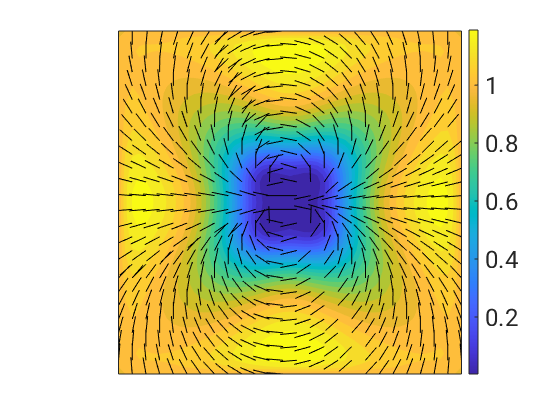}
    \includegraphics[width=3.65cm]{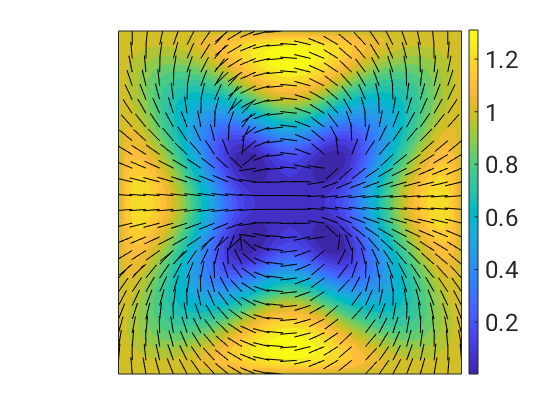}
    \includegraphics[width=3.65cm]{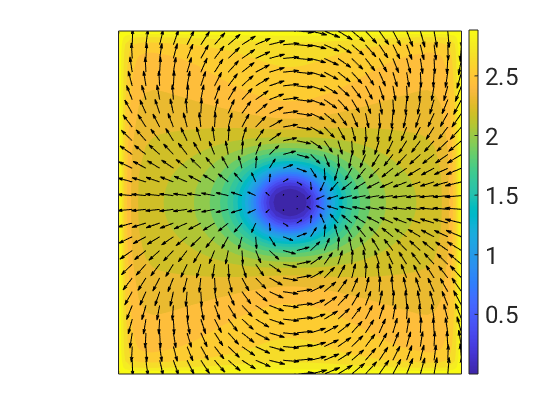}
    \includegraphics[width=3.65cm]{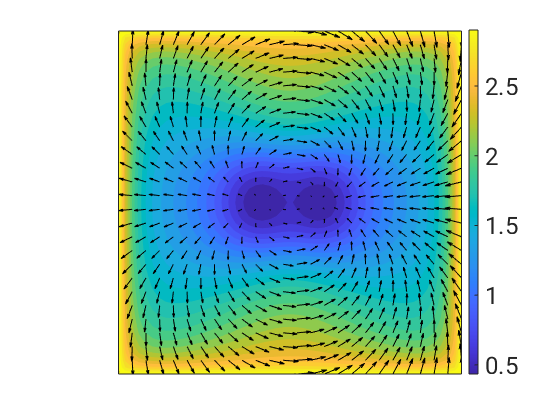}
    \\
    \vspace{2mm}
    \includegraphics[width=3.65cm]{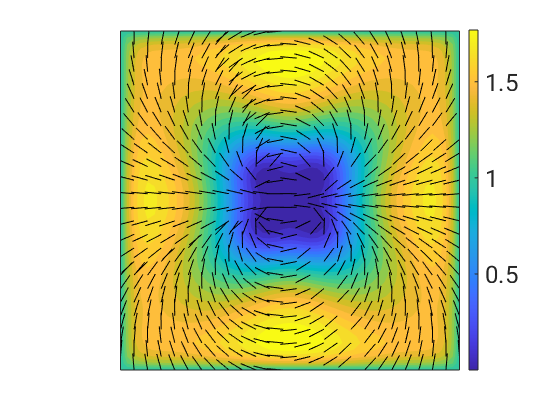}
    \includegraphics[width=3.65cm]{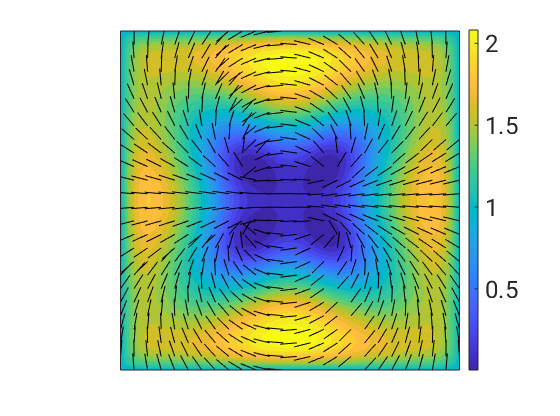}
    \includegraphics[width=3.65cm]{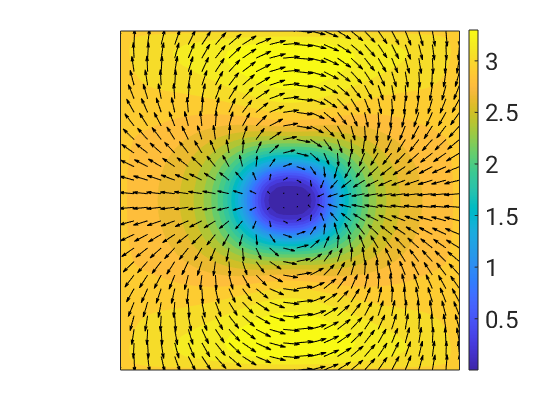}
    \includegraphics[width=3.65cm]{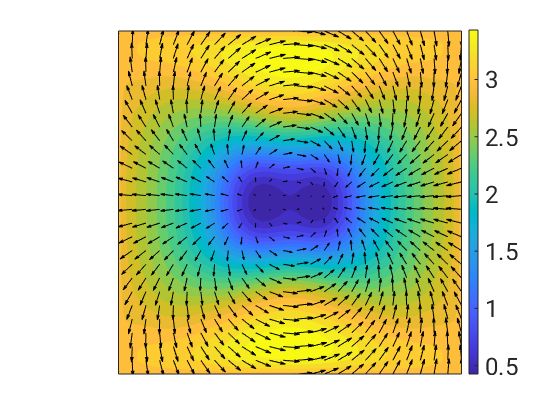}
    \\
    \vspace{2mm}
    \includegraphics[width=3.65cm]{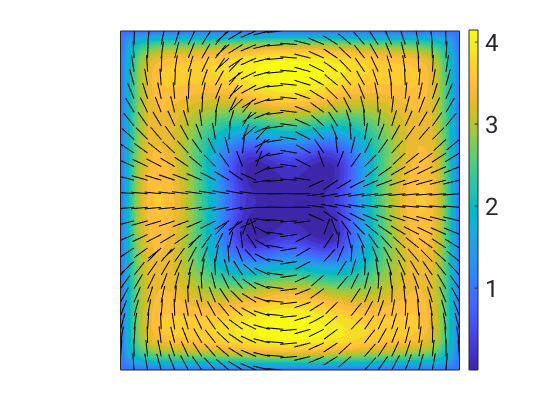}
    \includegraphics[width=3.65cm]{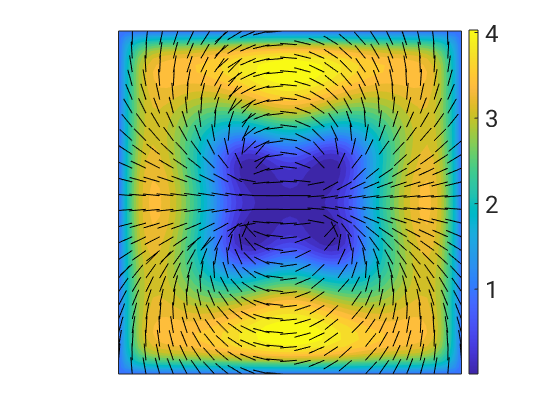}
    \includegraphics[width=3.49cm]{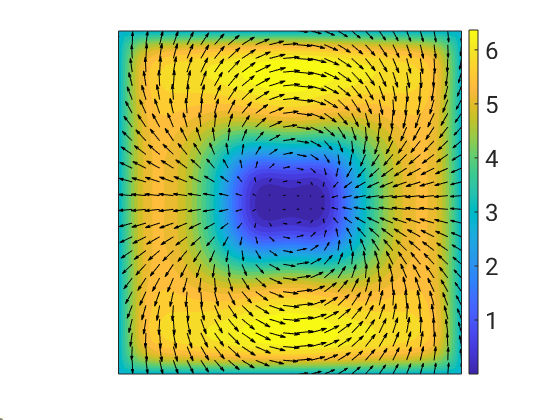}
    \includegraphics[width=3.65cm]{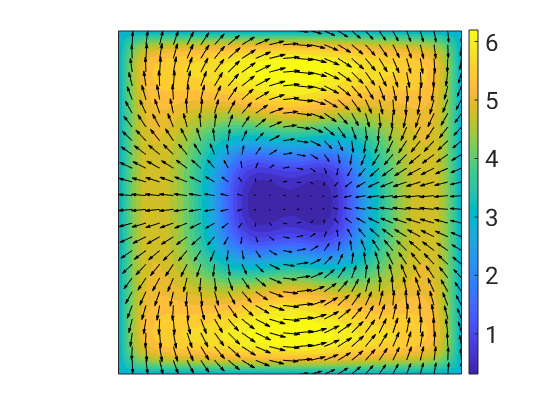}
    \\
    \vspace{2mm}
    \includegraphics[width=3.65cm]{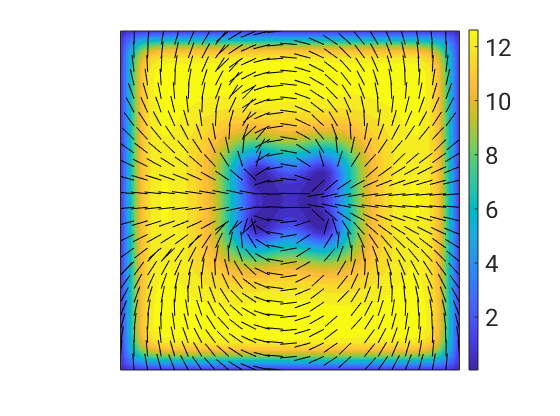}
    \includegraphics[width=3.65cm]{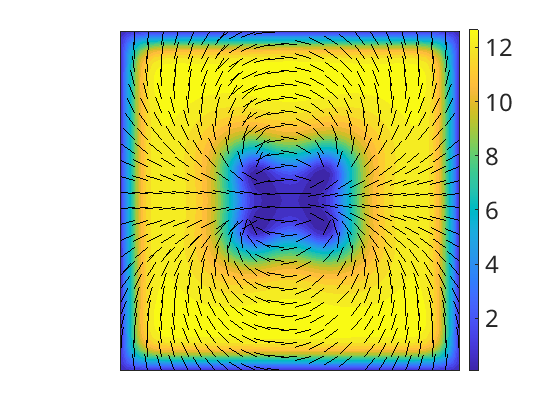}
    \includegraphics[width=3.65cm]{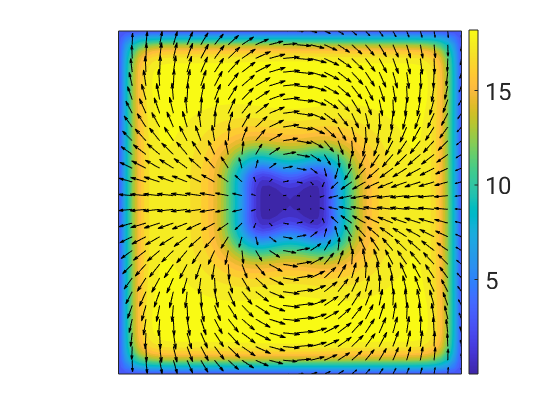}
    \includegraphics[width=3.65cm]{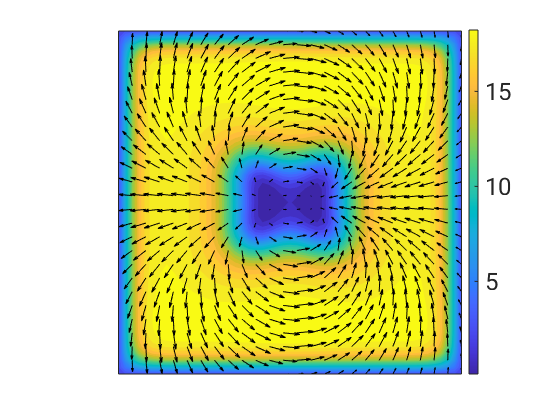}
     \caption{Nematic and magnetic configurations at $\Hvec_{ext}=\left(0.4, 0, 0\right)$ and different $c_1=1, 2, 3, 5$, varying vertically downwards, both with and without a stray field energy i.e., $M_3\neq0$ and $M_3=0$. 
     Parameter set: $l'_1=l_2=0.01$, $\xi=1$, $c_2=8$, $c_3=2$, $\eta_1=\eta_2=1$; $k=2$ and at time $t=0.01$.}
    \label{Fig5}
\end{figure}
As before, we consider the cases $M_3=0$ and $M_3 \neq 0$ (including stray field energy) separately. The external magnetic field is relatively weak and we do not see strong realignment of the nematic director and magnetization vector with $\Hvec_{ext}$ in all four cases. The effects of the stray field is to primarily split defects and push them away from each other, although we do not observe expulsion of interior defects since the external magnetic field is relatively weak. The primary effect of increasing $c_1$ is enhanced co-alignment between $\nvec$ and $\Mvec$ everywhere, including near defect cores. This is most evident by comparing the first and fourth rows of \Cref{Fig5}. For $M_3=0$ and $c_1=1$ (weak nemato-magnetic coupling), there are four almost pinned $+1/2$ nematic defects accompanied by a smaller magnetic defect core (which could be two $+1$-magnetic vortices almost superimposed on each other). The co-alignment between $\nvec$ and $\Mvec$ is much weaker near defects or regions of small $|\Qvec|$ and $|\Mvec|$. For $c_1=5$, there are $4$ clearly separated but centrally localised $+1/2$ nematic defects and $\Mvec$ exhibits a clear dipole pair of $+1$-magnetic vortices near the nematic defect locations. The nematic and magnetic defects may differ in their topology but the co-alignment between $\nvec$ and $\Mvec$ is much stronger even near defects, with $c_1=5$. The co-alignment effects are more pronounced when $M_3 \neq 0$, if we compare the second and fourth columns of the first and fourth rows respectively. In the first row, when $M_3 \neq 0$, there are four well-separated but localised nematic defects accompanied by two distinct $+1$-magnetic vortices in the $\Mvec$-profile. The nematic defects move closer together as $c_1$ increases (even with $M_3 \neq 0$) and the effects of the stray field are far less pronounced in the fourth row with $c_1=5$, when the co-alignment effects dominate the stray field energy and external magnetic field energy-induced effects. We note that the effect of increasing $c_1$ is comparable to the effect of decreasing $\xi$ or diminishing the importance of the magnetic energies. This is also reflected in the $L^\infty$ bounds for the solutions of \eqref{EL:1}--\eqref{EL:5}.

\subsubsection{Role of the parameter \texorpdfstring{$c_3$}{c3}}
We fix the parameters $l'_1=l_2=0.01$, $c_1=2$, $c_2=8$ and take $k=2$ in \Cref{Fig6}, with $\Hvec_{ext}=\left(0.4, 0, 0\right)$ and $c_3=2.5, 5, 10$ and $20$ in \Cref{Fig6} ($c_3$ is increasing vertically downwards). \begin{figure}[ht!]
    \centering
    $\Qvec$$\left(M_3=0\right)$\hspace{1.79cm}$\Qvec$$\left(M_3\ne0\right)$\hspace{1.79cm}$\Mvec$$\left(M_3=0\right)$\hspace{1.79cm}$\Mvec$$\left(M_3\ne 0\right)$\\
    \includegraphics[width=3.65cm]{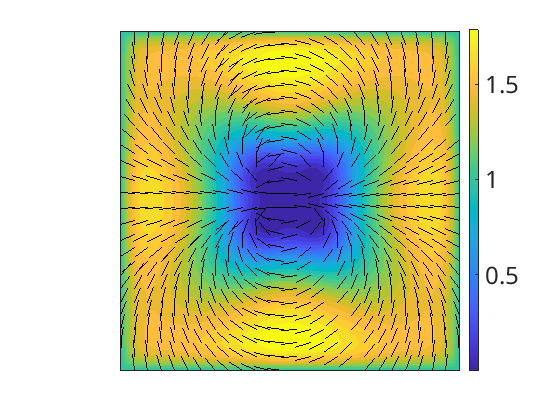}
    \includegraphics[width=3.65cm]{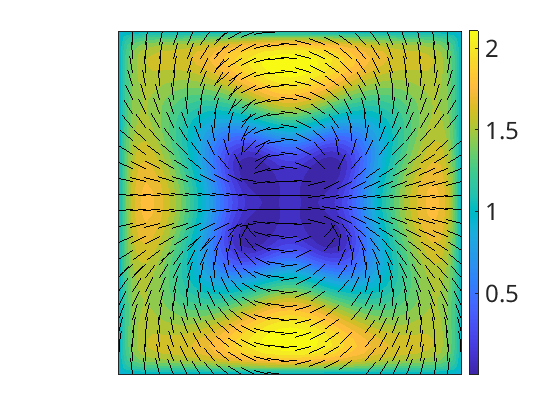}
    \includegraphics[width=3.65cm]{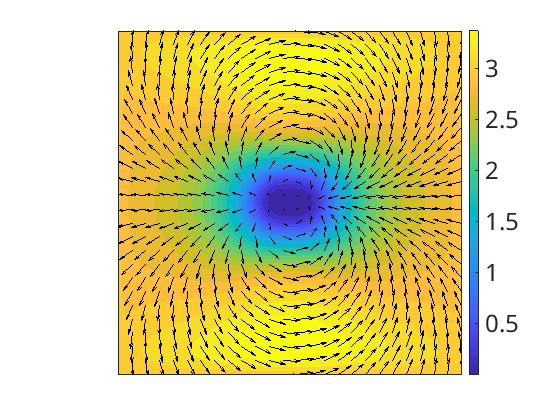}
    \includegraphics[width=3.65cm]{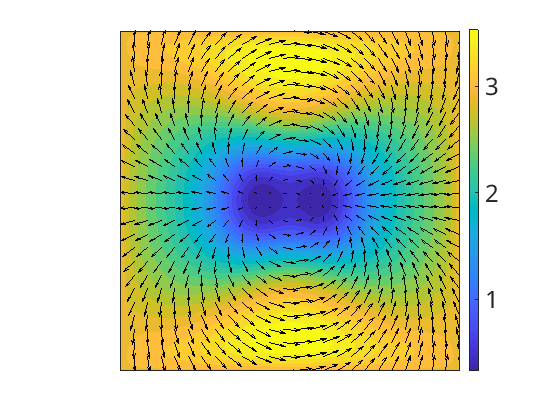}
    \\
    \vspace{2mm}
    \includegraphics[width=3.65cm]{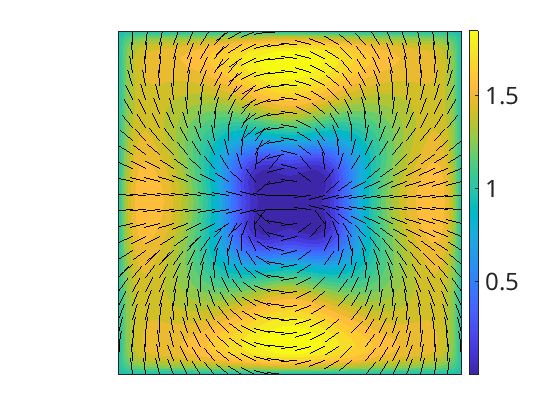}
    \includegraphics[width=3.65cm]{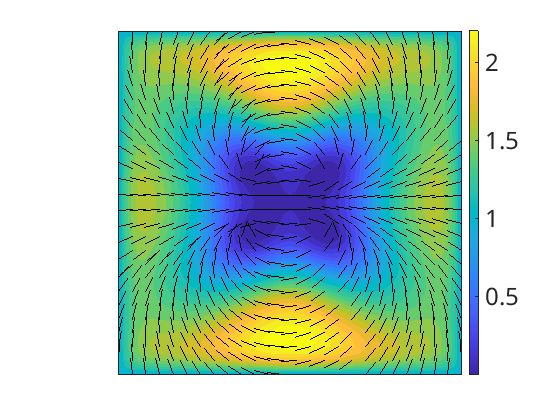}
    \includegraphics[width=3.65cm]{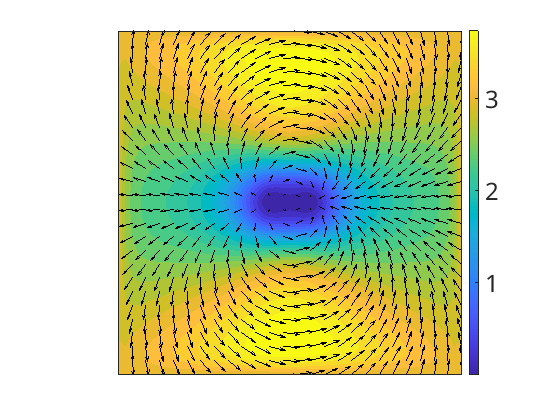}
    \includegraphics[width=3.65cm]{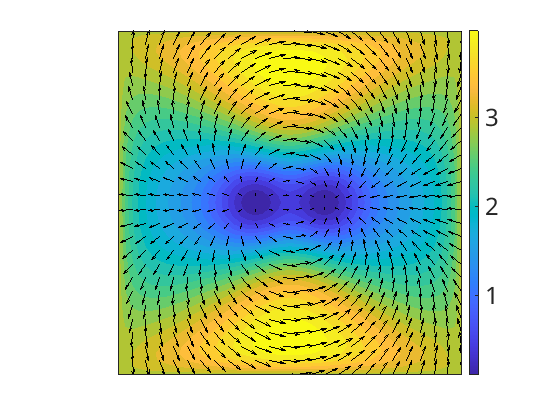}
    \\
    \vspace{2mm}
    \includegraphics[width=3.65cm]{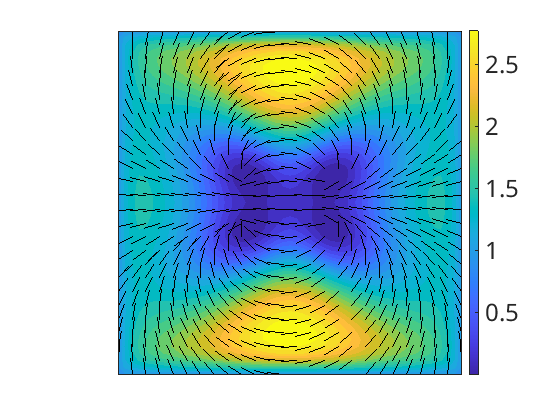}
    \includegraphics[width=3.65cm]{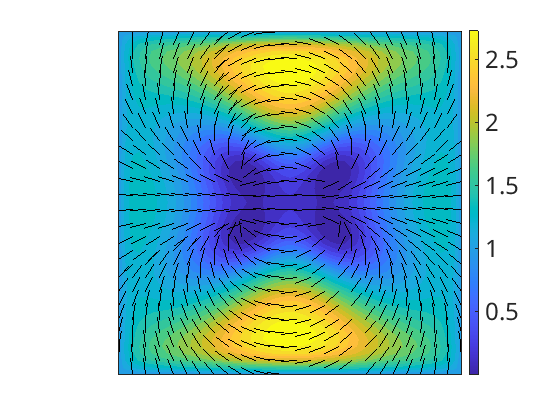}
    \includegraphics[width=3.65cm]{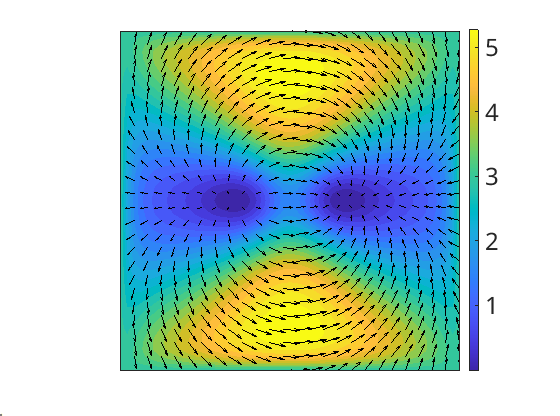}
    \includegraphics[width=3.65cm]{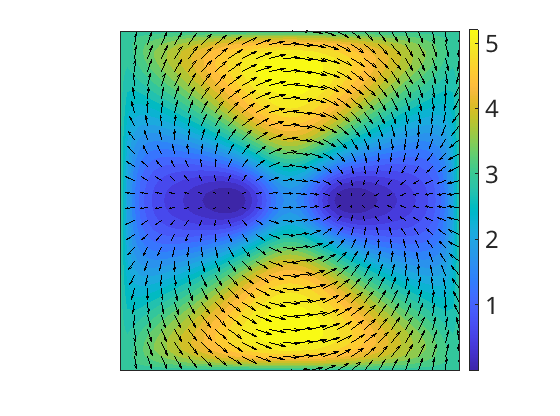}
    \\
    \vspace{2mm}
    \includegraphics[width=3.65cm]{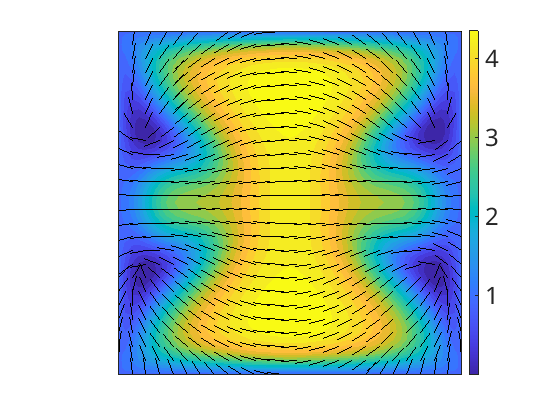}
    \includegraphics[width=3.65cm]{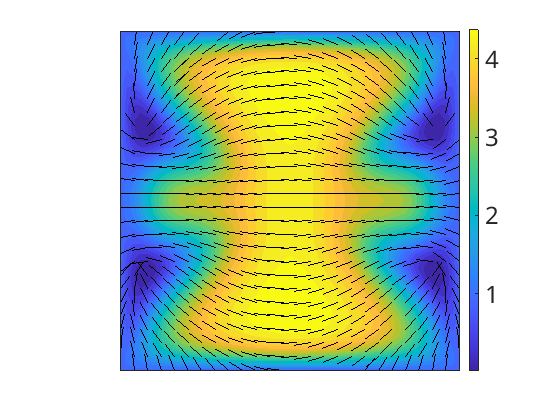}
    \includegraphics[width=3.65cm]{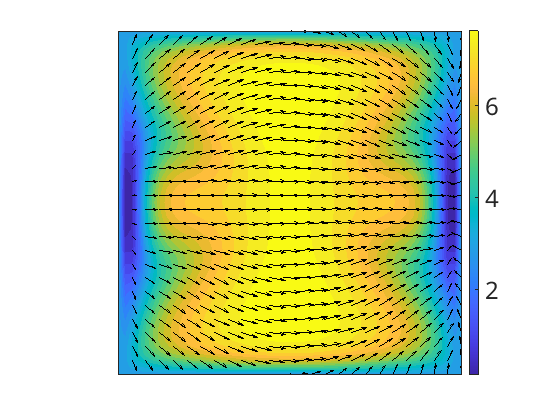}
    \includegraphics[width=3.65cm]{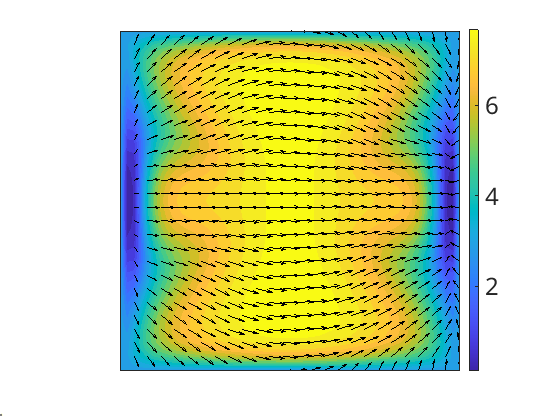}
     \caption{Nematic and magnetic configurations with $\Hvec_{ext}=\left(0.4, 0, 0\right)$ and different choices of $c_3=2.5, 5, 10, 20$, varying vertically downwards, both with and without a stray field energy, i.e., $M_3\neq0$ and $M_3=0$. 
     Parameter set: $l'_1=l_2=0.01$, $\xi=1$, $c_1=2$, $c_2=8$, $\eta_1=\eta_2=1$; $k=2$ and at time $t=0.01$.}
    \label{Fig6}
\end{figure}
As before, we consider the $M_3=0$ and $M_3 \neq 0$ cases separately. The effect of $c_3$ is clear - increasing $c_3$ enhances the effects of the stray field energy and the Zeeman energy. Consequently, we get almost complete co-alignment of $\nvec$ and $\Mvec$ with $\Hvec_{ext} = (0.4,0,0)$ and $c_3=20$. There are memories of the expelled nematic and interior defects near the left and right square edges. Thus, increasing $c_3$ promotes co-alignment of $\nvec$ and $\Mvec$ with the external magnetic field and consequently, also promotes the repulsion of defects and their migration to the boundaries. Therefore, if $c_3$ could be an experimentally tunable parameter, one could observe complete co-alignment of $\nvec$, $\Mvec$ and $\Hvec_{ext}$ with relatively low values of $|\Hvec_{ext}|$, which is of some practical importance. 

These numerical results clearly demonstrate how external magnetic fields and stray fields affect the multiplicity of defects, their locations, their repulsion and expulsion along with the nematic director and magnetization profiles. It is clear that the ferronematic solution landscape is much richer with magnetic fields, than without magnetic fields, and this naturally opens many doors for applications.

\section{Conclusions}\label{sec:6}
\label{sec:conclusions}
Our study is based on a ferronematic energy that comprises contributions from the Landau-de Gennes' theory of nematic liquid crystals and the theory of micromagnetics. In particular, we include a stray field energy contribution and energy contributions from external magnetic fields, and hence generalize the ferronematic model in \cite{Bisht2020}. Our numerical results concerning the influence of the stray field and the external magnetic field on ferronematics in two-dimensional geometries 
suggest that additional magnetic nanoparticles  can substantially affect observable nematic profiles: their director profiles, defect sets, stability etc.\ and hence, they can affect concomitant optical, mechanical and rheological properties too. 
In this context, we refer to the experimental works \cite{chen1983observation, Mertelj} that report that optical properties of manufactured ferronematics can be instantly controlled by using a weak strength external magnetic field (up to 1 mT); under an alternating magnetic field, such ferronematics exhibit an optical switching frequency above 100 Hz, which is comparable to commercial liquid crystal devices based on electrical switching. Thus, the inclusion of the stray field in our study will be more effective for such applications that demand weak-strength external magnetic fields. 

However, our observations are of qualitative nature.
For example, we have not used physically motivated values for the parameters $l_1'=2l_1, l_2, c_1, c_2, c_3, \xi$ in our numerical results.  
Such values are for instance reported in 
\cite{luo2012, newtonmottram, Zumar2009}. 
We do not comment on this further here because our purpose is to illustrate the rich solution landscapes of ferronematics in two dimensions and give examples of how these landscapes can be tuned by means of $\xi, c_1, c_2, c_3$ and the re-scaled elastic constants.

There are open questions about the choice of physically pertinent boundary conditions for $\Mvec$. Dirichlet conditions for $\Qvec$ are widely used in theoretical studies and are physically relevant e.g., \cite{luo2012, han2021tailored, tsakonas2007, wangprl}. However, it may not be experimentally feasible or practical to prescribe $\Mvec$ on the boundary, especially when $\Mvec$ is induced by suspended magnetic nanoparticles. In such cases, natural or Neumann boundary conditions for $\Mvec$ might be more appropriate, with $M_3=0$. For example, natural boundary conditions for $\Mvec$ on a square domain would reduce to
$
\frac{\partial \Mvec}{\partial x} = 0$ on $x=0,1
$
and
$
\frac{\partial \Mvec}{\partial y} =0$ on $y=0,1.
$
For Neumann boundary conditions, we conjecture that the effects of geometric frustration would be lost for the $\Mvec$ profiles and the $\Mvec$ profiles would be largely tailored by $\Qvec$ profiles and the external fields. For example, we might still get the interior magnetic defects induced by the interior nematic defects in Figures~\ref{Fig5} and \ref{Fig6}, but the topological degrees of the magnetic defects and their locations might be different for Neumann as opposed to Dirichlet boundary conditions for $\Mvec$. This is simply because the system is less constrained with Neumann boundary conditions for $\Mvec$. We do not comment further but acknowledge that there are several physically relevant generalisations of our work. 

Our mathematical and numerical study of a complex multi-physics system with multiple order parameters yields mathematical and some practical insight into the relevance of stray field energies and external magnetic fields for quasi two-dimensional ferronematics that can contribute to future experimental and theoretical studies in the same vein.

\section{Declaration} \sloppy S.D.: conceptualization, formal analysis, investigation, methodology, software, validation, writing-original draft, writing-review and editing, funding acquisition; J.D.: methodology, software, validation, writing-review and editing; A.M.: conceptualization, investigation, methodology, validation, supervision, writing-review and editing, funding acquisition; A.S.: conceptualization, investigation, methodology, validation, supervision, writing-review and editing.

\section*{Acknowledgments}
S.D. is grateful to DAAD (German Academic Exchange Service): The work was made possible with the support of a scholarship from
the German Academic Exchange Service (DAAD). A.M. gratefully acknowledges support from the Leverhulme Research Project Grant RPG-2021-401, a Humboldt Foundation Friedrich Wilhelm Bessel Research Award and a University of Strathclyde GEF Horizon Europe Faculty award STR1005-377.

\bibliographystyle{unsrtnat}  
\bibliography{references}      

\begin{thebibliography}{40}
\providecommand{\natexlab}[1]{#1}
\providecommand{\url}[1]{\texttt{#1}}
\expandafter\ifx\csname urlstyle\endcsname\relax
  \providecommand{\doi}[1]{doi: #1}\else
  \providecommand{\doi}{doi: \begingroup \urlstyle{rm}\Url}\fi

\bibitem[de~Gennes and Prost(1995)]{Gennes}
Pierre~Gilles de~Gennes and Jacques Prost.
\newblock The physics of liquid crystals.
\newblock \emph{International series of monographs on physics}, 2, 1995.

\bibitem[Brown(1966)]{brown1966magnetoelastic}
William~Fuller Brown.
\newblock \emph{Magnetoelastic interactions}, volume~9.
\newblock Springer, 1966.

\bibitem[Lagerwall and Scalia(2012)]{lagerwall2012new}
Jan~PF Lagerwall and Giusy Scalia.
\newblock A new era for liquid crystal research: Applications of liquid
  crystals in soft matter nano-, bio-and microtechnology.
\newblock \emph{Current Applied Physics}, 12\penalty0 (6):\penalty0 1387--1412,
  2012.

\bibitem[Shuai et~al.(2016)Shuai, Klittnick, Shen, Smith, Tuchband, Zhu,
  Petschek, Mertelj, Lisjak, {\v{C}}opi{\v{c}}, Maclennan, Glaser, and
  Clark]{Shui}
Min Shuai, Arthur Klittnick, Yongqiang Shen, Gregory~P Smith, Michael~R
  Tuchband, Chenhui Zhu, Rolfe~G Petschek, A.~Mertelj, Darja Lisjak, Martin
  {\v{C}}opi{\v{c}}, J.~E Maclennan, M.~A. Glaser, and N.~A. Clark.
\newblock Spontaneous liquid crystal and ferromagnetic ordering of colloidal
  magnetic nanoplates.
\newblock \emph{Nature Communications}, 7\penalty0 (1):\penalty0 10394, 2016.

\bibitem[Brochard and De~Gennes(1970)]{brochard1970theory}
F~Brochard and PG~De~Gennes.
\newblock Theory of magnetic suspensions in liquid crystals.
\newblock \emph{Journal de Physique}, 31\penalty0 (7):\penalty0 691--708, 1970.

\bibitem[Chen and Amer(1983)]{chen1983observation}
Shu-Hsia Chen and Nabil~M Amer.
\newblock Observation of macroscopic collective behavior and new texture in
  magnetically doped liquid crystals.
\newblock \emph{Physical Review Letters}, 51\penalty0 (25):\penalty0 2298,
  1983.

\bibitem[Mertelj et~al.(2013)Mertelj, Lisjak, Drofenik, and
  {\v{C}}opi{\v{c}}]{Mertelj}
Alenka Mertelj, Darja Lisjak, Miha Drofenik, and Martin {\v{C}}opi{\v{c}}.
\newblock Ferromagnetism in suspensions of magnetic platelets in liquid
  crystal.
\newblock \emph{Nature}, 504\penalty0 (7479):\penalty0 237--241, 2013.

\bibitem[Bisht et~al.(2020)Bisht, Wang, Banerjee, and Majumdar]{Bisht2020}
Konark Bisht, Yiwei Wang, Varsha Banerjee, and Apala Majumdar.
\newblock Tailored morphologies in two-dimensional ferronematic wells.
\newblock \emph{Physical Review E}, 101\penalty0 (2):\penalty0 022706, 2020.

\bibitem[Canevari et~al.(2025)Canevari, Dipasquale, and
  Stroffolini]{Canevari2025}
Giacomo Canevari, Federico~Luigi Dipasquale, and Bianca Stroffolini.
\newblock The formation of gradient-driven singular structures of codimension
  one and two in two-dimensions: The case study of ferronematics.
\newblock \emph{arXiv:2505.07506}, 2025.

\bibitem[Canevari et~al.(2023)Canevari, Majumdar, Stroffolini, and
  Wang]{Canevari}
Giacomo Canevari, Apala Majumdar, Bianca Stroffolini, and Yiwei Wang.
\newblock Two-dimensional ferronematics, canonical harmonic maps and minimal
  connections.
\newblock \emph{Archive for Rational Mechanics and Analysis}, 247\penalty0
  (6):\penalty0 110, 2023.

\bibitem[Dalby et~al.(2022)Dalby, Farrell, Majumdar, and Xia]{Dalby}
James Dalby, Patrick~E Farrell, Apala Majumdar, and Jingmin Xia.
\newblock One-dimensional ferronematics in a channel: Order reconstruction,
  bifurcations, and multistability.
\newblock \emph{SIAM Journal on Applied Mathematics}, 82\penalty0 (2):\penalty0
  694--719, 2022.

\bibitem[Maity et~al.(2021)Maity, Majumdar, and Nataraj]{maity2021parameter}
Ruma~Rani Maity, Apala Majumdar, and Neela Nataraj.
\newblock Parameter dependent finite element analysis for ferronematics
  solutions.
\newblock \emph{Computers \& Mathematics with Applications}, 103:\penalty0
  127--155, 2021.

\bibitem[Han et~al.(2020)Han, Majumdar, and Zhang]{han2020reduced}
Yucen Han, Apala Majumdar, and Lei Zhang.
\newblock A reduced study for nematic equilibria on two-dimensional polygons.
\newblock \emph{SIAM Journal on Applied Mathematics}, 80\penalty0 (4):\penalty0
  1678--1703, 2020.

\bibitem[Hubert and Sch{\"a}fer(2008)]{hubert2008magnetic}
Alex Hubert and Rudolf Sch{\"a}fer.
\newblock \emph{Magnetic domains: the analysis of magnetic microstructures}.
\newblock Springer Science \& Business Media, 2008.

\bibitem[Kru{\v z}{\' i}k and Prohl(2006)]{kruzikprohl}
Martin Kru{\v z}{\' i}k and Andreas Prohl.
\newblock Recent developments in the modeling, analysis, and numerics of
  ferromagnetism.
\newblock \emph{SIAM Review}, 48\penalty0 (3):\penalty0 439–--483, 2006.

\bibitem[DeSimone et~al.(2004)DeSimone, Kohn, M{\"u}ller, and
  Otto]{desimone2004recent}
Antonio DeSimone, Robert~V Kohn, Stefan M{\"u}ller, and Felix Otto.
\newblock Recent analytical developments in micromagnetics.
\newblock \emph{Preprint by MPI, Leipzig}, 2004.

\bibitem[Gioia and James(1997)]{Giogia}
Gustavo Gioia and Richard~D James.
\newblock Micromagnetics of very thin films.
\newblock \emph{Proceedings of the Royal Society of London. Series A:
  Mathematical, Physical and Engineering Sciences}, 453\penalty0
  (1956):\penalty0 213--223, 1997.

\bibitem[Garc{\'i}a-Cevera(2004)]{Garcia2004}
Carlos Garc{\'i}a-Cevera.
\newblock One-dimensional magnetic domain walls.
\newblock \emph{Euro. Jnl of Applied Mathematics}, 15:\penalty0 451--486, 2004.

\bibitem[Pleiner et~al.(2001)Pleiner, Jarkova, M{\"u}ller, and Brand]{Pleiner}
Harald Pleiner, E~Jarkova, HW~M{\"u}ller, and HR~Brand.
\newblock Landau description of ferrofluid to ferronematic phase transition.
\newblock \emph{Magnetohydrodynamics}, 37\penalty0 (254):\penalty0 146, 2001.

\bibitem[Calderer et~al.(2014)Calderer, DeSimone, Golovaty, and
  Panchenko]{Calderer}
Maria-Carme Calderer, Antonio DeSimone, Dmitry Golovaty, and Alexander
  Panchenko.
\newblock An effective model for nematic liquid crystal composites with
  ferromagnetic inclusions.
\newblock \emph{SIAM Journal on Applied Mathematics}, 74\penalty0 (2):\penalty0
  237--262, 2014.

\bibitem[Burden and Faires(1985)]{Burden-Faires}
RL~Burden and DJ~Faires.
\newblock \emph{Numerical analysis. 3rd edn Boston}.
\newblock MA: PWS Publishing Company, 1985.

\bibitem[Han et~al.(2021)Han, Harris, Walton, and Majumdar]{han2021tailored}
Yucen Han, Joseph Harris, Joshua Walton, and Apala Majumdar.
\newblock Tailored nematic and magnetization profiles on two-dimensional
  polygons.
\newblock \emph{Physical Review E}, 103\penalty0 (5):\penalty0 052702, 2021.

\bibitem[Kalousek et~al.(2021)Kalousek, Kortum, and
  Schl{\"o}merkemper]{kalousek2021mathematical}
Martin Kalousek, Joshua Kortum, and Anja Schl{\"o}merkemper.
\newblock Mathematical analysis of weak and strong solutions to an evolutionary
  model for magnetoviscoelasticity.
\newblock \emph{Discrete \& Continuous Dynamical Systems-Series S}, 14\penalty0
  (1), 2021.

\bibitem[DeSimone et~al.(2002)DeSimone, Kohn, M{\"u}ller, and
  Otto]{desimone2002reduced}
Antonio DeSimone, Robert~V Kohn, Stefan M{\"u}ller, and Felix Otto.
\newblock A reduced theory for thin-film micromagnetics, 2002.

\bibitem[Di~Fratta et~al.(2023)Di~Fratta, Muratov, and
  Slastikov]{di2023reduced}
Giovanni Di~Fratta, Cyrill~B Muratov, and Valeriy~V Slastikov.
\newblock Reduced energies for thin ferromagnetic films with perpendicular
  anisotropy.
\newblock \emph{arXiv preprint arXiv:2306.07634}, 2023.

\bibitem[Kohn and Slastikov(2005)]{kohn2005another}
Robert~V Kohn and Valeriy~V Slastikov.
\newblock Another thin-film limit of micromagnetics.
\newblock \emph{Archive for Rational Mechanics and Analysis}, 178:\penalty0
  227--245, 2005.

\bibitem[Golovaty et~al.(2015)Golovaty, Montero, and
  Sternberg]{golovaty2015dimension}
Dmitry Golovaty, Jos{\'e}~Alberto Montero, and Peter Sternberg.
\newblock Dimension reduction for the {L}andau-de {G}ennes model in planar
  nematic thin films.
\newblock \emph{Journal of Nonlinear Science}, 25\penalty0 (6):\penalty0
  1431--1451, 2015.

\bibitem[Mottram and Newton(2014)]{newtonmottram}
Nigel~J. Mottram and Chris J.~P. Newton.
\newblock Introduction to q-tensor theory.
\newblock \emph{arXiv:1409.3542}, 2014.

\bibitem[Ciarlet(2021)]{ciarlet2021mathematical}
Philippe~G Ciarlet.
\newblock \emph{Mathematical elasticity: Three-dimensional elasticity}.
\newblock SIAM, 2021.

\bibitem[Rindler(2018)]{Rindler}
Filip Rindler.
\newblock \emph{Calculus of variations}, volume~5.
\newblock Springer, 2018.

\bibitem[Evans(2022)]{evans2022partial}
Lawrence~C Evans.
\newblock \emph{Partial differential equations}, volume~19.
\newblock American Mathematical Society, 2022.

\bibitem[Onsager(1931)]{onsager1931reciprocal}
Lars Onsager.
\newblock Reciprocal relations in irreversible processes. {II.}
\newblock \emph{Physical Review}, 38\penalty0 (12):\penalty0 2265, 1931.

\bibitem[Smith(1985)]{smith1985numerical}
Gordon~D Smith.
\newblock \emph{Numerical solution of partial differential equations: finite
  difference methods}.
\newblock Oxford university press, 1985.

\bibitem[Luo et~al.(2012)Luo, Majumdar, and Erban]{luo2012}
C.~Luo, A.~Majumdar, and R.~Erban.
\newblock Multistability in planar liquid crystal wells.
\newblock \emph{Phys. Rev. E}, 85\penalty0 (6):\penalty0 061702, 2012.

\bibitem[Tsakonas et~al.(2007)Tsakonas, Davidson, Brown, and
  Mottram]{tsakonas2007}
C.~Tsakonas, A.~J. Davidson, C.~V. Brown, and N.~J. Mottram.
\newblock Multistable alignment states in nematic liquid crystal filled wells.
\newblock \emph{Applied Physics Letters}, 90\penalty0 (11):\penalty0 111913,
  2007.

\bibitem[Ball(2017)]{ball_notes}
J.~M. Ball.
\newblock Liquid crystals and their defects.
\newblock \emph{arXiv:1706.06861v3 [cond-mat.soft]}, 2017.

\bibitem[Wachowiak et~al.(2002)Wachowiak, Wiebe, Bode, Pietzsch, Morgenstern,
  and Wiesendanger]{wachowiak2002direct}
A~Wachowiak, J~Wiebe, M~Bode, O~Pietzsch, M~Morgenstern, and R~Wiesendanger.
\newblock Direct observation of internal spin structure of magnetic vortex
  cores.
\newblock \emph{Science}, 298\penalty0 (5593):\penalty0 577--580, 2002.

\bibitem[Bethuel et~al.(1994)Bethuel, Brezis, and
  H{\'e}lein]{bethuel1994ginzburg}
Fabrice Bethuel, Ha{\"\i}m Brezis, and Fr{\'e}d{\'e}ric H{\'e}lein.
\newblock \emph{Ginzburg-landau vortices}, volume~13.
\newblock Springer, 1994.

\bibitem[Ravnik and {\v{Z}}umer(2009)]{Zumar2009}
M.~Ravnik and S~{\v{Z}}umer.
\newblock Landau–de {G}ennes modelling of nematic liquid crystal colloids.
\newblock \emph{Liquid Crystals}, 36\penalty0 (10-11):\penalty0 1201--1214,
  2009.

\bibitem[Yin et~al.(2020)Yin, Wang, Chen, Zhang, and Zhang]{wangprl}
J.~Y. Yin, Y.~W. Wang, J.~Z.~Y. Chen, P.~W. Zhang, and L.~Zhang.
\newblock {Construction of a pathway map on a complicated energy landscape}.
\newblock \emph{Physical Review Letters}, 124:\penalty0 090601, 2020.

\end{thebibliography}

\end{document}